\documentclass[12pt,reqno]{elsarticle}

\usepackage{amsfonts,color,amsmath,amssymb,fancyhdr}
\usepackage{amsthm,multirow}
\usepackage[colorlinks,linkcolor=blue, anchorcolor=red,citecolor=green]{hyperref}
\usepackage{lineno}

\newcommand{\ra}{\rangle}

\newcommand{\la}{\langle}
\newcommand{\F}{{\mathbb F}}

\newcommand{\cB}{{\mathcal B}}

\newcommand{\cK}{{\mathcal K}}
\newcommand{\cQ}{{\mathcal Q}}

\newcommand{\cM}{{\mathcal M}}

\newcommand{\fg}{{\mathfrak g}}
\newcommand{\PG}{\textup{PG}}

\newcommand{\Aut}{\textup{Aut}}

\newcommand{\tr}{\textup{Tr}}
\newcommand{\PGaSp}{\textup{P}\Gamma\textup{Sp}}
\newcommand{\im}{\textup{Im}}

\newtheorem{thm}{Theorem}

\newtheorem{lemma}[thm]{Lemma}

\newtheorem{corollary}[thm]{Corollary}

\newtheorem{example}[thm]{Example}
\numberwithin{equation}{section}
\numberwithin{thm}{section}

\newtheorem{remark}[thm]{Remark}

\newtheorem{construction}[thm]{Construction}
\newtheorem{Notation}{Notation}[section]

\begin{document}
%\pagewiselinenumbers

\title{ The point regular automorphism groups of the Payne derived quadrangle of $W(q)$}
\author[add1]{Tao Feng} \ead{tfeng@zju.edu.cn}
\author[add1,add2]{Weicong Li}\ead{conglw@zju.edu.cn}
\address[add1]{School of Mathematical Sciences, Zhejiang University, Hangzhou 310027, Zhejiang, P.R. China}
\address[add2]{Department of Mathematics, Southern University of Science and Technology,	Shenzhen 518055, Guangdong, P.R. China}
%\classno{05B25 (primary),  05E18, 51E12 (secondary)}
%\extraline{Supported by National Natural Science Foundation of China under Grant No. 11771392.}

\begin{abstract}
	
In this paper, we completely determine the point regular automorphism groups of the Payne derived quadrangle of the symplectic quadrangle $W(q)$, $q$ odd. As a corollary, we show that the finite groups that act regularly on the points of a  finite generalized quadrangle can have unbounded nilpotency class.

\end{abstract}
\maketitle

\tableofcontents

\section{Introduction}

A finite \textit{generalized quadrangle} $\cQ$ of order $(s,t)$ is a point-line incidence structure such that each point is incident with $t+1$ lines, each line is incident with $s+1$ points, and for each non-incident point-line pair $(P,\ell)$ there is a unique point $Q$ incident with $\ell$ such that $P,\,Q$ are collinear. The quadrangle $\cQ$ is called \textit{thick} if $s>1$ and $t>1$. If we interchange the role of points and lines of a generalized quadrangle $\cQ$ of order $(s,t)$, we get the \textit{dual} quadrangle of order $(t,s)$. The \textit{classical} generalized quadrangles are those that arise as point-line incidence structures of finite classical polar spaces of rank $2$.  The standard textbook on finite generalized quadrangles is the monograph \cite{FGQ}.

The study of generalized quadrangles has close connections with group theory and other branches of mathematics. J. Tits \cite{JTits1959Sur} introduced the notion of generalized polygons in order to better understand the Chevalley groups of rank $2$. A generalized $3$-gon is a projective plane, and a generalized $4$-gon is a generalized quadrangle. In analogy with the study of finite projective planes, there has been extensive work on the classification of generalized quadrangles that exhibit high symmetry using a blend of geometric arguments and deep results in group theory. The reader is referred to the monographs \cite{Thas2006bookTGQ,KThas2004bookSFGQ} for an account of  the history and recent developments in this field.

Besides the classical examples and their duals, there are three further types of known finite generalized quadrangles up to duality: translation generalized quadrangles, flock generalized quadrangles, and the generalized quadrangles of order $(q-1,q+1)$. Most of the known examples except those with parameter $(q-1,q+1)$ can be described by their respective \textit{Kantor family} or \textit{$4$-gonal family}, a concept first introduced by Kantor in \cite{Kantor1980GQG2q}.  The first examples of generalized quadrangles of order $(q-1,q+1)$, $q$ a prime power, are due to  Ahrens and  Szekeres \cite{ASGQ1}, and independently to M. Hall, Jr. \cite{ASGQ2} for $q$ even. Payne has a general construction method, now known as the \textit{Payne derivation}, that yields a generalized quadrangle $\cQ^P$ of order $(q-1,q+1)$ from a generalized quadrangle $\cQ$ of order $(q,q)$ with a regular point $P$, cf.  \cite{Payne1971EquiCentainGQ,Payne1971NGQ,Payne1972}. The resulting quadrangle $\cQ^P$ is called a \textit{Payne derived quadrangle}. All the known generalized quadrangles of order $(q-1,q+1)$ arise as Payne derived quadrangles. The automorphism groups of Payne derived quadrangles were studied in \cite{SSymGQ1994} and \cite{De2007AutoGPaG}.

Ghinelli \cite{Ghinelli1992RGonGQ} was the first to study the  generalized quadrangles admitting a point regular group, where she used representation theory and difference sets to study the case where the generalized quadrangle has order $(s,s)$ with $s$ even. Further progress was made in \cite{Yoshiara2007GQwAGRP} and \cite{Swartz2019OGQPRG}. For instance, it is shown in \cite{Yoshiara2007GQwAGRP} that finite thick generalized quadrangles of order $(t^2,t)$ does not admit a point regular automorphism group. By combining the results in \cite{Ghinelli2012ChAS} and \cite{Yoshiara2007GQwAGRP}, it was shown in \cite{Bamberg2015AScon} that any skew-translation generalized quadrangle of order $(q,q)$, $q$ odd, is isomorphic to the classical symplectic quadrangle $W(q)$.  Swartz \cite{Swartz2019OGQPRG}  initiated the study of generalized quadrangles admitting an automorphism group that acts regularly on both points and lines. Up till 2011, all the known finite generalized quadrangles admitting a point regular group arise by  Payne derivation from a thick elation quadrangle $\cQ$ of order $(s,s)$ with a regular point, and their point regular groups are induced from the elation groups of $\cQ$. Motivated by this observation,  De Winter, K. Thas and  Shult attempted to show that there are no other examples in a series of papers \cite{De2006GQASG,De2007AutoGPaG,De2008GQaSTHG,SDeWinter2009SingerQuad}. In the preprint \cite{SDeWinter2009SingerQuad}, the known generalized quadrangles admitting a point regular automorphism group were classified in a combinatorial fashion.

A breakthrough in this direction is the classification of thick classical generalized quadrangles admitting a point regular group that leads to the discovery of three sporadic examples  and the construction of new point regular groups of the derived quadrangles of the symplectic quadrangle $W(q)$ in \cite{Bamberg2011PRG}. This corrected a small error in  \cite{SDeWinter2009SingerQuad} and disproved some conjectures in \cite{De2008GQaSTHG}. In the same paper, the authors calculated and listed the point regular automorphism groups of the derived quadrangle of $W(q)$ for small $q$ by Magma \cite{Magma}, and the results suggest that ``the problem is wild". In the case $q$ is odd, Chen \cite{Chen2013}, K. Thas and De Winter \cite{DeWinter2014Criterion} independently classified the linear case, i.e., the group is induced by a linear group of the ambient projective space of $W(q)$. There are also some constructions in the even characteristic case in \cite{DeWinter2014Criterion}. All the known finite groups that act regularly on the point set of a finite generalized quadrangle so far have nilpotency class at most $3$ except for some groups of small order. The point regular groups of the derived quadrangles have  applications in the constructions of uniform lattices in $\widetilde{C_2}$-buildings, cf. \cite{DeWinter2014Criterion,lattice}.

In this paper, we systematically study the point regular automorphism groups of the Payne derived quadrangle $\cQ^P$ of the symplectic quadrangle $\cQ=W(q)$ with respect to a regular point $P$. Every point of $W(q)$ is regular, and different choices of $P$'s yield isomorphic derived quadangles. By \cite[Corollary 2.4]{SSymGQ1994}, the full automorphism group of  $\cQ^P$ is the stabilizer of $P$ in $\PGaSp(4,q)$ when $q\geq 5$. We call a point regular group $G$ of $\cQ^P$ \textit{linear} if
$G$ is a subgroup of $\textup{PGL}(4,q)$, and call it \textit{nonlinear} otherwise. We completely determine all the  point regular groups of $\cQ^P$ for odd  $q$
and all the linear point regular groups for even $q$. This leads to four (\textit{resp.} two) constructions in the odd (\textit{resp.} even) characteristic case. We derive relatively tight upper and lower bounds on the nilpotency classes of the resulting groups in the case $q$ is odd.  Our result  contributes to the fundamental problem as for which finite groups can act regularly on the points of a  finite generalized quadrangle. In particular, we see that such a group can have arbitrarily large nilpotency class.

The	 paper is organized as follows. In Section 2.1, we introduce the model for a point regular group of the derived quadrangle $\cQ^P$. In Section 2.2, we give a brief summary of the main results in this paper, including the classification theorem for $q$ odd and bounds on the nilpotency classes of the point regular groups that we construct. We also give a brief description of our strategy in this subsection. In Section 3, we prove some technical lemmas about the arithmetic of finite fields. In Section \ref{sec_linear},  we completely determine the linear point regular groups of $\cQ^P$ for $q\ge 5$. In Section \ref{sec_structure}, we analyze the structure of a putative nonlinear point regular group $G$ of $\cQ^P$, and obtain information on the Frobenius part and the matrix part of $G$. In Section \ref{sec_sum_struc_qodd}, we summarize the structural results and state the classification theorem, Theorem \ref{thm_nonlinear}, for nonlinear point regular groups in the case $q$ is odd. Section \ref{sec_nonlinear_req0} and Section \ref{sec_nonlinear_req1} are devoted to the proof of Theorem \ref{thm_nonlinear}. In Section \ref{sec_iso}, we first simplify the constructions obtained in Sections \ref{sec_nonlinear_req0} and \ref{sec_nonlinear_req1} up to conjugacy and then calculate the various group invariants of the resulting point regular groups. In the last section, we conclude the paper with some problems for further research.

\section{Preliminaries}

The group theoretical terminology that we use is standard, cf. \cite{FGroup,ACouOnGp,Suzuki1986}. Let $G$ be a finite group. The \textit{exponent} of $G$, denoted $\exp(G)$,  is the smallest positive integer $n$ such that $g^n=1$ for all $g\in G$. For $g,\,h\in G$, their \textit{commutator} is $[g,h]=g^{-1}h^{-1}gh$. For $g \in G$, the \textit{centralizer} $C_G(g)$ of $g$ is the set of elements $h\in G$ such that $[g,h]=1$.   The \textit{center}  of $G$ is $Z(G)=\{g\in G:\,[g,h]=1,\,\forall\, h\in G\}$. For two subgroups $H_1,\,H_2$ of $G$, we use $[H_1,H_2]$ for the subgroup $\la [h_1,h_2]:\,h_1\in H_1,\,h_2\in H_2\ra$. In particular, the \textit{derived subgroup} $G'$ is the subgroup $[G,G]$. We use the symbol $\gamma_{i}(G)$ for the $i$-th term of the lower central series of $G$. Inductively, we have $\gamma_1(G)=G$ and $\gamma_{i+1}(G)=[\gamma_i(G),G]$ for $i\ge 1$. The group $G$ is   \textit{nilpotent} if $\gamma_{c+1}(G)=1$ for some integer $c$ and the smallest such integer is the \textit{nilpotency class} of $G$.  Similarly, we use the symbol $Z_{i}(G)$ for the $i$-th term of the upper central series of $G$, where $Z_0(G)=1$,  and $Z_{i+1}(G)$ is defined by the property $Z_{i+1}(G)/Z_{i}(G)=Z(G/Z_{i}(G))$ for $i\ge 0$. The upper central series and the lower central series of a finite nilpotent group have the same length. Let $d$ be the maximum size of an abelian subgroup of a finite $p$-group $G$. The \textit{Thompson subgroup} of $G$ is generated by all abelian subgroups of order $d$, and is denoted by $J(G)$.

\subsection{The Payne derived quadrangle $\cQ^P$ of $\cQ=W(q)$}\label{subsec_payneGQ}
We adopt the standard notions on generalized quadrangles as can be found in the monograph \cite{FGQ}.  Let $q=p^m$ be a prime power with $p$ prime. Let $\perp$ be a fixed symplectic polarity of $\PG(3,q)$. The classical generalized quadrangle $\cQ=W(q)$ has the same point set as $\PG(3,q)$ and has the totally isotropic lines of $\PG(3,q)$ as its lines.  The Payne derived quadrangle $\cQ^P$ of $\cQ$ with respect to a  point $P$ has points of $\PG(3,q)\setminus P^\perp$ as its points and has two types of lines: the totally isotropic lines not containing $P$, and the lines $\la P,\,Q\ra$ with $Q\not\in P^\perp$. The order of $\cQ^P$ is $(q-1,q+1)$. The automorphism group of $\cQ$ acts transitively on its point set, so the different choices of $P$ lead to isomorphic derived quadrangles. By \cite[Corollary 2.4]{SSymGQ1994}, the stabilizer of $P$ in $\PGaSp(4,q)$ is the full automorphism group of $\cQ^P$ when $q\geq 5$.

Let $V:=\F_q^4$ be a vector space over $\F_q$ equipped with the alternating form
$(x,y):=x_1y_4-x_4y_1+x_2y_3-x_3y_2$. Set $\delta:=\text{diag}(\gamma,\gamma,1,1)$, where $\gamma$ is a primitive element of $\F_q$. For $\sigma\in\Aut(\F_q)$, define $x^\sigma=(x_1^\sigma,\ldots,x_4^\sigma)$. Then we have $\Gamma\textup{Sp}(4,q)=\langle \textup{Sp}(4,q),\delta\rangle\rtimes\Aut(\F_q)$ by \cite[Section 1.7.1]{BHR407}. Each element of $\Gamma\textup{Sp}(4,q)$ has the form $x\mapsto x^\sigma A$, where $\sigma\in\Aut(\F_q)$ and $A\in \textup{GL}(4,q)$, and we denote this element by $(A,\sigma)$. We refer to $A$ as the \textit{matrix part} and refer to $\sigma$ as the \textit{Frobenius part} of the element $(A,\,\sigma)$ respectively.

Take a fixed projective point $P:=\langle(1,0,0,0)\rangle$. The stabilizer $\textup{Sp}(4,q)_P$ of $P$ in  $\textup{Sp}(4,q)$ consists of the matrices
\begin{equation}\label{eqn_SpP}
\begin{pmatrix}
\lambda&  0 & 0\\
-HJ\textbf{v}^{T} & H & 0\\
z & \textbf{v} & \lambda^{-1}
\end{pmatrix}
\end{equation}
with $H\in \textup{SL}(2,q)$, $\textbf{v}\in \F_q^2$, $z \in \F_q$ and $\lambda\in\F_q^*$,
where $J=\begin{pmatrix}0 & 1\\-1 & 0 \end{pmatrix}$. Together with $\delta$ and $\sigma$, they generate  $\Gamma\textup{Sp}(4,q)_P$. Modulo the center of $\Gamma\textup{Sp}(4,q)$, we obtain $\PGaSp(4,q)_P$, and we write $\delta$, $\sigma$ and $(A,\sigma)$ for their respective images by abuse of notation. We observe that $\delta$ has order $q-1$.

Let $G$ be a subgroup of $\PGaSp(4,q)_P$ that acts regularly on the points of $\cQ^P$. For each $g\in\PGaSp(4,q)_P$, $g$ stabilizes $\cQ^P$ and so $g^{-1}Gg$ is also a point regular group of $\cQ^P$. We thus assume without loss of generality that $G$ is contained in the Sylow $p$-subgroup of $\PGaSp(4,q)_P$ generated by the Sylow $p$-subgroup of $\Aut(\F_q)$ and the Sylow $p$-subgroup of $\textup{PSp}(4,q)_P$ consisting of lower triangular matrices with diagonal entries $1$. The element of $\textup{PSp}(4,q)_P$ in \eqref{eqn_SpP} is such a matrix if and only if $\lambda=1$, $H=\begin{pmatrix}1& 0\\t& 1 \end{pmatrix}$. Therefore, each element of $G$ is of the form $(E(a,b,c,t),\sigma)$, where $\sigma\in\Aut(\F_q)$ has order $p^i$ for some $i$, and
\begin{equation}\label{eqn_EMat}
E(a,b,c,t):=\begin{pmatrix}
1& 0& 0 & 0\\
-c& 1& 0& 0\\
b-ct&t&1&0\\
a & b&c & 1
\end{pmatrix},\quad a,\,b,\,c,\,t \in \F_q.
\end{equation}

The point set of $\cQ^P$ is $\{\la (a,b,c,1)\ra:\,a,\,b,\,c\in\F_q\}$. In order for $G$ to be point regular, for each triple $(a,b,c)$, there should be exactly one element of $G$ that maps $\la(0,0,0,1)\ra$ to $\la(a,b,c,1)\ra$, i.e., whose matrix part has $(a,b,c,1)$ as its last row. We denote such an element  by $\mathfrak{g}_{a,b,c}= (\cM_{a,b,c},\,\theta_{a,b,c})$,  where $\theta_{a,b,c}\in \Aut(\F_q)$ and
$\cM_{a,b,c}:=E\left(a,b,c,T(a,b,c)\right)$.
Here, $T$ is a function from $\F_q^3$ to $\F_q$. The  group multiplication $\circ$ of $G$ is
\begin{equation}\label{groupmulti}
\fg_{a,b,c}\circ \fg_{x,y,z}=(\cM_{a,b,c}^{\theta_{x,y,z}}\cdot \cM_{x,y,z},\,\theta_{a,b,c}\theta_{x,y,z}),
\end{equation}
where $\cM^{\theta}$ is the matrix obtained by applying $\theta$ to each entry of the matrix $\cM$, and $\cdot$ is the usual matrix multiplication.  To summarize, up to conjugacy in $\PGaSp(4,q)_P$ a point regular subgroup $G$ of $\cQ^P$ is of the form
$G=\{\fg_{a,b,c}:\, a,\,b,\,c \in \F_q \}$ for some functions $T:\,\F_q^3\rightarrow\F_q$ and $\theta:\,\F_q^3\rightarrow\Aut(\F_q)$, where $\theta(x,y,z)=\theta_{x,y,z}$.
\begin{thm} \label{Main}
Let $T:\,\F_q^3\rightarrow\F_q$ and $\theta:\,\F_q^3\rightarrow\Aut(\F_q)$ be two functions. Set
\[
\fg_{a,b,c}:=(\cM_{a,b,c},\,\theta_{a,b,c})\;\textup{ with } \cM_{a,b,c}:=E(a,b,c,T(a,b,c)),
\]
and write $\theta_{a,b,c}:=\theta(a,b,c)$. Define the set
$G:=\{\fg_{a,b,c}:\, a,\,b,\,c \in \F_q \}$. Then $G$ is a point regular group of the derived quadrangle $\cQ^P$ if and only if for any triples $(a,b,c)$ and $(x,y,z)$ we have $\fg_{a,b,c}\circ\fg_{x,y,z}=\fg_{u,v,w}$, which is equivalent to
\begin{align}
\theta_{a,b,c}\theta_{x,y,z}=&\theta_{u,v,w},\label{eqctheta}\\
T(a,b,c)^{\theta_{x,y,z}}+T(x,y,z)=&T(u,v,w),\label{eqc1}
\end{align}
where    $w=c^{\theta_{x,y,z}}+z$,  $v=b^{\theta_{x,y,z}}+y+c^{\theta_{x,y,z}}T(x,y,z)$ and
\begin{align*}
u&=a^{\theta_{x,y,z}}+x-b^{\theta_{x,y,z}}z+c^{\theta_{x,y,z}}y-c^{\theta_{x,y,z}}zT(x,y,z).
\end{align*}
\end{thm}
\begin{proof}
The element $\fg_{a,b,c}$ maps $\la(0,0,0,1)\ra$ to $\la(a,b,c,1)\ra$, so $G$ acts regularly on the points of $\cQ^P$ provided that it is indeed a group. Since $G$ is finite, it suffices to make sure that $G$ is closed under the multiplication $\circ$ in  \eqref{groupmulti}. By direct calculations, we deduce that the last row of the matrix part of $\fg_{a,b,c}\circ\fg_{x,y,z}$ is $(u,v,w,1)$ with $u,\,v,\,w$ as in the statement of the theorem. Therefore, we need to have $\fg_{a,b,c}\circ\fg_{x,y,z}=\fg_{u,v,w}$. The conditions in the theorem are obtained by comparing their Frobenius parts and the ($3,\,2$)-nd entry of their matrix parts.
\end{proof}

Let $G$ be a point regular automorphism group of the derived quadrangle $\cQ^P$ with associated functions $T$ and $\theta$ as in Theorem \ref{Main}. We define
\begin{align}
G_A&:=\{\fg_{a,0,0}:\,a\in \F_q\},\label{eqn_GAdef}\\
G_B&:=\{\fg_{0,b,0}:\,b\in \F_q\},\label{eqn_GBdef}\\
G_{A,B}&:=\{\fg_{a,b,0}:\,a,\,b \in \F_q\}.\label{eqn_GABdef}
\end{align}

\begin{corollary}\label{subgp}
Each of $G_A$, $G_B$ and $G_{A,B}$ is a subgroup of $G$.
\end{corollary}
\begin{proof}
Since $G$ is a finite group, it suffices to show that each of them is closed under multiplication. We take $G_A$ for example. By Theorem \ref{Main}, we deduce that $\fg_{a,0,0}\circ\fg_{x,0,0}=\fg_{u,0,0}$ with $u=a^{\theta_{x,0,0}}+x$, so $G_A$ is indeed closed under multiplication. The proofs are similar for $G_B$ and $G_{A,B}$.
\end{proof}

The following notation will be used throughout this  paper.
\begin{Notation}\label{notation_sigmaLMS}
Let $G$ be a point regular automorphism group of the derived quadrangle $\cQ^P$ with associated functions $T$ and $\theta$ as in Theorem \ref{Main}, and define
\[
\psi:\,G\rightarrow \Aut(\F_q),\;\fg_{a,b,c}\mapsto\theta_{a,b,c}
\]
Let  $G_A$, $G_B$ and $G_{A,B}$ be as in \eqref{eqn_GAdef}-\eqref{eqn_GABdef}, and let $p^{r_A}$,  $p^{r_B}$, $p^{r_{A,B}}$ be the exponents of their images under $\psi$ respectively. We use the shorthand notation for special cases of $T$ and $\theta$ in Table \ref{tab_notLMS}.
\begin{table}[hptp]
\caption{The shorthand notation for special cases of $T$ and $\theta$}
\label{tab_notLMS}
\begin{center}
\begin{tabular}{c|c}
\hline
$L(x)$     & $T(x,0,0)$       \\ \hline
$M(y)$     & $T(0,y,0)$       \\ \hline
$S(z)$     & $T(0,0,z)$       \\ \hline
$\sigma_c$ & $\theta_{0,0,c}$ \\ \hline
\end{tabular}
\end{center}
\end{table}
Finally, set $p^{r_C}:=\max\{o(\sigma_c):\,c\in\F_q\}$, and $s:=\max\{0,r_C-r_{A,B}\}$.
\end{Notation}

In general, the condition in Theorem \ref{Main} is fairly complicated. We consider some special cases, where the conditions can be simplified.

\begin{corollary}\label{cond1}
Suppose that $G$ is a point regular group of $\cQ^P$. Then for  $a$, $b$, $c$, $x$, $y$, $z$ in $\F_q$, it holds that
\begin{enumerate}
\item[(1)] $\theta_{a,0,0}\theta_{x,0,0}=\theta_{(a^{\theta_{x,0,0}}+x),0,0}$ and $L(a)^{\theta_{x,0,0}}+L(x)=L(a^{\theta_{x,0,0}}+x);$
\item[(2)]  $\theta_{0,b,0}\theta_{0,y,0}=\theta_{0,(b^{\theta_{0,y,0}}+y),0}$ and $M(b)^{\theta_{0,y,0}}+M(y)=L(b^{\theta_{0,y,0}}+y);$
\item[(3)]  $\sigma_c\sigma_z=\theta_{u,v,w}$ and  $S(c)^{\sigma_{z}}+S(z)=T(u,v,w)$, where $u=-\sigma_{z}(c)zS(z)$, $v=\sigma_{z}(c)S(z)$ and $w=\sigma_{z}(c)+z$;
\item[(4)]$\theta_{a,0,0}\theta_{0,b,0}=\theta_{a^{\theta_{0,b,0}},b,0}$ and  $L(a)^{\theta_{0,b,0}}+M(b)=T(a^{\theta_{0,b,0}},b,0);$
\item[(5)] $\theta_{0,b,0}\theta_{a,0,0}=\theta_{a,b^{\theta_{a,0,0}},0}$ and  $L(a)+M(b)^{\theta_{a,0,0}}=T(a,b^{\theta_{a,0,0}},0);$
\item[(6)] $\fg_{a,b,c}=\fg_{a',b',0}\circ \fg_{0,0,c}$,  or equivalently,
     $\theta_{a,b,c}=\theta_{a',b',0}\theta_{0,0,c}$ and $T(a,b,c)=T(a',b',0)^{\sigma_c}+S(c)$, where $a'=\sigma_c^{-1}(a+bc)$,  $b'=\sigma_c^{-1}(b)$;
\item[(7)] $\sigma_c\theta_{a,b,0}=\theta_{u,v,0}\sigma_{w} $ and $S(c)^{\theta_{a,b,0}}+T(a,b,0)=T(u,v,0)^{\sigma_{w}}+S(w)$,
where $w=c^{\theta_{a,b,0}}$, $v=\sigma_{w}^{-1}\left(b+c^{\theta_{a,b,0}}T(a,b,0) \right)$ and $u=\sigma_{w}^{-1}\left(a+2bc^{\theta_{a,b,0}}+c^{2\theta_{a,b,0}}T(a,b,0)\right)$.
 \end{enumerate}
\end{corollary}
\begin{proof}
The claims (1)-(5) are  special cases of  Theorem \ref{Main}, where exactly one of $\{a,b,c\}$ and exactly one of $\{x,y,z\}$ are nonzero. The first equation in claim (6) holds because the group $G$ is point regular and $\fg_{a',b',0}\circ \fg_{0,0,c}$ maps $(0,0,0,1)$ to
\[
(a'^{\sigma_c},b'^{\sigma_c},0,1)\,E(0,0,c,S(c))=
(a+bc,b,0,1)\begin{pmatrix}1& 0& 0 & 0\\-c& 1& 0& 0\\-cS(c)&S(c)&1&0\\0& 0&c & 1\end{pmatrix}
=(a,b,c,1).
\]
The second part of (6) then follows by Theorem \ref{Main}.

It remains to prove (7). Fix a triple $(a,\,b,\,c)$. By Theorem \ref{Main}, we have $\fg_{x,y,w}=\fg_{0,0,c}\circ\fg_{a,b,0}$, where  $x=a+bc^{\theta_{a,b,0}}$, $y=b+c^{\theta_{a,b,0}}T(a,b,0)$, $w=c^{\theta_{a,b,0}})$. By (6), we have $\fg_{x,y,w}=\fg_{u,v,0}\circ \fg_{0,0,w}$, where $u,v$ are as in the statement of (7). These two expressions of $\fg_{x,y,w}$ lead to two expressions of $\theta_{x,y,w}$ and $T(x,y,w)$ via \eqref{eqctheta} and \eqref{eqc1} respectively, and the claim in (7) follows by equating the respective expressions.
\end{proof}

\begin{corollary}\label{cor_rABnonlin}
Let $G$ be as in Notation \ref{notation_sigmaLMS}. Then $r_{A,B}=\max\{r_A,r_B\}$, and $G$ is nonlinear if and only if $r_{A,B}>0$ or $r_C>0$.
\end{corollary}
\begin{proof}
The claim on $r_{A,B}$ is a consequence of (4) of Corollary \ref{cond1}, since $\theta_{u,v,0}=\theta_{a,0,0}\theta_{0,v,0}$ with $a=\theta_{0,v,0}^{-1}(u)$.
Take $a,b,c\in\F_q$. By (6) of Corollary \ref{cond1}, we have $\theta_{a,b,c}=\theta_{a',b',0}\sigma_c$ for some $a',b'\in\F_q$. Hence $G$ is linear if and only if $G_{A,B}$ is linear and $r_C=0$. The second claim now follows from (1).
\end{proof}

\begin{remark}\label{rem_EMult}
Let $E$ be  as defined in  \eqref{eqn_EMat}. We shall make extensive use of the following calculations throughout the paper:
\begin{align*}
E(a,b,c,t)\cdot E(x,y,z,w)&=E(a+x-bz+cy-czw,b+y+cw,c+z,t+w),\\
E(a,b,c,t)^{-1}&=E(-a,-b+ct,-c, -t).
\end{align*}
In each  equation, the last two coordinates, i.e., the $(4,3)$-rd and the $(3,2)$-nd entries of the matrix, on the right hand side has a very simple form. This is particularly helpful in many circumstances where we are only concerned with these two coordinates.
\end{remark}

\subsection{Summary of the point regular groups of $\cQ^P$,  $q$ odd}
The following is a list of point regular automorphism groups of $\cQ^P$ for odd $q$ that we construct in this paper. The first construction also works for even $q$. The fact that the desired elements in each construction exist will be established in the subsequent sections.
There is another construction, Construction \ref{Const2}, for $q$ even.
\begin{construction}\label{const_S1}
For a prime power $q=p^m$, set $\theta_{a,b,c}\equiv 1$ and $T(a,b,c):=S_1(c)$ for a  reduced linearized polynomial $S_1$ over $\F_q$. The set $G$ as defined  in Theorem \ref{Main} for the prescribed functions $T$ and $\theta$ is  a point regular subgroup of $\cQ^P$.
\end{construction}

\begin{construction}\label{const_S2}
Suppose that $q=p^{pl}$ with $p$ an odd prime and $l$ a positive integer, and let $g\in\Aut(\F_q)$ be such that $g(x)=x^{p^l}$. Let $S_1(X)$ be a reduced linearized polynomial whose coefficients lie in $\F_{p^l}$. Take $\mu_C\in\F_{p^l}^*$. Set $\theta_{a,b,c}:=g^{\tr_{\F_q/\F_p}(\mu_C c)}$, and $T(a,b,c):=S_1(c)$ for $a,\,b,\,c\in\F_q$.  Then the set $G$ as defined in Theorem \ref{Main} with the prescribed functions $ T$ and $\theta$ is a  point regular group of $\cQ^P$.
\end{construction}

\begin{construction}\label{const_S3}
Suppose that $q=p^{pl}$ with $p$ an odd prime and $l$ a positive integer, and let $g\in\Aut(\F_q)$ be such that $g(x)=x^{p^l}$. Take $\mu_B\in\F_{p^l}^*$ and  a tuple $(s_0,\,s_1,\cdots,s_{pl-1})$ with entries in $\F_{p^l}$ such that $\mu_B s_i- s_{pl-i}^{p^i}\mu_B^{p^i}=0$ for $1\leq i\leq pl-1$. Set $S_1(X):=\sum_{i=0}^{pl-1}s_iX^{p^i}$, and $Q(x):=-\tr_{\F_q/\F_p}(\mu_BxS_1(x))$ for $x\in\F_q$. Set \[\theta_{a,b,c}:=g^{\frac{1}{2}Q(c)+\tr_{\F_q/\F_p}(\mu_B b)},\,\quad  T(a,b,c):=S_1(c) \text{ for } a,\,b,\,c\in\F_q.\] Then the set $G$ as defined in Theorem \ref{Main} with the prescribed functions $T$ and $\theta$ is a point regular subgroup  of $\cQ^P$.
\end{construction}

\begin{construction}\label{const_S4}
Suppose that $q=3^{9l}$ with $l$ a positive integer, and take $g\in\Aut(\F_q)$ such that $g(x)=x^{3^l}$. Set $g_1:=g^3$.
\begin{enumerate}
	\item[(i)] Take $u\in \F_{3^{3l}}$ such that $\mu_C:=u-u^g\in\F_{3^l}^*$;
	\item[(ii)] Take   $t_C\in\F_q^*$  such that $\lambda_C:=\tr_{\F_q/\F_3}(\mu_C t_C)\ne 0$;
	\item[(iii)] Take $\mu_B\in\F_{3^l}^*$;
	\item[(iv)] Take a tuple $(s_0,\,s_1,\cdots,s_{9l-1})$ with entries in $\F_{3^l}$ that satisfies
	\[
	-\mu_B s_i+s_{9l-i}^{3^i}\mu_B^{3^i}=\mu_C u^{3^i}-u\mu_C^{3^i},\;1\leq i\leq 9l-1;
	\]
	\item[(v)] Take $\alpha\in\F_{3^{3l}}$, $\lambda\in\F_3$ such that
		$g(\alpha)-\alpha=\lambda_Cu+\lambda\mu_C$.
\end{enumerate}
Set $S_1(X):=\sum_{i=0}^{9l-1}s_iX^{3^i}$, $Q(x):=-\tr_{\F_q/\F_3}(\mu_BxS_1(x))$ for $x\in\F_q$.  Set $K:=\{z\in\F_q :\,\tr_{\F_q/\F_3}(\mu_C z)=0\}$. For $a,\,b\in\F_q$ and $c\in K$, define
		\[
		\cM_{a,b,c}=E(a,b,c,S_1(c)),\quad \theta_{a,b,c}:=g_1^{\frac{1}{2}Q(c)+\tr_{\F_q/\F_3}(\alpha c+\mu_Bb)},
		\]
		and $\fg_{a,b,c}=(\cM_{a,b,c},\,\theta_{a,b,c})$. Then $G_{K}:=\{\fg_{a,b,c}:\,a,\,b\in\F_q,\, c\in K\}$ is a group  of order $q^3/3$. Set $\fg_{0,0,t_C}:=(\cM_{0,0,t_C},\, g)$ with $\cM_{0,0,t_C}:=E(0,0,t_C,S_1(t_C))$. Then $G:=\la G_K,\, \fg_{0,0,t_C}\ra$ is a point regular group of $\cQ^P$.
	\end{construction}

The main body of this paper is devoted to the proof of the following classification theorem in the odd characteristic case.
\begin{thm}\label{thm_conjugate}
Let $G$ be a point regular automorphism group of the derived quadrangle $\cQ^P$ of the symplectic quadrangle $\cQ=W(q)$, $q$ odd and $q\ge 5$. Then $G$ is conjugate to one of the groups in Constructions \ref{const_S1}-\ref{const_S4}.
\end{thm}

Here we give a brief description of our strategy for the proof of Theorem \ref{thm_conjugate}. The linear case is handled in Section \ref{sec_linear} by exploring the conditions in Corollary \ref{cond1} which take simple forms in this case. The arguments also work for even $q$ and lead to a new construction, Construction \ref{Const2}. The nonlinear case also relies on a detailed analysis of the conditions in Corollary \ref{cond1} but is considerably more involved. The first major step is to bound the order of the Frobenius part of $G$ with the aid of Lemma \ref{Wt1cond}, cf. Theorems \ref{thm_GA_struc1} and \ref{thm_G_struc}. This step is crucial, since it implies that the ``linear" part of the group $G$, i.e., the kernel $H$ of the homomorphism from $G$ to $\Aut(\F_q)$ that maps each element to its Frobenius part, has a small index in $G$. This step also works for even $q$. As in the linear case, we are able to obtain a good description of the group structure of $H$ by examining the conditions in Corollary \ref{cond1}. Take an element $g\in G$ that has a Frobenius part of the largest possible order $b=[G:\,H]$. Then we have $G=\la H,g\ra$, $g^b\in H$ and $H\unlhd G$. We then further derive restrictions on the parameters of $H$ and $g$ by considering the conditions $H^g\le H$, $g^b\in H$ and $G$ is transitive on $\cQ^P$. The actual analysis is more complicated but follows the same spirit. This leads to the classification of the nonlinear case into three families for odd $q$, $q\ge 5$, in Theorem \ref{thm_nonlinear}. In Section \ref{sec_iso}, we further simplify the constructions up to conjugacy in $\PGaSp(4,q)_P$, which concludes the proof of Theorem \ref{thm_conjugate}.

\begin{remark}
Construction \ref{const_S1} is implicitly known  in \cite{Chen2013} and \cite{DeWinter2014Criterion} in the case $q$ is odd,  and it includes the constructions given in \cite{Bamberg2011PRG}. Write $q=p^m$ with $p$ prime. For $\alpha\in\F_q$, define $\theta_{\alpha}:=E(0,0,\alpha,\alpha)$ and $t_{0,0,\alpha}:=E(0,0,\alpha,0)$. Let $\{\alpha_1,\cdots,\alpha_m\}$ be a basis of $\F_q$ over $\F_p$. Take the dual basis $\{\beta_1,\cdots, \beta_m\}$ such that $\tr_{\F_q/\F_p}(\alpha_i \beta_j)=1$ or $0$ according as $i=j$ or not. For $1\le k\le m-1$, set $T_k(x,y,z):=\sum_{i=1}^k\tr_{\F_q/\F_p}(\beta_i z)\alpha_i$,
and let $G_k$ be the group arising from Construction \ref{const_S1} with $T=T_k$. Then in $G_k$ we have $\fg_{0,0,\alpha_i}=\theta_{\alpha_i}$ for $1\le i\le k$ and $\fg_{0,0,\alpha_i}=t_{0,0,\alpha_i}$ for $k+1\le i\le m$. In the notation of \cite{Bamberg2011PRG},   the group $G_k$ is $S_{U,W}$ with $U=\la \alpha_1,\cdots,\alpha_k\ra_{\F_p}$ and $W=\la \alpha_{k+1},\cdots,\alpha_m\ra_{\F_p}$.
\end{remark}

In Section \ref{sec_iso}, we will show that the point regular groups arising from  Constructions \ref{const_S1}-\ref{const_S4} are in general nonisomorphic by calculating their group invariants such as exponents and Thompson subgroups. In the case of Constructions \ref{const_S2}-\ref{const_S4}, we will show in Theorem \ref{thm_G1nil} that the nilpotency class of the  resulting  group lies in the range $[2p^e,3p^e]$ in the case $l>1$, where $p^e=o(g)$. In Table \ref{table_ncODD}, we give explicit values of nilpotency classes for some special cases of Construction \ref{const_S2} with $\mu_C=1$ and Construction \ref{const_S3} with $\alpha=0$, $\mu_B=1$. In both cases, we assume that $l>1$. From the table we see that in general the nilpotency class of $G$ is larger if $\ker(S_1)$ has a smaller size.
\begin{table}[hptp]
	\caption{ The nilpotency class of the point regular group $G$. In Construction \ref{const_S2} (\textit{resp.} \ref{const_S3}), take  $\mu_C=1$ (\textit{resp.} $\alpha=0$, $\mu_B=1$) and assume that $l>1$.}\label{table_ncODD}
	\begin{center}
		\begin{tabular}{|c|c|c|c|}
			\hline
			Construction & $S_1(z)$                     & nilpotency class & Condition  \\ \hline
			\multirow{4}{*}{\begin{tabular}[c]{@{}c@{}} \ref{const_S2} \end{tabular}} & $0$                                    & $2p$    &                                                      \\ \cline{2-4}
			& $z^{p^k}$                        & $3p$    &     $l\nmid k$                                                                              \\ \cline{2-4}
			& $z^{p^k}$                        & $3p-1$  &     $l\mid k$                                                                         \\ \cline{2-4}
			& $(1-g)^k(z)$                     & $3p-k$  &      $1\le k\le p-1$                                                                            \\ \hline
			\multirow{4}{*}{\begin{tabular}[c]{@{}c@{}}\ref{const_S3}\\  \end{tabular}} & $0$                                    & $2p$    &   \\ \cline{2-4}
			& $z$                                               & $3p-1$  &                                                                                  \\ \cline{2-4}
			& $z^{p^k}+z^{p^{pl-k}}$       & $3p$    &    $1\leq k \leq pl-1$                                                                              \\ \cline{2-4}
			& $(1-g)^{2k} (z^{p^{pl-kl}})$ & $3p-2k$ &    $2\leq 2k \leq p-1$                                                                              \\ \hline
		\end{tabular}
	\end{center}
\end{table}

\section{Arithmetic of the field $\F_q$}
	
This section contains facts and technical results that concern the arithmetic properties of the finite field $\F_q$. We refer the reader to the standard textbook \cite{LidlFF} on finite fields. The reader is suggested to skip the proofs in this section and focus on the more important proofs in the later sections for the first reading.  

Let $\F_q$ be the finite field with $q$ elements, where $q=p^m$ with $p$  prime. Let $\Aut(\F_q)$ be the Galois group of the field $\F_q$, consisting of the Frobenius maps $x\mapsto x^{p^i}$, $0\le i\le m-1$. For $g\in\Aut(\F_q)$, we write both $x^g$ and $g(x)$ for its action on $x$. For a divisor $d$ of $m$, the trace function from $\F_q$ to the subfield $\F_{p^d}$ is
\[
\tr_{\F_q/\F_{p^d}}(x):=x+x^{p^d}+\cdots+x^{p^{m-d}},\quad x \in \F_q.
\]
The trace function is surjective, i.e., $\{\tr_{\F_q/\F_{p^d}}(x):\,x\in\F_{q}\}=\F_{p^d}$.

\begin{lemma}\label{lem_trace_int}
Let $d$ be a divisor of a positive integer $m$, and set $q=p^m$ for a prime $p$. Then $\tr_{\F_q/\F_{p^d}}(x^{p^{id}}y)=\tr_{\F_q/\F_{p^d}}(xy^{p^{m-id}})$ for $0\le i\le m/d$ and $x,\,y\in\F_q$.
\end{lemma}
\begin{proof}
Write $s_i:=\tr_{\F_q/\F_{p^d}}(x^{p^{id}}y)$ for $0\le i\le m/d$. It is in $\F_{p^d}$, so we have
\[
s_i=s_i^{p^{(m/d-i)d}}=\tr_{\F_q/\F_{p^d}}\left((x^{p^{id}}y)^{p^{(m/d-i)d}}\right)=\tr_{\F_q/\F_{p^d}}(xy^{p^{m-id}}).
\]
This proves the lemma.
\end{proof}

\begin{lemma}\cite[Theorem 2.25]{LidlFF}\label{lem_ff1}
Let $d$ be a divisor of a positive integer $m$. For $\alpha\in \F_{p^m}$,  $\tr_{\F_{p^m}/\F_{p^d}}(\alpha)=0$ if and only if $\alpha=\beta^{p^d}-\beta$ for some $\beta \in \F_{p^m}$.
\end{lemma}

By \cite[Theorem 2.23]{LidlFF}, the trace function $\tr_{\F_q/\F_p}$ is $\F_p$-linear, and $\tr_{\F_q/\F_p}(x^p)=\tr_{\F_q/\F_p}(x)$ for $x\in \F_q$. For each $\beta\in\F_q$, define
\begin{equation*} 
L_\beta:\,\F_q\rightarrow \F_p, \quad x\mapsto \tr_{\F_q/\F_p}(\beta x).
\end{equation*}
We have $L_0\equiv 0$, $\lambda L_\beta=L_{\lambda\beta}$ for $\lambda\in\F_p$.
\begin{lemma}\cite[Theorem 2.24]{LidlFF}\label{lem_ff2}
Each $\F_p$-linear transformation from $\F_q$ to $\F_p$ equals $L_{\beta}$ for some $\beta\in \F_q$, and $L_{\beta}=L_{\gamma}$ if and only if $\beta=\gamma$.
\end{lemma}
As a corollary,  $\tr_{\F_q/\F_p}(\beta x)=0$ for all $x\in\F_q$ if and only if $\beta=0$.

\begin{lemma}\label{lem_tr}
For $\alpha,\beta\in \F_q$ with $\beta\ne 0$, if $\ker(L_\alpha)$ contains the subspace $K=\{x\in \F_q:\,\tr_{\F_q/\F_p}(\beta x)=0\}$, then $\alpha=\lambda\beta$ for some $\lambda\in \F_p$.
\end{lemma}
\begin{proof}
We have $K=\ker(L_\beta)$. Take $u\in\F_q\setminus K$ such that $\tr_{\F_q/\F_p}(\beta u)=1$. Then $\F_q=K\oplus\F_p\cdot u$. Since $K\subseteq \ker(L_\alpha)$, we have $L_\alpha(x+\lambda u)=\lambda L_\alpha(u)$ for $x\in K$ and $\lambda\in\F_p$. On the other hand, $L_\beta(x+\lambda u)=\lambda$, so $L_\alpha(z)=L_\alpha(u)L_\beta(z)$ for all $z\in\F_q$. Since $L_\alpha(u)$ is in $\F_p$ and $L_\alpha(u)L_\beta=L_{L_\alpha(u)\beta}$, the claim now follows from Lemma \ref{lem_ff2}.
\end{proof}

\begin{lemma}\label{lem_mu_inv}
Suppose that $q=p^m$ with $p$ prime, and $g\in\Aut(\F_q)$ has order $p^r$ with $r\ge0$. If $K=\{x\in\F_q:\,\tr_{\F_q/\F_p}(\mu x)=0\}$ is $g$-invariant,  then $g(\mu)=\mu$.
\end{lemma}
\begin{proof}
The cases $r=0$ and $\mu=0$ are both trivial, so we assume that $r\ge 1$ and $\mu\ne 0$. We have $g(K)=\{x\in\F_q:\,\tr_{\F_q/\F_p}(g(\mu) x)=0\}$, i.e., $g(K)=\ker(L_{g(\mu)})$. Since $K=g(K)$, $\ker(L_{g(\mu)})$ equals $K$, and so $g(\mu)=\lambda\mu$ for some $\lambda\in\F_p$ by Lemma \ref{lem_tr}.
Taking the relative norm to the subfield fixed by $g$, we see that $\lambda^{p^r}=1$, i.e., $\lambda=1$. This completes the proof.
\end{proof}

\begin{lemma}\label{Imdim1}
Let $x\mapsto f(x)$ be an $\F_p$-linear transformation of $\F_q$, and set $K=\{x\in\F_q:\,\tr_{\F_q/\F_p}(\eta x)=0\}$ for some $\eta\in\F_q$. If $\im(f|_{K})=\F_p\cdot \omega$, then there exist  $u\in\F_q,\,\mu\in\F_q^*$ such that $f(x)=\omega\tr_{\F_q/\F_p}(\mu x)+u\tr_{\F_q/\F_p}(\eta x)$ for $x\in\F_q$.
\end{lemma}
\begin{proof}
We first consider the case $\eta=0$, i.e., $\im(f)=\F_p\cdot \omega$. If $\omega=0$, then $f$ is constantly zero, and the claim holds with $u=\mu=1$. If $\omega\ne 0$, then $x\mapsto\omega^{-1}f(x)$ is a nonzero $\F_p$-linear map from $\F_q$ to $\F_p$, and so $f(x)=\omega\tr_{\F_q/\F_p}(\mu x)$ for some $\mu\in\F_q^*$ by Lemma \ref{lem_ff2}. The element $u$ is irrelevant here, since $\eta=0$.

We next consider the case  $\eta\ne 0$. Take $\beta\in\F_q$ such that $\tr_{\F_q/\F_p}(\eta  \beta)=1$, and define $F(x):=f(x)-f(\beta)\cdot\tr_{\F_q/\F_p}(\eta x)$. Then $\F_q=K\oplus\F_p\cdot u$, $F$ is $\F_p$-linear, and $F(a+\lambda \beta)=f(a)$ for $a\in K$ and $\lambda\in\F_p$. It follows that $\im(F)=\F_p\cdot \omega$. By the previous case, we have $F(x)=\omega\tr_{\F_q/\F_p}(\mu x)$ for some $\mu\in\F_q^*$, and the claim follows. This completes the proof.
\end{proof}

For two subsets $A,\,B$ of $\F_q$, we define $A\cdot B:=\{xy:\, x \in A, \,y \in B\}$, and write $\langle A \rangle_{\F_p}$ for the $\F_p$-subspace spanned by elements of $A$.

\begin{lemma} \label{lem_AB_prod}
	Suppose that $q=p^m$ with $p$ prime and $m> 2$, and let $A,\,B$ be two $\F_p$-subspaces of $(\F_q,+)$ of codimension $1$. Then  $\la A\cdot A\ra_{\F_p}=\la A\cdot B\ra_{\F_p}=\F_q$.
\end{lemma}
\begin{proof}
By replacing $A$ with $\{ax:\,a\in A\}$ for some $x\in\F_q^*$ if necessary, we assume that $1\in A$ without loss of generality; similarly, assume that $1\in B$. Write $W=\la A\cdot A\ra_{\F_p}$, and assume that $W\ne\F_q$. Since $A\le W\le\F_q$ and $A$ has codimension $1$, we have $W=A$. It follows that the subspace  $A$  is also closed under multiplication, so $A$ is a proper subfield of $\F_q$. We thus have $q\ge |A|^2$.  On the other hand, $A$ has codimension $1$ by assumption, i.e., $q=p\cdot|A|$.  We deduce that either $(|A|,\,q)=(p,\,p^2)$ or $(|A|,\,q)=(1,p)$, both contradicting the assumption that $m>2$. To sum up, we have shown that $\la A\cdot A\ra_{\F_p}=\F_q$.
	
Now assume that $A\ne B$. Then we have $A+B=\F_q$ by considering dimensions. Write $U=\la A\cdot B\ra_{\F_p}$. It contains both $A$ and $B$ by the assumption $1\in A\cap B$, so it also contains the subspace $A+B=\F_q$. This completes the proof.
\end{proof}

By \cite[Theorem 1.71]{LidlFF}, for each function $f:\,\F_q\rightarrow\F_q$ there is a unique polynomial $F(X)$ of degree at most $q-1$ in $\F_q[X]$ such that $F(a)=f(a)$ for all $a\in\F_q$. We call $F(X)$ the reduced polynomial associated with the function $f$. It is conventional to write $f(X)$ for the associated reduced polynomial of a function $f$. 
A \textit{linearized} polynomial over $\F_q$ is a polynomial  $f$ of the form $f(X)=\sum_{i=0}^{n}a_iX^{p^k}$, $a_i\in\F_q$. It is \textit{reduced} if $n\le m-1$, where $q=p^m$. There is a bijection between $\F_p$-linear transformations of $\F_q$ and the reduced linearized polynomials over $\F_q$. If $f(X) $ is a reduced linearized polynomial over $\F_q$, then its \textit{trace dual} is the (unique) reduced linearized polynomial $\tilde{f}(X)$ such that $\tr_{\F_q/\F_p}(f(x)y)=\tr_{\F_q/\F_p}(\tilde{f}(y)x)$ for $x,\,y\in\F_q$. If $f(X)=\sum_{i=0}^{m-1}s_iX^{p^i}$, then we apply Lemma  \ref{lem_trace_int} to obtain that
$\tilde{f}(X)=\sum\limits_{i=0}^{m-1}s_{m-i}^{p^i}X^{p^i}$.

A map $\cB:\,\F_q\times\F_q\rightarrow\F_p$ is a \textit{bilinear form} if $\cB(x,y)$ is additive in both $x$ and $y$. It is \textit{symmetric} if $\cB(x,y)=\cB(y,x)$ for $x,y\in\F_q$. In the next lemma, we associate bilinear forms over $\F_q$ with linearized polynomials.
\begin{lemma}\label{lem_bform}
Suppose that $\cB:\,\F_q\times\F_q\rightarrow\F_p$ is a bilinear form. Then there is a reduced linearized polynomial $f$ over $\F_q$ such that $\cB(x,y)=\tr_{\F_q/\F_p}(xf(y))$.
\end{lemma}
\begin{proof}
By Lemma \ref{lem_ff2}, for each $y\in \F_q$ there exists an element $f(y)\in\F_q$ such that
$\cB(x,y)=\tr_{\F_q/\F_p}(xf(y))$. For $y,z\in\F_q$, we have $\cB(x,y+z)=\cB(x,y)+\cB(x,z)$, i.e., $\tr_{\F_q/\F_p}(x(f(y+z)-f(y)-f(z)))=0$. This holds for all $x\in\F_q$, so $f(y+z)-f(y)-f(z)=0$, i.e., $f$ is additive. This completes the proof.
\end{proof}

\begin{lemma}\label{Lpoly}
Suppose that $q=p^{m}$  with $p$ prime and  $m=p^el$, and take $g\in\Aut(\F_q)$ such that $g(x)= x^{p^l}$. Take $\eta\in\F_{p^l}$, and define $K:=\{x \in\F_q:\,\tr_{\F_q/\F_p}(\eta x)=0\}$. If $f$ is an $\F_p$-linear transformation of $\F_q$ such that $g(f(g^{-1}(x)))=f(x)$  for $x\in K$, then there exists a reduced linearized polynomial  $f_1(X)$ over $\F_q$ with coefficients in $\F_{p^l}$ such that $f(x)=f_1(x)$ for $x\in K$.
\end{lemma}
\begin{proof}
Suppose that $f(x)=\sum_{i=0}^{m-1}d_ix^{p^i}$ for $x\in\F_q$. We set
\begin{equation*}
	D(x):=g(f(g^{-1}(x)))-f(x)=\sum_{i=0}^{m-1}\left(g(d_i)-d_i\right)x^{p^i},
\end{equation*}
which is zero for $x\in K$ by assumption, i.e., $K\le\ker(D)$. It follows that $\dim_{\F_p}(\im(D))\le 1$. If $\im(D)=0$, then we have $g(d_i)=d_i$ for $0\le i\le m-1$, i.e., $d_i\in\F_{p^l}$, and the claim follows. Therefore, we assume that   $\im(D)$ has dimension $1$ over $\F_p$. From $K\le\ker(D)$ we deduce that $K\ne\F_q$, i.e., $\eta\ne 0$; also, we have $D(x)=u\tr_{\F_q/\F_p}(\eta x)$ for  some $u\in\F_q$ by Lemma \ref{Imdim1}.	Therefore, we have a polynomial equation
$\sum_{i=0}^{m-1}(g(d_i)-d_i)X^{p^i}=\sum_{i=0}^{m-1}u\eta^{p^i} X^{p^i}$.
By comparing the coefficients of both sides,  we deduce that $g(d_i)-d_i=u \eta^{p^i},\ 0\leq i\leq m-1$. By assumption we have $g(\eta)=\eta$, so $\eta\tr_{\F_q/\F_{p^l}}(u)=\tr_{\F_q/\F_{p^l}}(g(d_0)-d_0)=0.$
By Lemma \ref{lem_ff1}, $u=g(v)-v$ for some $v\in\F_{q}$.
We deduce from $g(d_i)-d_i=u\eta^{p^i}$ that $h_i:=d_i-v\eta^{p^i}$ is fixed by $g$, i.e., lies in $\F_{p^l}$. Now define $f_1(X):=\sum_{i=0}^{m-1}h_iX^{p^i}$. For $x\in K$, we have
\[
	f(x)-f_1(x)=\sum_{i=0}^{m-1}v\eta^{p^i}x^{p^i}=v\tr_{\F_q/\F_p}(\eta x)=0,
\]
so $f_1$ is the desired polynomial. This completes the proof.
\end{proof}

\begin{lemma}\label{lem_Tracelinear}
	Suppose that $q=p^m$ with $p$ prime. Let $f(X)=\sum_{i=0}^{m-1}s_iX^{p^i}$ be a reduced linearized polynomial over $\F_q$ and $\tilde{f}$ be its trace dual. Take $\mu\in\F_q$. Suppose that $\cB(c,z)=\tr_{\F_q/\F_p}(\mu c f(z))$ is a symmetric bilinear form on the $\F_p$-subspace $K=\{x\in\F_q:\,\tr_{\F_q/\F_p}(\eta x)=0\}$ for some $\eta\in\F_q$. Then  there exists  $u\in\F_q$ such that $ \tilde{f}(\mu x)=\mu f(x)-\eta\tr_{\F_q/\F_p}(ux)+u\tr_{\F_q/\F_p}(\eta x)$ for $x\in\F_q$ and
	\begin{equation}\label{eqn_sismmi}
	\mu s_i- s_{m-i}^{p^i}\mu^{p^i}=\eta u^{p^i}-u\eta^{p^i},\quad 0\leq i\leq m-1.
	\end{equation}
	Moreover, if $q$ is even and $K=\F_q$, then $\cB(c,c)=0$ for all $c\in\F_q$ if and only if $\mu s_0=0$.
\end{lemma}
\begin{proof}
The proofs for the case $\eta=0$ and the case $\eta\ne 0$ are similar, and we only prove the more complicated case $\eta\ne 0$ here. If $\mu=0$, then we can simply take $u=0$ and the claim is trivial. We assume that $\mu\ne 0$ in the sequel. We have $\cB(c,z)=\tr_{\F_q/\F_p}(\tilde{f}(\mu c)z)$, where $\tilde{f}$ is the trace dual of $f$. Since $\cB$ is symmetric on $K$, we have $\cB(c,z)=\cB(z,c)$, i.e., $\tr_{\F_q/\F_p}((\mu f(c)-\tilde{f}(\mu c))z)=0$ for $c,z\in K$. By Lemma \ref{lem_tr}, we deduce that $\mu f(c)-\tilde{f}(\mu c)\in\F_p\cdot \eta$ for each $c\in K$. By Lemma \ref{Imdim1}, there exist $u,\, v\in\F_q^*$ such that $\mu f(c)-\tilde{f}(\mu c)=\eta\tr_{\F_q/\F_p}(uc)+v\tr_{\F_q/\F_p}(\eta x)$.
	The first part of the claim then follows by comparing the coefficients of the corresponding polynomial identity $\mu f(X)-\tilde{f}(\mu X)=\sum_{i=0}^{m-1}\eta u^{p^i}X^{p^i}+\sum_{i=0}^{m-1}v\eta^{p^i}X^{p^i}$. Here, by comparing the coefficient of $X$, which is $0$ on the left hand side, we deduce that $v=-u$.
	
	Now assume that $q$ is even and $K=\F_q$, i.e., $\eta=0$. We only handle the case $m$ is even, since the case $m$ is odd is similar. By taking $i=m/2$ in  \eqref{eqn_sismmi}, we deduce that $\mu s_{m/2}\in\F_{2^{m/2}}$. We calculate that
	\begin{align*}
	\cB(c,c)&=\tr(\mu s_0c^2)+\tr(\mu s_{m/2}c^{2^{m/2}+1})+\sum_{i=1}^{m/2-1}\tr( \mu s_ic^{2^i+1}+\mu s_{m-i}c^{2^{m-i}+1})\\
	&=\tr(\mu s_0c^2),
	\end{align*}
where $\tr=\tr_{\F_q/\F_2}$. The second trace term is $0$ since $\mu s_{m/2}c^{2^{m/2}+1}\in\F_{2^{m/2}}$, and each summand in the third sum vanishes by  \eqref{eqn_sismmi} and Lemma \ref{lem_trace_int}. The second part of the lemma now follows. This completes the proof.
\end{proof}

Suppose that $q=p^m$ and $m=p^el$ for some positive integers $e,l$.  Take $g\in\Aut(\F_q)$ such that $g(x)=x^{p^l}$ for $x\in\F_q$. Let $\F$ be a subfield of $\F_q$. For $\alpha=\sum_{i=0}^{p^e-1}a_ig^i\in\F[\la g\ra]$ and $x\in\F_q$, we define $\alpha(x):=\sum_{i=0}^{p^e-1}a_ig^i(x)$. In this way, $\F_q$ becomes an $\F[\la g\ra]$-module. Here are some basic facts, cf. \cite{AninGR,Groupring}:
\begin{enumerate}
	\item[(i)] The ring $\F[\la g\ra]$ is a uniserial local ring, and $(1-g)^i\F[\la g\ra]$, $0\leq i\leq p^{e}$, are all its ideals. The dimension of $(1-g)^i\F[\la g\ra]$ over $\F$ is $p^e-i$, where $0\leq i\leq p^{e}$.
	\item[(ii)] We have $(1-g)^{p-1}=1+g+\cdots+g^{p-1}$ by binomial expansion, since
    \[
     \binom{p-1}{i}=\frac{(p-1)\cdots(p-i+1)(p-i)}{1\cdots (i-1)i}\equiv(-1)^i\pmod{p}.
    \]
    It follows that $(1-g)^p=1-g^p$. Inductively, it holds that
	\begin{equation}\label{eqn_1mgpow}
	(1-g)^{p^i-1}=(1-g)^{p-1}(1-g^p)^{p^{i-1}-1}=1+g+\cdots+g^{p^i-1}.
	\end{equation}
\end{enumerate}
As a corollary, we have $\tr_{\F_q/\F_{p^l}}(x)=\sum_{i=0}^{p^e-1}g^i(x)=(1-g)^{p^e-1}(x)$ for $x\in\F_q$.

\begin{lemma}\label{Wt1cond}
	Suppose that $q=p^{p^el}$ with $p$ prime and $e\ge 1$, and take $g\in\Aut(\F_q)$ such that $g(x)=x^{p^l}$ for $x\in\F_q$.  Then there exists a pair $(W,\,t)$  such that
	\begin{itemize}
		\item[(1)] $W$ is a $g$-invariant $\F_p$-subspace of codimension $h$ in $\F_q$ with $0< h\leq e$,
		\item[(2)] $t$ is an element of $\F_q$ such that $W_i:=W+t_i$, $0\leq i\leq p^h-1$, are pairwise disjoint, where $t_0=0$ and $t_i= (1+\cdots+g^{i-1})(t)$,  $1\leq i\leq p^h-1$,
	\end{itemize}
if and only if $h \geq p^{h-1}$, i.e., $h=1$ if $p$ is odd and $h=1$ or $2$ if $p=2$.
\end{lemma}
\begin{proof}
Write $R=\F_p[\la g\ra]$, and define  $R_i:=(1-g)^iR$  for  $0\le i\le p^e$. The ring $R$ is local with a unique maximal ideal $R_1$. The ideal $R_1$ is nilpotent, and an element $a$ is invertible in $R$ if and only if its image in the quotient ring $R/R_1\cong\F$ is invertible. We have a chain of ideals of $R$:
\begin{equation*}
R=R_0\supseteq R_1\supseteq\cdots \supseteq R_{p^e-1}\supseteq R_{p^e}=0,
\end{equation*}
with $\dim_{\F_p}(R_i/R_{i+1})=1$ for $0\le i\le p^e-1$. In particular, $\dim_{\F_p}(R_i)=p^e-i$.

There exists a (normal) basis of $\F_q$ over $\F_{p^l}$ of the form $\{g^i(\eta):\,0\le i\le p^e-1\}$ with $\eta\in\F_q$ by \cite[Theorem 2.35]{LidlFF}, so $\F_q$ is a free $\F_{p^l}[\la g\ra]$-module with generator $\eta$, i.e., $\F_q=R'\cdot \eta$, where $R'=\F_{p^l}[\la g\ra]$.
Take $\xi_1,\xi_2,\cdots,\xi_{l}$  to be a basis of $\F_{p^l}$ over $\F_p$. Then $R'= \xi_1 R\oplus \xi_2R\oplus\cdots\oplus \xi_lR$ and
\begin{equation*}
\F_q= R\cdot \xi_1\eta\oplus R\cdot \xi_2\eta\oplus \cdots \oplus R\cdot \xi_l\eta.
\end{equation*}
Here, each $R\cdot \xi_i\eta$ is a free $R$-module with generator $\xi_i\eta$.  The submodules of $R\cdot \xi_i\eta$ are $R_k\cdot \xi_i\eta$, $0\le k\le p^e-1$. In particular, $R\cdot \xi_i\eta$ has a unique submodule of each possible dimension, and these submodules form a chain under containment.

We claim that  the condition (2) can be reduced to $t_{p^{h-1}}\not\in W$ assuming that the condition (1) holds. Take two numbers $i,j$ such that $0\le i<j\le p^h-1$, and write $j-i=p^ku$, where $\gcd(p,u)=1$ and $0\le k\le h-1$. It holds that $t_j-t_i=g^i(t_{j-i})$, so we deduce that $W+t_i\cap W+t_j=\emptyset$ if and only if $t_j-t_i\in W$, i.e., $t_{j-i}\not\in W$ by the $g$-invariance of $W$. It holds that $\sum_{i=0}^{p^ku-1} g^i=(\sum_{i=0}^{u-1}g^{ip^k})\cdot (\sum_{j=0}^{p^k-1}g^{j})$, so $t_{j-i}=(1+g_k+\cdots+g_k^{u-1})(t_{p^k})$, where $g_k:=g^{p^k}$. Since the quotient image of $1+g_k+\cdots+g_k^{u-1}$ in $R/R_1$ equals $\bar{u}$ and is invertible in $R/R_1\cong\F_p$, we deduce that it is invertible in $R$. Since $W$ is a $R$-module by the condition (1), it follows that $t_{j-i}\not\in W$ if and only if $t_{p^k}\not\in W$. The condition (2) is now reduced to $t_{p^k}\not\in W$ for $0\le k\le h-1$. By  \eqref{eqn_1mgpow}, $t_{p^k}=(1-g)^{p^k-1}(t)$, which is a generator of the submodule $R_{p^k-1}\cdot t=R\cdot t_{p^k}$. By applying the ideals in the chain $\{R_i\}$ to $t$, we obtain the chain $R\cdot t_{p^{0}}\supseteq\cdots\supseteq R\cdot t_{p^{h-1}}$. Since $W$ is a $R$-module, we see that $t_{p^{h-1}}\not\in W$ implies that $t_{p^k}\not\in W$ for $0\le k\le h-1$. This proves the claim.

We are now in a position to complete the proof. The condition (1) is equivalent to that $W$ is a $R$-submodule of $\F_q$ of codimension $h$. By  \eqref{eqn_1mgpow}, $t_{p^{h-1}}=(1-g)^{p^{h-1}-1}(t)$. The existence of $t$ with $(1-g)^{p^{h-1}-1}(t)\not\in W$ is equivalent to $W^*\not\le W$, where $W^*:=(1-g)^{p^{h-1}-1}(\F_q)$. By  the decomposition $\F_q= \oplus_{i=1}^nR\cdot \xi_i\eta$, we have $W^*= W_1^*\oplus \cdots\oplus W_l^*$, where $W_i^*:=R_{p^{h-1}-1}\cdot \xi_i\eta$.
Each component $W_i^*$ has dimension $p^e-p^{h-1}+1$ over $\F_p$.
	
If $h\ge p^{h-1}$, take $W$ to be the direct sum of $R\cdot \xi_2\eta\oplus \cdots\oplus R\cdot \xi_l\eta$ and a $R$-submodule of $ R\cdot (\xi_1\eta)$ of dimension $p^el-h-p^{e}(l-1)=p^e-h$. Then $W$ is a $R$-submodule of $\F_q$ of codimension $h$. Since $p^e-h< \dim_{\F_p}(W_1^*)$ and $R\cdot\xi_1\eta$ is a uniserial $R$-module, the component $W_1^*=R_{p^{h-1}-1}\cdot \xi_1\eta$ of $W^*$ is not contained in the chosen $W$. Since $R_{p^{h-1}-1}=(1-g)^{p^{h-1}-1}R$, we deduce that $(1-g)^{p^{h-1}-1}(\xi_1\eta)\not\in W$. Therefore, $(W,\,\xi_1\eta)$ satisfies both conditions and is a desired pair.

Conversely, if $h< p^{h-1}$, then $\dim_{\F_p} \left(W \cap (R\cdot\xi_i\eta)\right)$ is at least
\begin{align*}
  \dim_{\F_p} W+\dim_{\F_p}  (R\cdot\xi_i\eta) -\dim_{\F_p} (\F_q)=p^{e}l-h+p^{e}-p^{e}l= p^{e}-h\ge \dim_{\F_p}(W_i^*),
\end{align*}
so $W$ contains $W_i^*$  by the fact that $R\cdot \xi_i\eta$ is uniserial. It follows that $W^*\le W$, and thus there is no pair $(W,\,t)$ with the desired properties. This completes the proof.
\end{proof}

\section{The classification of linear point regular groups of $\cQ^P$}\label{sec_linear}

Let $G$ be a point regular group of the Payne derived quadrangle $\cQ^P$, where $\cQ=W(q)$, $P=\la(1,0,0,0)\ra$. In this section, we consider the case where the group $G$  is linear, i.e.,  $G$ is a subgroup of $\textup{PGL}(4,q)$. In \cite{Bamberg2011PRG}, the authors enumerated all the point regular groups of the quadrangle $\cQ^P$  for $q\le 25$ by Magma \cite{Magma}, so we only consider the case $q\ge 5$ below. The main result of this section is the following theorem, the odd characteristic case of which is also due to \cite{Chen2013} and \cite{DeWinter2014Criterion} independently.
\begin{thm}\label{thm_linearPRG}
Let $G$ be a subgroup of $\textup{PGL}(4,q)$ that acts regularly on the points of the derived quadrangle $\cQ^P$ of $\cQ=W(q)$, $q\geq 5$. Then $G$ is conjugate to one of the groups in Construction \ref{const_S1} or Construction \ref{Const2} below.
\end{thm}

The rest of this section is devoted to the proof of Theorem \ref{thm_linearPRG}. By the analysis in Section \ref{subsec_payneGQ}, we can assume that $G$ is as defined in Theorem \ref{Main} for some functions $T$ and $\theta$ up to conjugacy. Take the same notation as in Theorem \ref{Main}, and let $L,\,M,\,S$ be as in Notation \ref{notation_sigmaLMS}. In this case, $\theta_{a,b,c}\equiv 1$. By (1), (2), (4) and (5) of Corollary \ref{cond1}, the maps $L$ and $M$ are both additive and
\begin{equation}\label{eqn_linear_LM}
T(a,b,0)=L(a)+M(b),\quad a,\,b\in \F_q.
\end{equation}

Also, by (7) it holds for all $a,\,b,\,c \in \F_q$ that
\begin{equation*}
L\left(2bc+c^2L(a)+c^2M(b)\right)+M\left(cL(a)+cM(b)\right)=0.
\end{equation*}
By setting $b=0$ and $a=0$ respectively, we get
\begin{equation}\label{eqLM1}
L(c^2L(a))+M(cL(a))=0,  \textup{ for } a,\,c\in \F_q,
\end{equation}
\begin{equation}\label{eqLM2}
L(2bc+c^2M(b))+M(cM(b))=0, \textup{ for } b,\,c\in\F_q.
\end{equation}

\begin{lemma} \label{LM}
Take notation as above. If $q $ is odd, then $L(x)\equiv 0$ and $M(y)\equiv 0$. If $q$ is even, there exist $\omega,\,\mu\in\F_q$ such that
\begin{equation}\label{eqn_LM_exp}
L(x)=\omega \tr_{\F_q/\F_2}(\mu^2 \omega x),\quad M(y)=\omega \tr_{\F_q/\F_2}(\mu y).
\end{equation}
\end{lemma}
\begin{proof}
Write $q=p^m$ with $p$ prime. Let $L(X):=\sum_{i=0}^{m-1}u_iX^{p^i}$, $M(X):=\sum_{i=0}^{m-1}v_iX^{p^i}$ be the corresponding reduced polynomials for the additive maps $x\mapsto L(x)$ and $y\mapsto M(y)$ respectively. The subscripts of $u_i$'s and $v_j$'s are taken modulo $m$.

First assume that $q$ is odd.  Since $2p^{m-1}\le q-1$, we deduce from  \eqref{eqLM1} that $L(X^2L(a))+M(XL(a))=0$ in $\F_q[X]$ for $a\in\F_q$. After expansion we obtain
$\sum_{i=0}^{m-1}u_i L(a)^{p^i}  X^{2p^i}+\sum_{j=0}^{m-1} v_j L(a)^{p^j} X^{p^j}=0$.
By comparing the coefficients of $X^{2p^i}$,  we get $u_iL(a)^{p^i}=0$  for $0\le i\le m-1$. If there is $a\in\F_q$ such that $L(a)\ne 0$, then we deduce that $u_i=0$ for each $i$, i.e., $L(X)=0$: a contradiction. Hence we must have $L(X)=0$. Then  \eqref{eqLM2} reduces to $M(XM(b))=0$, and the same argument yields $M(X)=0$.

Next assume that $q$ is even. If $L(X)=0$, then we deduce from  \eqref{eqLM2} that $M(X)=0$ in the same way, and we can take  $\omega=\mu=0$ in \eqref{eqn_LM_exp}. So  assume that $L(X)\ne 0$ in the sequel. Since $L(X)\ne 0$, there is at least one nonzero coefficient $u_i$. Take $a\in\F_q$ such that $L(a)\ne 0$. By converting  \eqref{eqLM1} into the reduced polynomial form and comparing the coefficients of $X^{2^{i+1}}$, we get $u_iL(a)^{2^i}+v_{i+1}L(a)^{2^{i+1}}=0$, i.e., $u_i=v_{i+1}L(a)^{2^i}$. This holds for all $a\in\F_q$ such that $L(a)\ne 0$, so $\im(L)$ has only one nonzero element, say $\omega$. By applying Lemma \ref{Imdim1} with $\eta =0$, we deduce that there exists $\mu_A\in\F_q^*$ such that  $L(x)=\omega\tr_{\F_q/\F_2}(\mu_Ax)$, i.e., $L(X)=\sum_{i=0}^{m-1}\omega\mu_A^{2^{i}}X^{2^i}$. It follows that $u_i=\omega\mu_A^{2^{i}}$, $v_{i+1}=\omega^{1-2^i}\mu_A^{2^i}$, and $M(y)=\omega\tr_{\F_q/\F_2}(\omega^{-q/2}\mu_A^{q/2}y)$. By setting $\mu=\omega^{-q/2}\mu_A^{q/2}$, we get the desired expressions for $L$ and $M$ in \eqref{eqn_LM_exp}.
\end{proof}

\begin{lemma}\label{lem_LM0_linear}
If $L(a)=M(b)=0$ for all $a,\,b\in\F_q$, then $T(a,b,c)=S(c)$,  $S$ is additive, and  $G$ arises from Construction \ref{const_S1}.
\end{lemma}
\begin{proof}
We have $T(a,b,c)=S(c)$ by (6) of Corollary \ref{cond1}. By (3) of Corollary \ref{cond1}, $S$ is additive. Therefore, $G$ arises from Construction \ref{const_S1} with $S_1=S$.
\end{proof}

\begin{remark}\label{rem_notimp}
Let $G$ be as in Construction \ref{const_S1} with $\theta_{x,y,z}\equiv 1$ and $T(a,b,c)=S_1(c)$ for an additive function $S_1$. The conditions in Theorem \ref{Main} reduce to $S_1(c+z)=S_1(c)+S_1(z)$ which clearly holds, so $G$ is indeed a point regular group of $\cQ^P$.
\end{remark}

By Lemma \ref{LM}, Lemma \ref{lem_LM0_linear} and Remark \ref{rem_notimp}, Theorem \ref{thm_linearPRG} holds for odd $q$. In the sequel, we consider the case $q=2^m$, $m\ge 3$. By Lemma \ref{LM}, there exist $\omega$ and $\mu$ such that \eqref{eqn_LM_exp} holds. In the case $\omega=0$ or $\mu=0$, we have $L\equiv 0$, $M\equiv 0$, and Lemma \ref{lem_LM0_linear} applies. We thus assume that $\omega$ and $\mu$ are nonzero. Set
\begin{equation}\label{eqn_linearBdef}
\cB(c,z):=S(c+z)+S(c)+S(z),
\end{equation}
which is symmetric in $c,\,z$.
\begin{lemma}\label{lem_linearqevenF}
There is a reduced polynomial $F(X)=\sum_{i=0}^{m-1} f_i X^{2^i}$ with $f_0=0$ and
\begin{equation}\label{FcondF}
\mu f_i=(\mu f_{m-i})^{2^i}\, \text{ for } 1\leq i\leq m-1.
\end{equation}
such that $\cB(c,z)=\omega\tr_{\F_q/\F_2}(\mu c F(z))$, where $F(z)^2=\omega S(z)+S(z)^2$.
\end{lemma}
\begin{proof}
By  (3) and (6)  of Corollary \ref{cond1}, we have
 $S(c)+S(z)=T(czS(z),cS(z),c+z)$ and $T(u,v,w)=T(u+vw,v,0)+S(w)$, so
\begin{align}\label{Striangle}
\cB(c,z)=&S(c+z)+T(czS(z),cS(z),c+z)\notag\\
=&T\left(czS(z)+cS(z)(c+z),cS(z),0\right)\notag\\
=&L(c^2S(z))+M(cS(z))\notag\\
=&\omega\tr_{\F_q/\F_2}\left(\mu^2 c^2(\omega S(z)+S(z)^2)\right).
\end{align}
In the third equality we used \eqref{eqn_linear_LM}, and in the fourth we used \eqref{eqn_LM_exp}. The function $B(c,z)$ is symmetric in $c,z$ by \eqref{eqn_linearBdef} and is additive in $c$ by \eqref{Striangle}, so it is also additive in $z$. Set $F(z):=(\omega S(z)+S(z)^2)^{q/2}$, so that $F(z)^2=\omega S(z)+S(z)^2$ and
\[
\cB(c,z)=\omega\tr_{\F_q/\F_2}\left(\mu^2 c^2F(c)^2\right)=\omega\tr_{\F_q/\F_2}(\mu c F(z))
\]
by the fact $\tr_{\F_q/\F_2}(x^2)=\tr_{\F_q/\F_2}(x)$. From $\cB(c,z_1+z_2)=\cB(c,z_1)+\cB(c,z_2)$ for all $c\in\F_q$ we deduce that $F(z_1+z_2)=F(z_1)+F(z_2)$, i.e., $z\mapsto F(z)$ is additive.  Let $F(X)=\sum_{i=0}^{m-1} f_i X^{2^i}$ be the corresponding reduced polynomial. We have $\cB(c,c)=S(2c)+2S(c)=0$ for $c\in\F_q$.  By Lemma \ref{lem_Tracelinear}, we deduce that $f_0=0$ and the equations in \eqref{FcondF} hold. This completes the proof.
\end{proof}

Take notation as in Lemma \ref{lem_linearqevenF}, and set $H(x):=\sum_{0\le i<j\le m-1}\mu^{2^i}f_{j-i}^{2^i} x^{2^i+2^{j}}$. Its value lies in $F_2$, since
\begin{align*}
		H(x)+H(x)^2&=\sum_{0\le i<j\le m-1} \mu^{2^i}f_{j-i}^{2^i} x^{2^{i}+2^{j}}+ \sum_{1\le i<j\le m} \mu^{2^{i}}f_{j-i}^{2^{i}} x^{2^{i}+2^{j}}\\
		&=  \sum_{1\le j\le m-1} \mu f_{j} x^{1+2^{j}}+ \sum_{1\le i\le m-1} (\mu f_{m-i})^{2^{i}} x^{2^{i}+1}=0,
\end{align*}
where we  used  \eqref{FcondF} to get the last equality.

\begin{lemma}\label{lem_linearqevenS1}
The function $S_1(z):=S(z)+\omega\cdot H(z)$ is linearized. Moreover, if $S_1(X):=\sum_{i=0}^{m-1} s_i X^{2^i}$ is the corresponding reduced polynomial, then
\begin{equation}\label{eqn_sifi}
		s_{i+1}=\sum_{j=1}^i\omega^{-2^j+1}f_{i+1-j}^{2^j}+w^{-2^{i+1}+1}s_0^{2^{i+1}},\quad 0\le i\le m-1,
\end{equation}
where $s_m:=s_0$.
\end{lemma}
\begin{proof}
To show that $S_1$ is additive, i.e., $S_1(c+z)+S_1(c)+S_1(z)=0$, is equivalent to showing that $H(c+z)+H(c)+H(z)=\omega^{-1} \cB(c,z)$, cf. \eqref{eqn_linearBdef}. We compute that  $H(c+z)+H(c)+H(z)$ equals
	\begin{align}\label{eqn_BczEq}
		& \sum_{i<j}\mu^{2^i}f_{j-i}^{2^i}\left( (c+z)^{2^i+2^{j}}+c^{2^i+2^{j}}+z^{2^i+2^{j}}\right)\notag\\
		=&\sum_{i<j}(\mu f_{j-i})^{2^i}c^{2^i}z^{2^j}
		+\omega\cdot\sum_{i<j}(\mu f_{j-i})^{2^i}z^{2^i}c^{2^j}\notag\\
		=&\sum_{i<j}(\mu f_{j-i})^{2^i}c^{2^i}z^{2^j}
		+\omega\cdot\sum_{i<j}(\mu f_{i-j})^{2^j}c^{2^j}z^{2^i}\notag\\
		=&\sum_{i,j}(\mu f_{j-i})^{2^i}c^{2^i}z^{2^j}=\tr_{\F_q/\F_2}(\mu cF(z)).
	\end{align}
Here, the subscripts are taken modulo $m$. In the second equality we used the fact $\mu f_{j-i}=(\mu f_{i-j})^{2^{j-i}}$ in  \eqref{FcondF}, and in the third equality we interchanged the label of $i,\,j$ in the last summation and used the fact $f_0=0$. The claim then follows from Lemma \ref{lem_linearqevenF}.

Let $S_1(X):=\sum_{i=0}^{m-1} s_i X^{2^i}$ be the corresponding reduced polynomial. We have $S(x)=S_1(x)+\omega\cdot H(x)$, i.e.,
\begin{equation}\label{eqSZ}
  S(x)=\sum_{i=0}^{m-1}s_i x^{2^i}+ \omega \cdot \sum_{0\le i<j\le m-1}\mu^{2^i}f_{j-i}^{2^i} x^{2^i+2^{j}}.
\end{equation}
We now consider the relation $F(x)^2=\omega S(x)+S(x)^2$. The left hand side is $\sum_{i=0}^{m-1}f_{i-1}^{2}x^{2^i}$.  The right hand side equals $S_1(x)^2+wS_1(x)=\sum_{i=0}^{m-1}(\omega s_i+s_{i-1}^2) x^{2^i}$, since $H(x)+H(x)^2 =0$. Both expressions have degree not exceeding $q-1$, so they are equal as polynomials. By comparing coefficients, we get $\omega s_{i+1}+s_{i}^2=f_{i}^2$ for $0\le i\le m-1$. Here, the subscripts are taken modulo $m$. The equations in \eqref{eqn_sifi} follow by induction on $i$.
\end{proof}

In Lemma \ref{lem_linearqevenS1}, the case $i=m-1$ in \eqref{eqn_sifi} takes the form
\begin{equation}\label{eqn_fsum}
		\sum_{j=1}^{m-1}\omega^{-2^j+1}f_{m-j}^{2^j}=0.
\end{equation}

\begin{lemma}\label{lem_num}
For fixed nonzero elements $\omega,\,\mu$, the number of $(m+1)$-tuples $(f_0,\cdots,f_{m-1},s_0)$ such that $f_0=0$ and all the conditions in  \eqref{FcondF} and \eqref{eqn_fsum} hold  is $2q^{(m-1)/2}$.
\end{lemma}
\begin{proof}
First, assume that $m$ is odd. By  \eqref{FcondF}, we can  express $f_{(m+1)/2},\cdots,f_{m-1}$ in terms of $f_1,\cdots,f_{(m-1)/2}$ as follows:  $f_{m-i}=\mu^{2^{m-i}-1} f_i^{2^{m-i}}$ for $(m+1)/2\le  i\leq m-1$. Plugging them into  \eqref{eqn_fsum} and dividing both sizes by $\omega^{2}\mu$, we get
\begin{equation*}
\beta+\beta^2=\sum_{i=2}^{(m-1)/2}(\mu^{-2^i}\omega^{-1-2^{i}}f_i+\mu^{-1}\omega^{-1-2^{m-i}}f_{i}^{2^{m-i}})
\end{equation*}
with $\beta=\mu^{-1}\omega^{-1-2^{m-1}}f_{1}^{2^{m-1}}$.  The right hand side has absolute trace $0$, since $\mu^{-2^i}\omega^{-1-2^{i}}f_i=(\mu^{-1}\omega^{-1-2^{m-i}}f_{i}^{2^{m-i}})^{2^i}$ for each $i$. Therefore, by Lemma \ref{lem_ff1}, there exist two solutions in $\beta$ for any chosen tuple $(f_2,\cdots,f_{(m-1)/2})$. The claim now follows in this case.

Next, consider the case $m$ is even. The argument is basically the same, and the only distinction is that in showing that the right hand side has absolute trace $0$ we need the observation that $f_{m/2}\mu$ is in the subfield $\F_{2^{m/2}}$, cf. \eqref{FcondF} with $i=m/2$. This completes the proof.
\end{proof}

By (6) of Corollary \ref{cond1} and \eqref{eqn_linear_LM}, we have
\begin{equation}\label{eqn_linear_T}
T(a,b,c)=T(a+bc,b,0)+S(c)=L(a+bc)+M(b)+S(c).
\end{equation}
It turns out that the conditions that we have derived so far are also sufficient, and we have the following construction.
\begin{construction}\label{Const2}
Suppose that $q=2^m$ with $m>1$, and let $\omega,\,\mu$ be two nonzero elements of $\F_q$. Take any tuple $(f_0,\cdots,f_{m-1},s_0)$ satisfying the conditions in Lemma \ref{lem_num}, and define $s_1,\cdots, s_{m-1}$ by  \eqref{eqn_sifi}.  Set $\theta_{a,b,c}\equiv 1$, and
\[
T(a,b,c)=\omega\tr_{\F_q/\F_2}\left(\mu^2\omega (a+bc)+\mu b\right)+\sum_{i=0}^{m-1}s_i c^{2^i}+\omega \sum_{0\leq i<j\leq m-1} \mu^{2^i}f_{j-i}^{2^i} c^{2^{i}+2^{j}}.
\]
Then the  set $G$ as defined in Theorem \ref{Main} with the prescribed functions $T$ and $\theta$ is a point regular group of $\cQ^P$.
\end{construction}
\begin{proof}
Let $L$, $M$ and $S$ be the corresponding functions as defined in  \eqref{eqn_LM_exp} and \eqref{eqSZ}. Then we can verify that  \eqref{eqn_linear_T} holds. Set $F(x):=\sum_{i=0}^{m-1}f_ix^{2^i}$, $\cB(c,z):=S(c+z)+S(c)+S(z)$. We verify that $\omega s_{i+1}+s_{i}^2=f_{i}^2$ holds for $0\le i\le m-1$ by \eqref{eqn_sifi}, so $F(x)^2=\omega S(x)+S(x)^2$. By the calculations in \eqref{eqn_BczEq}, we obtain $\cB(c,z)=\omega\tr_{\F_q/\F_2}(\mu c F(z))$.

It remains to show that $G$ is a point regular group of $\cQ^P$. By Theorem \ref{Main}, it suffices to verify that $T(a,b,c)+T(x,y,z)=T(u,v,w)$, where $w=c+z$, $v=b+y+cT(x,y,z)$ and $u=a+x+bz+cy+czT(x,y,z)$. We deduce that $(u+vw)+(a+bc)+(x+yz)=c^2T(x,y,z)$. Observe that both $L$ and $M$ are additive. Therefore,
\begin{align*}	
&T(a,b,c)+T(x,y,z)+T(u,v,w)\\
=&L(a+bc+x+yz+u+vw)+M(b+y+v)+S(c)+S(z)+S(c+z)\\
=&L(c^2T(x,y,z))+M(cT(x,y,z))+\cB(c,z)\\
=& \omega \tr\left(  \mu^2  c^2\omega T(x,y,z)+ \mu cT(x,y,z) \right)+\cB(c,z) \\
=& \omega \tr( \mu^2 c^2\omega^2 \beta+ \mu c \omega\beta)+\omega\tr(\mu^2c^2\omega S(z)+\mu c S(z))+\cB(c,z)\\
=&\omega\tr(\mu^2c^2\omega S(z)+\mu^2 c^2 S(z)^2)+\cB(c,z).
\end{align*}
Here, $\tr=\tr_{\F_q/\F_2}$, $\beta=\omega^{-1}(L(x+yz)+M(y))$. The last equality holds because  $\beta\in \F_2$ and $\tr(h^2)=\tr(h)$ for $h\in\F_q$. We conclude that this sum is $0$ by the facts that $F(z)^2=\omega S(z)+S(z)^2$, $\cB(c,z)=\omega\tr(\mu c F(z))$. This completes the proof.
\end{proof}

To summarize, we have now completed the classification for $q$ even in the linear case: in the case $\omega=0$ or $\mu=0$ in \eqref{LM}, the group arises from Construction \ref{const_S1} by Lemma \ref{lem_LM0_linear}; in the case $\omega$ and $\mu$ are nonzero, the group arises from Construction \ref{Const2}. This completes the proof of Theorem \ref{thm_linearPRG}.

\begin{remark}\label{rem_even_linear}
Suppose that $G$ is the point regular group of $\cQ^P$ obtained from either of Construction \ref{const_S1} or Construction \ref{Const2}, and assume that $q$ is even. If $T(a,b,c)\equiv 0$, then $G$ is elementary abelian, so we assume that $T(a,b,c)\not\equiv 0$. It is routine to deduce that $G$ has exponent $4$ and nilpotency class $2$ and its center is $Z(G)=\{\fg_{a,b,0}:\,a,\,b\in\F_q,\ T(a,b,0)=0\}$ in both cases. In Construction \ref{const_S1}, $T(a,b,0)\equiv 0$ and $Z(G)$ has size $q^2$; in Construction \ref{Const2}, $T(a,b,0)=\omega\tr_{\F_q/\F_2}(\mu^2\omega a+\mu b)$ and $Z(G)$ has size $q^2/2$. Therefore, the two constructions yield non-isomorphic groups in the even characteristic case.
\end{remark}

\section{The structure of a nonlinear point regular group of $\cQ^P$}\label{sec_structure}

Suppose that $G$ is a nonlinear point regular  group  of the quadrangle $\cQ^P$ with associated functions $T$ and $\theta$ as described in Theorem \ref{Main}. By Corollary \ref{cor_rABnonlin}, we have either $r_{A,B}>0$ or $r_C>0$. We introduce the following notation.
\begin{Notation}\label{notation_4.1}
Take the same notation as in Notation \ref{notation_sigmaLMS}, and assume that $r_{A,B}>0$ or $r_C>0$. Take elements $g_A$, $g_B$, $g_C$ of $\Aut(\F_q)$ of order $p^{r_A}$, $p^{r_B}$, $p^{r_C}$ respectively, and let $t_A$, $t_B$, $t_C$  be the corresponding elements of $\F_q$ such that
\[
\theta_{t_A,0,0}=g_A,\;\theta_{0,t_B,0}=g_B,\;\sigma_{t_C}=g_C.
\]
Also, define $\cK_0^*:=\{z\in \F_q:\, \sigma_z^{p^{r_{A,B}}}=1\}$.
\end{Notation}

\subsection{The Frobenius part of the group $G$}\label{sec_Frob}

\begin{lemma}\label{lem_KAKB}
Take notation as in Notation \ref{notation_4.1}, and set $K_A:=\{a\in\F_q:\,\theta_{a,0,0}=1\}$, $K_B:=\{b\in\F_q:\,\theta_{0,b,0}=1\}$. Then
\begin{enumerate}
\item[(1)] $G_{A,K}:=\{\theta_{a,0,0}:\,a\in K_{A}\}$ is a normal subgroup of $G_A$ of index $p^{r_A}$, and
    \begin{equation}\label{eqn_GAGAK}
      G_A=\la G_{A,K},\,\fg_{t_A,0,0}\ra=\cup_{i=0}^{p^{r_A}-1}G_{A,K}\circ \fg_{t_A,0,0}^i.
    \end{equation}
\item[(2)] $K_A$ is a $g_A$-invariant $\F_p$-subspace of codimension $r_A$ in $\F_q$, $L(a)^{g_A}=L(a^{g_A})$ for $a\in K_A$, and $L$ is additive on $K_A$.
\end{enumerate}
The same conclusions also hold for $G_B$ after we replace $L$ by $M$, $\fg_{t_A,0,0}$ by $\fg_{0,t_B,0}$, and the $A$'s  in the subscripts by $B$'s.
\end{lemma}
\begin{proof}
Since the arguments for $G_A$ and $G_B$ are the same, we only give the proof for  $G_A$ below. In the case $r_A=0$, we have $K_A=\F_q$ and the map $L$ is additive  by (1) of Corollary \ref{cond1}. The claims are trivial in this case, so  we assume that $r_A\ge 1$ in the sequel. \\

\noindent(1). Recall that the group homomorphism $\psi:\,G\rightarrow \Aut(\F_q)$ maps $\fg_{a,b,c}$ to its Frobenius part $\theta_{a,b,c}$. The set $G_{A,K}$ is the kernel of $\psi_A:=\psi|_{G_A}$, the restriction of $\psi$ to the subgroup $G_A$, so it is normal in $G_A$. By the choice of $t_A$ in Notation \ref{notation_4.1}, we have $\im(\psi_A)=\la\psi(\fg_{t_A,0,0})\ra$. The leads to the desired coset decomposition in \eqref{eqn_GAGAK}. \\

\noindent(2). By (1) of Corollary \ref{cond1}, we have $L(a)^{\theta_{x,0,0}}+L(x)=L(a^{\theta_{x,0,0}}+x)$, and
$\theta_{a,0,0}\theta_{x,0,0}=\theta_{u,0,0}$ with $u=a^{\theta_{x,0,0}}+x$  for $a,\,x \in\F_q$.
If $a,x\in K_A$, then $\theta_{a,0,0}=\theta_{x,0,0}=1$, and so $L(a)+L(x)=L(a+x)$,  $\theta_{a+x,0,0}=1$. It follows that $a+x\in K_A$. We conclude that $K_A$ is closed under addition, i.e., it is an $\F_p$-subspace of $\F_q$, and $L$ is additive on $K_A$.

We next explore the fact that $G_{A,K}$ is a normal subgroup of $G_A$. For $a\in K_A$, we consider the element $\fg_{t_A,0,0}^{-1} \circ \fg_{a,0,0} \circ \fg_{t_A,0,0}=(\cM_{-t_A,0,0}\cdot \cM_{a,0,0}^{g_A}\cdot\cM_{t_A,0,0},\, 1)$. The  last row of its matrix part  is $(g_A(a),0,0, 1)$, so it equals $\fg_{g_A(a),0,0}$.  Comparing the $(3,2)$-nd entries of their matrix parts, we get $L(a)^{g_A}=L(g_A(a))$. It has a trivial Frobenius part, so $g_A(a)\in K_A$.  It follows that $K_A$ is $g_A$-invariant.

Finally, we have $|K_A|=|G_{A,K}|$ and $[G_A:\,G_{A,K}]=|\im(\psi_A)|=p^{r_A}$.  This yields the claim on the size of $K_A$ and completes the proof of (2).
\end{proof}

\begin{thm}\label{thm_GA_struc1}
Take notation as in Notation \ref{notation_4.1}. If $q$ is odd, then $r_{A,B}\le 1$; if $q$ is even, then $r_{A,B}\le 2$.
\end{thm}
\begin{proof}
We have $r_{A,B}=\max\{r_A,r_B\}$ by Lemma \ref{cor_rABnonlin}, so the claim is equivalent to that both $r_A$ and $r_B$ are upper bounded by $1$ or $2$ according as $q$ is odd or even.
Since the arguments for $G_A$ and $G_B$ are the same, we only give the proof for  $G_A$ here. We assume that $r_A\ge 1$ in the sequel. For the ease of notation, we write $r=r_A$ in this proof. For each  $i\ge 0$, define
\begin{equation}\label{eqn_AtiKi}
t_{i}=t_A+g_A(t_A)+\cdots+g_A^{i-1}(t_A),\quad K_{i}:=\{a \in \F_q:\, \theta_{a,0,0}= g_A^i\}.
\end{equation}
Here, we have $t_0=0$, $t_A\in K_{1}$ and $K_{i}=K_{i+p^r}$. Let $K_A$ be as in Lemma \ref{lem_KAKB}; we have $K_A=K_0$. It is a $g_A$-invariant $\F_p$-subspace of $\F_q$ of codimension $r$ by (2) of Lemma \ref{lem_KAKB}.

We examine the coset decomposition in  \eqref{eqn_GAGAK}. For $a\in K_A$ and $i\ge 1$, set $\fg_i:=\fg_{a,0,0}\circ\fg_{t_A,0,0}^i$. We compute that the last row of the matrix part of $\fg_i$ is $(g_A^i(a)+t_i,0,0,1)$ by induction, so it equals $\fg_{g_A^i(a)+t_i,0,0}$.
\begin{enumerate}
\item[(a)] The Frobenius part of $\fg_i$ is $g_A^i$, so $K_A+t_i\subseteq K_i$ for $i\ge 1$; in particular, with $i=p^r$ we have $K_A+t_{p^r}=K_A$ by the fact $g_A^{p^r}=1$. It follows that $t_{p^r}\in K_A$.
\item[(b)] By (1) of Corollary \ref{cond1} with $x=t_A$ and $a\in K_i$, we have $\theta_{g_A(a)+t_A,0,0}=\theta_{a,0,0}\theta_{t_A,0,0}=g_A^{i+1}$. Therefore, $a\in K_i$ implies that $t_A+g_A(a)\in K_{i+1}$, and so $t_A+g_A(K_i)\subseteq K_{i+1}$. It follows that $|K_i|=|t_A+g_A(K_i)|\leq |K_{i+1}|$ for each $i$. Since $K_0=K_{p^r}$, we conclude that all the $K_i$'s have the same sizes.
\end{enumerate}
As a corollary of (a) and (b), we have $K_A+t_i=K_i$. It follows by definition that $K_0,\cdots,K_{p^r-1}$ form a partition of $\F_q$. All the conditions in Lemma \ref{Wt1cond} are now satisfied for the pair $(K_A,\,t_A)$ with $e=h=r$, so we have $r\geq p^{r-1}$. We deduce that $r\le 1$ if $p$ is odd and $r\le 2$ if $p=2$ as desired.
\end{proof}

We shall need the following technical lemma in the next section.
\begin{lemma}\label{lem_LsumLofx}
Take the same notation as in Notation \ref{notation_4.1}. For $i\ge 1$, set $t_{A,i}:=(1+g_A+\cdots+g_A^{i-1})(t_A)$. Then  we have $t_{A,p^{r_A}}\in K_{A}$. Moreover, for $x\in K_A$ it holds that
\begin{equation}\label{eqn_Lofx}
    L(g_A^i(x)+t_{A,i})=g_A^i(L(x))+\sum_{j=0}^{i-1}g_A^j(L(t_A)),\quad i\ge 1.
\end{equation}
If we replace the $A$'s in the subscripts by $B$ and $L$ by $M$, then the claims also hold.
\end{lemma}
\begin{proof}
By (2) of Lemma \ref{lem_KAKB}, we have $L(g_A(x))=g_A(L(x))$ for $x\in K_A$.
We showed in the proof of Theorem \ref{thm_GA_struc1} that $t_{A,p^{r_A}}\in K_{A}$, and $\fg_{x,0,0}\circ\fg_{t_A,0,0}^i=\fg_{g_A^i(x)+t_{A,i},0,0}$ for $i\ge 1$. Inductively, we compute that the $(3,2)$-nd entry of $\fg_{x,0,0}\circ\fg_{t_A,0,0}^i$ equals $g_A^i(L(x))+\sum_{j=0}^{i-1}g_A(L(t_A))$ for  $i\ge 1$. The $(3,2)$-nd entry of $\fg_{g_A^i(x)+t_{A,i},0,0}$ equals $T(g_A^i(x)+t_{A,i},0,0)=L(g_A^i(x)+t_{A,i})$ by the definition of $L$ in Notation \ref{notation_sigmaLMS}. Therefore, \eqref{eqn_Lofx} holds. The $(B,M)$ version is proved in the same way.
\end{proof}

\begin{thm}\label{thm_GA_struc2}
Take notation as in Notation \ref{notation_4.1}, and take $g_2\in\Aut(\F_ q)$ of order $p$. If $r_A\le 1$, then there exists   $\mu_A\in \F_q$  such that $g_2(\mu_A)=\mu_A$ and $\theta_{x,0,0}=g_2^{\tr_{\F_q/\F_p}(\mu_A x)}$. The same also holds for $\theta_{x,0,0}$ after we replace the $A$'s  in the subscripts by $B$'s.
\end{thm}
\begin{proof}
The arguments for $r_A$ and $r_B$ are the same, and we only give the proof for  $r_A$ here. In the case $r_A=0$, we can simply take $\mu_A=0$. Therefore, we assume that $r_A=1$ in the sequel. Since $g_A$ has order $p$, we have $g_A=g_2^d$ for some integer $d$ with $1\le d\le p-1$. Let $K_i$ and $t_i$ be as in \eqref{eqn_AtiKi}, and let $K_A$ be as in Lemma \ref{lem_KAKB}.

We claim that  $K_i=K_A+it_A$ for $0\le i\le p-1$. By Lemma \ref{lem_KAKB}, the $\F_p$-subspace $K_A$ has codimension $r=1$ in $\F_q$. Since $t_A\not\in K_A$, we have a partition $\F_q=\cup_{j\in \F_p} (K_A+j t_A)$. Since  $K_A+g_A(t_A)$ is a coset of $K_A$, there is  $\lambda\in\F_p^*$ such that $K_A+g_A(t_A)=K_A+\lambda t_A$. Since $K_A$ is $g_A$-invariant, we deduce that $K_A+g_A^i(t_A)=K_A+\lambda^i t_A$ by induction. In the quotient space $\F_q/K_A$, we thus have  $\overline{g_A^i(t_A)}=\lambda^i \cdot\overline{t_A}$ for $i\ge 0$. From the fact $t_p=\sum_{i=0}^{p-1}g_A^i(t_A)\in K_A$ in Lemma \ref{lem_LsumLofx}, we have $\sum_{i=0}^{p-1}\lambda^i \cdot \overline{t_A}=0$, i.e., $\sum_{i=0}^{p-1}\lambda^i=0$. We deduce that $\lambda=1$, since otherwise the sum equals $\frac{\lambda^p-1}{\lambda-1}=1$. We thus have established that $\overline{g_A^i(t_A)}=\overline{t_A}$ in $\F_q/K_A$ for $i\ge 0$. It follows that $\overline{t_i}=i\cdot \overline{t_A}$, and the claim follows.

Since $K_A$ has codimension $1$ and $t_A\not\in\F_q$, there exists $\mu\in\F_q^*$ such that $\tr_{\F_q/\F_p}(\mu x)=0$ for $x\in K_A$ and $\tr_{\F_q/\F_p}(\mu t_A)=1$. Then $K_i=K_A+it_A=\{x\in\F_q:\,\tr_{\F_q/\F_p}(\mu x)=i\}$ for $i\ge 0$. By the definitions of the $K_i$'s in \eqref{eqn_AtiKi}, we have $\theta_{x,0,0}=g_A^i=g_A^{\tr_{\F_q/\F_p}(\mu x)}=g_2^{\tr_{\F_q/\F_p}(d\mu x)}$ for each $x\in K_i$ and each $i\ge 0$. Since $K_A$ is $g_A$-invariant, we have $g_A(\mu)=\mu$ by Lemma \ref{lem_mu_inv}. To conclude, the theorem holds with $\mu_A=d\mu$.
\end{proof}

We can now give a good description of the Frobenius part of $G_{A,B}$.
\begin{corollary}\label{cor_theta_Ord}
Take  notation as in Notation \ref{notation_4.1}, and assume that $q$ is odd. Let $g_2$ be an element of $\Aut(\F_q)$ of order $p$. Then there exist $\mu_A$ and $\mu_B$ in $\F_q$  that are $g_2$-invariant  such that $\theta_{a,b,0}=g_2^{\tr_{\F_q/\F_p}(\mu_A a+\mu_B b)}$.
\end{corollary}
\begin{proof}
By Theorem \ref{thm_GA_struc1} and Theorem \ref{thm_GA_struc2}, we have $r_A\le 1$, $r_B\le 1$, and there exist $\mu_A$ and $\mu_B$ that are fixed by $g_2$ such that $\theta_{a,0,0}=g_2^{\tr_{\F_q/\F_p}(\mu_A a)}$, $\theta_{0,b,0}=g_2^{\tr_{\F_q/\F_p}(\mu_B b)}$. By (4) of Corollary \ref{cond1}, we have
\[
\theta_{a,b,0}=\theta_{\theta_{0,b,0}^{-1}(a),0,0}\cdot\theta_{0,b,0}=g_2^{\tr_{\F_q/\F_p}(\mu_A a+\mu_Bb)}.
\]
In the second equality, we used the fact that $\tr_{\F_q/\F_p}(\mu_A \theta_{0,b,0}^{-1}(a))=\tr_{\F_q/\F_p}(\theta_{0,b,0} (\mu_A)a)$ and the fact that $\theta_{0,b,0}$ is in $\la g_2\ra$ and thus fixes $\mu_A$. This completes the proof.
\end{proof}

In Theorem \ref{thm_GA_struc1}, we have obtained an upper bound on $r_{A,B}$. Our next objective is to bound $r_C$. To be specific, we will establish the following result.
\begin{thm}\label{thm_G_struc}
Take notation as in Notation \ref{notation_4.1}, and set
\[
\cK_i^*:=\{z\in \F_q:\,\sigma_z^{p^{r_{A,B}}}=g_C^{ip^{r_{A,B}}}\},\quad i\ge 0.
\]
If $q$ is even, then $r_C\le r_{A,B}+2$; if $q$ is odd, then $r_C\le r_{A,B}+1$, $\cK_0^*$ is a  $g_C$-invariant $\F_p$-subspace  of codimension $s=\max\,\{0,r_C-r_{A,B}\}$ and
\[
\cK_i^*=(1+g_C+\cdots+g_C^{i-1})(t_{C})+\cK_0^*,\quad i\ge 1.
\]
\end{thm}

We shall prove Theorem \ref{thm_G_struc} after a series of lemmas. In the case $r_C\le r_{A,B}$, we have $\cK_i^*=\F_q$ for each $i\ge 0$, and the claims in Theorem \ref{thm_G_struc} are trivial. Therefore, we assume that $r_C\ge r_{A,B}+1$, i.e., $s=\max\,\{0,\,r_C-r_{A,B}\}\ge 1$. Set
\[
H_{0}^*:=\{x \in \F_q:\, x+\cK_0^*\subseteq \cK_0^*\},
\]
which is an $\F_p$-subspace. Since $0\in \cK_0^*$, we have $H_0^* \subseteq \cK_0^*$.
For $i\ge 0$, define
\[
\cK_i:=\{z\in \F_q:\, \sigma_z=g_C^i\},\quad t_i:=(1+g_C+\cdots+g_C^{i-1})(t_C).
\]
Here, $t_0=0$, and  it holds that $t_{i+j}=t_{i}+g_C^{i}(t_{j})$ for $i,j\ge 0$. We have $t_1=t_C\in \cK_1$ by the choice of $t_C$.  For ease, we write $r:=r_{A,B}$.

\begin{lemma}\label{lem_FrobcKis1}
We have $\cK_i^*=t_i+g_C^i(\cK_0^*)$, $t_{p^si}\in \cK_0^*$ and $\cK_{i+p^sk}+g_C^{i+p^sk}(\cK_j^*)= \cK_{i+j}^*$ for any nonnegative integers $i,j,k$.
\end{lemma}
\begin{proof}
Combining (3) and (6) of Corollary \ref{cond1}, we have
\[
\sigma_c\sigma_z=\theta_{a',b',c'}=\theta_{a,b,0}\sigma_{c'},
\]
where $c'=c^{\sigma_z}+z$ and the expressions of $a,b,a',b'$ are irrelevant. Raising both sides   to the $p^r$-th power, we deduce that
\begin{equation}\label{eqn_KiKjs}
g^{i+p^sk}(c)+z \in \cK_{i+j}^*,\, \textup{ for } c\in \cK_j^*, \,z\in \cK_{i+p^sk}.
\end{equation}
We deduce from   \eqref{eqn_KiKjs} that  $\cK_{i+p^sk}+g^{i+p^sk}(\cK_j^*)\subseteq \cK_{i+j}^*$; in particular, $|\cK_j^*| \le |\cK_{i+j}^*|$. This holds for all nonnegative integer $i,\,j,\,k$'s. Since $\cK_i^*=\cK_{i+p^s}^*$ and the $\cK_i^*$'s form a partition of $\F_q$, we deduce that all the $\cK_i^*$'s have the same size $q/p^s$ as $\cK_0^*$. It follows that the equality holds in $\cK_{i+p^sk}+g^{i+p^sk}(\cK_j^*)\subseteq \cK_{i+j}^*$ by comparing sizes.

By taking $i=1$, $k=0$, $z=t_C$ in  \eqref{eqn_KiKjs}, we obtain $g_C(c)+t_C\in \cK_{j+1}^*$ for $c\in\cK_j^*$. It follows that $g_C(\cK_j^*)+t_C=\cK_{j+1}^*$ by comparing sizes. Inductively, we obtain $\cK_i^*=t_i+g_C^i(\cK_0^*)$ for  $i\ge 0$. In the case $i=p^sk$, we deduce from $0\in \cK_0^*$  and $\cK_{p^sk}^*=\cK_0^*$ that $t_{p^sk}\in \cK_0^*$. This completes the proof.
\end{proof}

\begin{lemma}\label{lem_FrobcKis2}
We have $\cK_{p^sk}-t_{p^sk}\subseteq g^j(H_0^*)$ for $j,k\ge 0$.
\end{lemma}
\begin{proof}
By Lemma \ref{lem_FrobcKis1}, we have $\cK_{p^sk}+g^{p^sk}(\cK_j^*)= \cK_{j}^*$ and  $\cK_j^*=g^j(\cK_0^*)+t_j$. After canceling $\cK_j^*$ we obtain
\begin{equation}\label{eqn_KKK}
\cK_{p^sk}+g^{j}(g^{p^sk}(\cK_0^*))+g^{p^sk}(t_j)= g^j(\cK_0^*)+t_j.
\end{equation}
It holds that  $-g^j(t_{p^sk})+g^{p^sk}(t_j)-t_j=-t_{p^sk}$, i.e.,
\[
-\sum_{l=j}^{p^sk+j-1}g^l(t_1)+\sum_{l=p^sk}^{ p^sk+j-1}g^{l}(t_1)-\sum_{l=0}^{j-1}g^l(t_1)=-\sum_{l=0}^ {p^sk-1}g^l(t_1),
\]
which is clear by comparing indices. It holds that $\cK_{p^sk}^*=\cK_0^*$ by the definition of $\cK_i^*$, so we have $g^{p^sk}(\cK_0^*)=\cK_0^*-t_{p^sk}$ by Lemma \ref{lem_FrobcKis1}. We thus deduce from \eqref{eqn_KKK} that $\cK_{p^sk}+g^j(\cK_0^*)-g^j(t_{p^sk})+g^{p^sk}(t_j)-t_j=g^j(\cK_0^*)$, i.e.,  $\cK_{p^sk}-t_{p^sk}+g^{j}(\cK_0^*)=g^j(\cK_0^*)$. The claim now follows from the definition of $H_0^*$.
\end{proof}

\noindent\textit{Proof of Theorem \ref{thm_G_struc}}. We continue with the arguments so far.
Set $W:=\cap_{i=0}^{p^{r+s}-1}g_C^i(H_0^*).$ Since $H_0^*$ is a subspace and $g_C$ has order $p^{r+s}$, $W$ is a $g_C$-invariant subspace of $\F_q$. By Lemma \ref{lem_FrobcKis2}, we have $\cK_{p^sk}\subseteq W+t_{p^sk}$ for each $k\ge 0$.  From the fact $0\in\cK_0=\cK_{p^{r+s}}\subseteq W+t_{p^{r+s}}$, we deduce that $t_{p^{r+s}}\in W$. Since each $t_{p^si}$ is in $\cK_0^*$ by Lemma \ref{lem_FrobcKis1} and $W\subseteq H_0^*\subseteq \cK_0^*$, we have
\begin{equation}\label{eqn_cK0_dec}
\cK_0^*=\cup_{i=0}^{p^r-1}\cK_{p^si}\subseteq \cup_{i=0}^{p^r-1}(W+t_{p^si})\subseteq \cK_0^*.
\end{equation}
Therefore, each containment becomes equality in the above equation.

Let $d$ be the smallest positive integer such that $t_{p^sd}\in W$. We just showed that $t_{p^{r+s}}\in W$, so $d\le p^r$. By the $g_C$-invariance of $W$ and the fact  $t_{p^sd(i+1)}=t_{p^sdi}+g^{p^sdi}(t_{p^sd})$, we deduce that $t_{p^{s}di}\in W$ for $i\ge 1$ by induction. Since $t_{p^s(j+di)}=t_{p^sj}+g_C^{p^sj}(t_{p^sdi})$, we deduce that $W+t_{p^s(j+di)}=W+t_{p^sj}$ for $i,j\ge 0$. Hence \eqref{eqn_cK0_dec} yields \begin{equation}\label{ckWrelat}
\cK_0^*=\cup_{i=0}^{d-1}(W+t_{p^si}).
\end{equation}

We claim that \eqref{ckWrelat} is a partition of $\cK_0^*$. If $W+t_{p^si}=W+t_{p^sj}$ with $0\le i<j\le d-1$, then $t_{p^sj}-t_{p^si}=g_C^{p^si}(t_{p^s(j-i)})\in W$, and so $t_{p^s(j-i)}\in W$ by the $g_C$-invariance of $W$, which contradicts the minimality of $d$. This proves the claim.

We plug \eqref{ckWrelat} into $\cK_i^*=g^i(\cK_0^*)+t_i$ and use the fact $t_{i+p^sj}=t_i+g^i(t_{p^sj})$ to obtain
$\cK_i^*=\cup_{j=0}^{d-1}(W+t_{i+p^sj})$, $0\le i\le p^s-1$.
Therefore, from the partition $\F_q=\cup_{i=0}^{p^s-1}\cK_i^*$ we get a refined partition
$\F_q =\cup_{i=0}^{p^s-1}\cup_{j=0}^{d-1}(W+t_{i+p^sj})$ of $\F_q$ into $p^{s}d$ distinct cosets of the subspace $W$. It follows that $d=p^{d_0}$ for some nonnegative integer $d_0$ by considering the divisibility. We have $0\le d_0\le r$ by the fact $d\le p^r$.  The pair $(W,\,t_C)$ satisfies all the conditions in Lemma \ref{Wt1cond} with $h=d_0+s$ and $e=r_C=r+s$, so  $p^{d_0+s-1}\leq d_0+s$. It follows that $d_0+s\le 1$ if $q$ is odd, $d_0+s\le 2$ if $q$ is even.

In the case $q$ is odd, we must have $d_0=0$, $s=1$ by the assumption $s\ge 1$ and the fact $d_0+s\le 1$. It follows that $\cK_0^*=W$ by  \eqref{ckWrelat}, and so $\cK_0^*$ is $g_C$-invariant. By Lemma \ref{lem_FrobcKis1}, we have $\cK_i^*=t_i+\cK_0^*$ for $i\ge 0$. This completes the proof of Theorem \ref{thm_G_struc}. \qed

\begin{corollary}\label{cor_subGcK}
Take notation as above, and assume that $q$ is odd. The set
\begin{equation}\label{eqn_GcK}
 G_{\cK_0^*}:=\{\fg_{a,b,c}:\, a,\,b\in\F_q,\,c\in\cK_0^*\}
\end{equation}
is a normal subgroup of $G$ of index $p^s$ and $G=\cup_{i=0}^{p^{s}-1}G_{\cK_0^*}\circ \fg_{0,0,t_C}^i$. Here, $\cK_0^*:=\{z\in \F_q:\, \sigma_z^{p^{r_{A,B}}}=1\}$, and $s=\max\{0,r_C-r_{A,B}\}$.
\end{corollary}
\begin{proof}
If $s=0$, then $\cK_0^*=\F_q$ and the claims are trivial. We assume that $s>0$ for the rest of this proof. In particular, $r_C=s+r_{A,B}$. Write $g=g_C$, $r=r_{A,B}$ for short. Define the group homomorphism $\psi_r:\,G\rightarrow \Aut(\F_q)$, $\fg_{a,b,c}\mapsto\theta_{a,b,c}^{p^r}$.

We claim that  $G_{\cK_0^*}$ lies in $\ker(\psi_r)$. By the definition of $r=r_{A,B}$ in Notation \ref{notation_sigmaLMS}, $\la \theta_{a,b,0}:\,a,\,b\in\F_q\ra$ has order $p^{r}$. By (6) of  Corollary \ref{cond1}, $\theta_{a,b,c}=\theta_{a',b',0}\sigma_{c}$ for some elements $a'$ and $b'$. Its $p^r$-th power is $1$ if $c\in\cK_0^*$, and the claim follows.

By Theorem \ref{thm_G_struc}, we deduce that the subgroup $G_{\cK_0^*}$ has size $q^2\cdot|\cK_0^*|=q^3/p^s$. Hence, $|\ker(\psi_r)|\ge q^3/p^s$. On the other hand, $\theta_{0,0,t_C}^{p^r}=g^{p^r}$ has order $p^s$, so $|\im(\psi_r)|\ge p^s$. Since $|G|=|\ker(\psi_r)|\cdot|\im(\psi_r)|$, we conclude that $G_{\cK_0^*}=\ker(\psi_r)$.
Since $\psi_r(\fg_{0,0,t_C})$ generates $\im(\psi_r)$, all the claims now follow.
\end{proof}

\subsection{The matrix part of $G$ in the odd characteristic case}\label{sec_mat}

In this subsection, we take the notation as introduced in Notation \ref{notation_4.1} and assume that $q$ is odd. By Theorem \ref{thm_GA_struc1} and Theorem \ref{thm_G_struc}, we have $r_{A,B}=\max\{r_A,\,r_B\}\le 1$ and $s=\max\,\{0,r_C-r_{A,B}\}\le 1$.
Let $g_2$ be an element of $\Aut(\F_q)$ of order $p$, and define
\begin{equation*} 
K_A:=\{a\in\F_q:\,\theta_{a,0,0}=1\},\quad K_B:=\{b\in\F_q:\,\theta_{0,b,0}=1\}.
\end{equation*}
We now collect some known facts.
\begin{enumerate}
\item[(F1)]By Corollary \ref{cor_theta_Ord}, $\theta_{a,b,0}=g_2^{\tr_{\F_q/\F_p}(\mu_A a+\mu_B b)}$ for $a,\,b\in\F_q$, where both $\mu_A$ and $\mu_B$ are $g_2$-invariant. In particular, we have $K_A=\{x\in\F_q:\,\tr_{\F_q/\F_p}(\mu_Ax)=0\}$. We have $r_A=0$ if $\mu_A=0$ and $r_A=1$ otherwise, cf. Notation \ref{notation_sigmaLMS}. The same is true if the $A$'s in the subscripts are replaced by $B$'s.
\item[(F2)]By (4), (5) of Corollary \ref{cond1}, we have $T(a,b,0)=L(a)+M(b)$ if either $a\in K_A$ or $b\in K_B$. By (1), (2) of Corollary \ref{cond1}, $L$   is additive on the subspace $K_A$  and $M$   is additive on $K_B$.
\item[(F3)] By (7) of Corollary \ref{cond1}, we have
	\begin{eqnarray}
		\label{eqn_thetaabc}
		&\sigma_{c}\ \theta_{a,b,0}=\theta_{a',b',0}\ \sigma_{c'};\\
		\label{eqn_ScTab}
		&S(c)^{\theta_{a,b,0}}+T(a,b,0)=T(a',b',0)^{\sigma_{c'}}+S(c');
	\end{eqnarray}
	where $c'=\theta_{a,b,0}(c)$, $b'=\sigma_{c'}^{-1}\left(b+c^{\theta_{a,b,0}}T(a,b,0)\right)$, and
	\begin{align*}		 a'&=\sigma_{c'}^{-1}\left(a+2bc^{\theta_{a,b,0}}+c^{2\theta_{a,b,0}}T(a,b,0)\right).
	\end{align*}
\end{enumerate}

\begin{lemma}\label{lem_sigmaCfixmuAmuB}
If $q$ is odd, then $\sigma_c$ leaves $\mu_A$ and $\mu_B$ invariant for $c\in\cK_0^*$.
\end{lemma}
\begin{proof}
Recall that $\cK_0^*=\{z\in \F_q:\, \sigma_z^{p^{r_{A,B}}}=1\}$. Since $r_{A,B}\le 1$ by Theorem \ref{thm_GA_struc1}, we have $\sigma_c\in\la g_2\ra$ for $c\in\cK_0^*$. The claim is now a consequence of (F1).
\end{proof}

\begin{lemma}\label{lem_qoddMatmuA0}
If $q$ is odd, then $\mu_A=0$, $r_A=0$, $K_A=\F_q$.
\end{lemma}
\begin{proof}
Take $a=0$, $b\in K_B$ and $c\in\cK_0^*$. We have $\theta_{0,b,0}=1$ and $c'=\theta_{0,b,0}(c)=c$, since $b$ is in $K_B$. By canceling out $\sigma_c=\sigma_{c'}$ and comparing exponents in \eqref{eqn_thetaabc}, we obtain
\begin{equation}\label{eqn_tr_KAKB}
\tr_{\F_q/\F_p}\left(\mu_A(2bc+c^2M(b))+\mu_B cM(b))\right)=0.
\end{equation}
Here, we used Lemma \ref{lem_sigmaCfixmuAmuB} and Lemma \ref{lem_trace_int}. Taking the difference of  \eqref{eqn_tr_KAKB} for $c=c_1,\,c_2\in\cK_0^*$, we obtain
\begin{equation}\label{eqn_uv_tr}
\tr_{\F_q/\F_p}\left((2\mu_A b+\mu_BM(b))\cdot v\right)=-2\tr_{\F_q/\F_p}\left(\mu_A M(b)uv\right).
\end{equation}
where $u=\frac{1}{2}(c_1+c_2)$, $v=c_1-c_2$.  We observe that $(u,v)$ ranges over $\cK_0^*\times\cK_0^*$ as $c_1,\,c_2$ vary in $\cK_0^*$. Take the difference of both sides of \eqref{eqn_uv_tr} for $u=u_1,\,u_2\in\cK_0^*$, and we deduce that $\tr_{\F_q/\F_p}\left(\mu_A M(b)vv'\right)=0$ for $v$ and $v'=u_1-u_2\in \cK_0^*$. Since  $\{vv':\,v,v'\in \cK_0^*\}$ spans $\F_q$ over $\F_p$ by Lemma \ref{lem_AB_prod}, we deduce that $\mu_A M(b)=0$ for $b\in K_B$.

Suppose that $\mu_A\ne 0$. Then we deduce from $\mu_A M(b)=0$ that $M(b)=0$ for $b\in K_B$. Hence \eqref{eqn_uv_tr} reduces to $\tr_{\F_q/\F_p}(2\mu_A bv)=0$  for $b\in K_B$, $v\in \cK_0^*$. It follows that the subspace $\{x\in\F_q:\,\tr_{\F_q/\F_p}(xv)=0\textup{ for all }v\in \cK_0^*\}$ contains $\la 2\mu_Ab:\,b\in K_B\ra_{\F_p}$.
Since $\cK_0^*$ has codimension $s$ in $\F_q$, the former subspace has dimension $s$. The latter subspace has size $|K_B|=q/p$, so we have $q/p\le p^s$. This is impossible, since $q\ge p^p$ and $s\le 1$.  To conclude, we have $\mu_A=0$. It follows that $K_A=\F_q$ and $r_A=0$, cf. (F1).
\end{proof}

\begin{lemma}\label{lem_qoddMatM2}
If $q$ is odd and $\mu_B\ne 0$, then  $g_C(\mu_B)=\mu_B$ and $M(b)=0$ for $b\in K_B$.
\end{lemma}
\begin{proof}
Assume that $\mu_B\ne 0$. Then $r_B=1$ by (F1). We have $r_{A,B}=\max\{r_A,r_B\}=1$ by Corollary \ref{cor_rABnonlin}. Take $a=0$, $b\in K_B$ and $c\in\F_q$; we have $\theta_{0,b,0}=1$. We cancel out $\sigma_c=\sigma_{c'}$ in \eqref{eqn_thetaabc} and compare exponents to obtain
\begin{equation}\label{eqn_muBne0}
\tr_{\F_q/\F_p}\left(\mu_B^{\sigma_c}(b+cM(b))\right)=0,\quad b\in K_B,\,c\in\F_q.
\end{equation}
Here, we again used Lemma \ref{lem_trace_int}.

First consider the case $\cK_0^*=\F_q$. In this case, we have $\mu_B^{\sigma_{c}}=\mu_B$ for all $c\in\F_q$ by Lemma \ref{lem_sigmaCfixmuAmuB}. In particular, $g_C=\sigma_{t_C}$ fixes $\mu_B$. The equation \eqref{eqn_muBne0} reduces to $\tr_{\F_q/\F_p}(\mu_B cM(b))=0$ for $b\in K_B$, $c\in\F_q$.  We thus deduce that $M(b)=0$ for $b\in K_B$ as desired. This settles the case $\cK_0^*=\F_q$.

We assume that $\cK_0^*\ne\F_q$ in the sequel. In this case, $s=\max\{0,r_C-r_{A,B}\}=1$ and so $r_C=2$. We have $o(g_C)=p^2$ and $\la g_2\ra=\la g_C^{p}\ra$. By Theorem \ref{thm_G_struc}, $\cK_0^*$ is a $g_C$-invariant $\F_p$-subspace of codimension $s=1$ in $\F_q$, and $\cK_i^*=t_{C,i}+\cK_0^*$ for $i\ge 1$, where $\cK_i^*:=\{z\in \F_q:\,\sigma_z^{p}=g_C^{ip}\}$. For $c\in\cK_i^*$, we have $\sigma_c=g_C^{i+p^sj}$ for some $j\ge 0$ and so $\sigma_c(\mu_B)=g_C^i(\mu_B)$ by the fact that $g_2(\mu_B)=\mu_B$, cf. (F1).

By Lemma \ref{lem_mu_inv}, there exists a $g_C$-invariant element $\mu_C$ such that $\cK_0^*=\{x\in\F_q:\,\tr_{\F_q/\F_p}(\mu_C x)=0\}$.  By taking the difference of \eqref{eqn_muBne0} over $c=c_1,\,c_2\in\cK_i^*$ and letting $c_1$ and $c_2$ vary, we deduce that $\tr_{\F_q/\F_p}(g_C^i(\mu_B)M(b)u)=0$  for all $u\in\cK_0^*$ and $b\in K_B$. By Lemma \ref{lem_tr}, we deduce that $g_C^i(\mu_B)M(b)\in\F_p\cdot \mu_C$ for $b\in K_B$ and $i\ge 0$.

We now show that $M(b)=0$ for $b\in K_B$. Suppose to the contrary that there is $b_0\in K_B$ such that $M(b_0)\ne 0$. Then we deduce from $\F_p\cdot g_C(\mu_B)M(b_0)=\F_p\cdot \mu_B M(b_0)=\F_p\cdot \mu_C$ that $g_C(\mu_B)=\lambda\mu_B$ for some $\lambda\in\F_p^*$. Taking the relative norm to the fixed subfield $\F'$ of $g_C$, we deduce that $\lambda^{[\F_q:\,\F']}=1$. Since $\F_q/\F'$ is a Galois extension, $[\F_q:\,\F']=o(g_C)=p^{r_C}$. It follows that $\lambda=1$, i.e., $g_C(\mu_B)=\mu_B$.  \eqref{eqn_muBne0} then simplifies to $\tr_{\F_q/\F_p}(\mu_BM(b)c)=0$ for $c\in\F_q$, from which we deduce that $M(b)=0$ for $b\in K_B$: a contradiction.

We next show that $g_C(\mu_B)=\mu_B$.  \eqref{eqn_muBne0} now reduces to $\tr_{\F_q/\F_p}\left(\mu_B^{\sigma_c}b\right)=0$ for $b\in K_B$. It follows that $\sigma_c(\mu_B)\in\F_p\cdot\mu_B$ for $c\in\F_q$, so $g_C(\mu_B)\in\F_p\cdot\mu_B$. We deduce that $g_C(\mu_B)=\mu_B$ again by taking the relative norm. This completes the proof.
\end{proof}

\begin{lemma}\label{lem_qoddMatLeq03}
If $q$ is odd, then $L(a)=0$  for  $a\in\F_q$.
\end{lemma}
\begin{proof}
First, consider the case $\mu_B\ne 0$. We have  $\mu_A=0$ and $K_A=\F_q$ by Lemma \ref{lem_qoddMatmuA0}, $g_C(\mu_B)=\mu_B$ and $M(b)=0$ for $b\in K_B$ by Lemma \ref{lem_qoddMatM2}. In particular, $\sigma_c(\mu_B)=\mu_B$ for $c\in\F_q$ by the definition of $g_C$, cf. Notation \ref{notation_4.1}. Taking $a\in \F_q$, $b=0$ and $c\in\F_q$ in \eqref{eqn_thetaabc} and comparing exponents, we obtain $\tr_{\F_q/\F_p}(\mu_BcL(a))=0$. This holds for all $c\in\F_q$, so $L(a)=0$ as desired. This settles the case $\mu_B\ne 0$.

Second, consider the case $\mu_B=0$. In this case, $\theta_{a,b,0}\equiv 1$ and $r_{A,B}=0$. Since $G$ is nonlinear, we have $r_C>0$ by Corollary \ref{cor_rABnonlin}. By Theorem \ref{thm_G_struc}, we have $s=\max\{r_C,0\}\le1$, so $s=r_C=1$. By the same theorem, $\cK_0^*=\{z\in \F_q:\, \sigma_z=1\}$ is a subspace of codimension $1$ of $\F_q$. We have $K_A=K_B=\F_q$ by (F1),  $T(a,b,0)=L(a)+M(b)$ for $a,b\in\F_q$ and both $L$ and $M$ are additive over $\F_q$ by (F2). Suppose to the contrary that $L(a_0)\ne 0$ for some $a_0\in\F_q$. Take $a=a_0$, $b=0$, $c\in\cK_0^*$ in  \eqref{eqn_ScTab}, and we obtain $L(c^2L(a_0))=-M(cL(a_0))$. Taking the difference over $c=c_1,\,c_2\in \cK_0^*$ and letting $c_1,c_2$ vary,  we deduce that $L(uvL(a_0))=-M(vL(a_0))$ for all $u,v\in \cK_0^*$. The right hand side is independent of $u$, so both sides equals $L(0\cdot vL(a))=0$. Therefore, $L(uvL(a_0))=0$ for all $u,v\in\cK_0^*$. By Lemma \ref{lem_AB_prod}, $\la uv:\,u,v\in\cK_0^*\ra_{\F_p}=\F_q$, so $L$ is constantly zero on $\F_q$: a contradiction. This completes the proof.
\end{proof}

\begin{lemma}\label{lem_qoddMatM4}
If $q$ is odd and $\mu_B=0$, we have $M(b)=0$ for $b\in\F_q$.
\end{lemma}
\begin{proof}
Assume that $\mu_B=0$ and $M$ is not constantly zero. In this case, $r_{A,B}=0$. As in the second case of the proof of Lemma \ref{lem_qoddMatLeq03}, we must have $s=r_C=1$ by the nonlinearity of $G$. By Lemma \ref{lem_qoddMatLeq03}, $L(a)=0$ for all $a\in\F_q$. Also, $K_A=K_B=\F_q$ by (F1),  $T(a,b,0)=M(b)$ for $a,b\in\F_q$ and $M$ is additive over $\F_q$ by (F2). Take $a=0$, $b\in\F_q$ and $c\in\F_q$ in  \eqref{eqn_ScTab}, and we obtain
\begin{equation}\label{eqn_Mb}
M(\sigma_c^{-1}(b))^{\sigma_c}-M(b)+M(\sigma_c^{-1}(cM(b)))^{\sigma_c}=0.
\end{equation}
By taking difference over $c=c_1,\,c_2\in\cK_i^*=\{z\in\F_q:\,\sigma_z=g_C^i\}$ and letting $c_1,c_2$ vary, we deduce that $M(g_C^{-i}(v M(b)))=0$ for $v\in \cK_0^*$ and $i\ge 0$.
Therefore, $\cK_0^*\cdot g_C^i(M(b))\subseteq \ker(M)$ for $i\ge 0$. Since $\cK_0^*$ has size $q/p$ by Theorem \ref{thm_G_struc} and $\ker(M)\ne\F_q$, we see that $\ker(M)=\cK_0^*\cdot g_C^i(M(b))$ so long as $M(b)\ne 0$ by comparing sizes. It follows that $\im(M)$ has dimension $1$, say, $\im(M)=\F_p\cdot\omega$ with $\omega\ne 0$. Then $\ker(M)=\cK_0^*\cdot g_C^i(\omega)$ for $i\ge 0$. Since $\cK_0^*$ is $g_C$-invariant, we deduce that $\ker(M)=\cK_0^*\cdot\omega$ is $g_C$-invariant.

Since $\cK_0^*$ is $g_C$-invariant and has codimension $1$, there exists a $g_C$-invariant element $\mu_C$ such that $\cK_0^*=\{x\in\F_q:\,\tr_{\F_q/\F_p}(\mu_C x)=0\}$ by Lemma \ref{lem_mu_inv}. Hence $\ker(M)=\{x\in\F_q:\,\tr_{\F_q/\F_p}(\mu_C \omega^{-1}x)=0\}$. By Lemma \ref{lem_mu_inv}, we deduce from the $g_C$-invariance of $\ker(M)$ that $g_C(\mu_C \omega^{-1})=\mu_C \omega^{-1}$. It follows that $g_C(\omega)=\omega$.

To sum up, we have shown that $M$ is $\F_p$-linear over $\F_q$, $\ker(M)=\cK_0^*\cdot \omega$ and $\im(M)=\F_p\cdot \omega$. By applying Lemma \ref{Imdim1} with $K=\F_q$, we see that
there exists  $\eta\in\F_q^*$ such that $M(x)=\omega\tr_{\F_q/\F_p}(\eta x)$. Plugging it into \eqref{eqn_Mb}, we get $\tr_{\F_q/\F_p}(\Delta b)=0$ with $\Delta=\sigma_c(\eta)-\eta+\tr_{\F_q/\F_p}(\sigma_c(\eta)c\omega)\eta$. Since $b$ is arbitrary, we have $\Delta=0$ for $c\in\F_q$. It follows that $\sigma_c(\eta)\eta^{-1}=1-\tr_{\F_q/\F_p}(\sigma_c(\eta)c\omega)$ lies in $\F_p$. We deduce that $\sigma_c(\eta)\eta^{-1}=1$ by taking the relative norm to the fixed subfield of $\sigma_c$. Now $\Delta=0$ reduces to $\tr_{\F_q/\F_p}(\eta c\omega)=0$ for $c\in\F_q$, and so $\eta\omega=0$: a contradiction to the fact $\eta\omega\ne 0$. This completes the proof.
\end{proof}

\begin{lemma}\label{lem_qoddMatMinfo5}
If $q$ is odd and $r_B=1$, then $(1+g_B+\cdots+g_B^{p-1})(M(t_B))=0$, and
$M(y)=(1+g_B+\cdots+g_B^{i-1})(M(t_B))$ if $\theta_{0,y,0}=g_B^i$, $1\le i\le p-1$.
\end{lemma}
\begin{proof}
Suppose that $r_B=1$. In this case, $\mu_B\ne 0$ and $K_B=\{x\in\F_q:\,\tr_{\F_q/\F_p}(\mu_Bx)=0\}$ is $g_B$-invariant by (F1). By Lemma \ref{lem_LsumLofx}, we have $(1+g_B+\cdots+g_B^{p-1})(t_B)\in K_B$. Since $M$ vanishes on $K_B$ by Lemma \ref{lem_qoddMatM2}, we deduce that $(1+g_B+\cdots+g_B^{p-1})(M(t_B))=0$ by the $(B,M)$-version of  \eqref{eqn_Lofx} with $x=0$ and $i=p$.

Take $y\in\F_q$ and assume that $\theta_{0,y,0}=g_B^i$. Then $y\in K_i$, where $K_{i}=\{b\in\F_q:\,\theta_{0,b,0}=g_B^i\}$. In the proof of Theorem \ref{thm_GA_struc1}, we showed that $K_{i}=K_B+(1+g_B+\cdots+g_B^{i-1})(t_B)$. Since $K_B$ is $g_B$-invariant, there exists $x\in K_B$ such that $y=g_B^i(x)+ (1+g_B+\cdots+g_B^{i-1})(t_B)$. By the $(B,M)$-version of  \eqref{eqn_Lofx} in Lemma \ref{lem_LsumLofx}, we have $M(y)=g_B^i(M(x))+\sum_{j=0}^{i-1}g_B^j(M(t_B))$. By Lemma \ref{lem_qoddMatM2}, we have $M(x)=0$. The desired expression of $M(y)$ then follows.
\end{proof}

Combining Lemmas \ref{lem_qoddMatmuA0}-\ref{lem_qoddMatMinfo5} and the facts (F1), (F2), we have the following theorem on the matrix part of the nonlinear point regular group $G$.
\begin{thm}\label{thm_matrix_G}
Take notation as in Notation \ref{notation_4.1}. Suppose that $q$ is odd, and let $g_2$ be an element of order $p$. Then there exists a $g_C$-invariant element $\mu_B$ such that   $\theta_{a,b,0}=g_2^{\tr_{\F_q/\F_p}(\mu_B b)}$. Also, $T(a,b,0)=M(b)$ for $a,b\in\F_q$, and $M$ vanishes on $K_B=\{b:\,\tr_{\F_q/\F_p}(\mu_B b)=0\}$. If we assume further that $\mu_B\ne 0$, then with $\nu_B:=M(t_B)$ we have $(1+g_B+\cdots+g_B^{p-1})(\nu_B)=0$, and
$M(b)=N_B(\theta_{0,b,0})$, where
    \begin{equation}\label{eqn_Nc}
       N_B(g_B^i):=\begin{cases}0,&\textup{ if }i=0,\\(1+g_B+\cdots+g_B^{i-1})(\nu_B),&\textup{ if }1\le i\le p-1.\end{cases}
    \end{equation}
\end{thm}

For the function $N_B$ introduced in Theorem \ref{thm_matrix_G}, it holds that
\begin{equation}\label{eqn_Nc_eq}
      g_B^j(N_B(g_B^i))+N_B(g_B^j)=N_B(g_B^{i+j}),\,\textup{ for }\, 0\le i,\,j\le p-1.
\end{equation}
This is clear by writing out the expressions on both sides.

\subsection{Summary of the structural results for $q$ odd}\label{sec_sum_struc_qodd}

Let $G$ be a point regular group of $\cQ^P$ with associated functions $T$ and $\theta$ as in Theorem \ref{Main}, and take notation as in Notation \ref{notation_4.1}. In this subsection, we summarize the structural results on the group $G$ in the case $q$ is odd.
\begin{lemma}\label{lem_thetasigmac}
For $a,b,c\in\F_q$, we have $\theta_{a,\sigma_c^{-1}(b),0}=\theta_{0,b,0}$, $M(\sigma_c^{-1}(b))=M(b)$.
\end{lemma}
\begin{proof}
By Theorem \ref{thm_matrix_G}, $\theta_{a,b,0}=g_2^{\tr_{\F_q/\F_p}(\mu_B b)}$ with $g_C(\mu_B)=\mu_B$. In particular, we have $\sigma_c(\mu_B)=\mu_B$ for all $c\in\F_q$ by the definition of $g_C$, cf. Notation \ref{notation_4.1}. The first claim then follows from the fact that $\tr_{\F_q/\F_p}(\mu_B \sigma_c^{-1}(b))=\tr_{\F_q/\F_p}( \sigma_c(\mu_B)b)$. The claim on $M$ then follows from the first claim and the fact $M(x)=N_B(\theta_{0,x,0})$, cf. Theorem \ref{thm_matrix_G}.
\end{proof}

\begin{lemma}\label{lem_FiveEqn}
For $a$, $b$, $c$, $x$, $y$, $z$ in $\F_q$,  it holds that
\begin{align}
    \theta_{a,b,c}&=\theta_{0,b,0}\sigma_{c}, \label{eqn_theta6}\\
     T(a,b,c)&=N_B(\theta_{0,b,0})^{\sigma_c}+S(c),\label{eqn_T6}\\
         \sigma_c\sigma_z&=\theta_{0,v,0}\sigma_{w},\label{eqn_sigc3}\\ S(c)^{\sigma_{z}}+S(z)&=N_B(\theta_{0,v,0})^{\sigma_w}+S(w),\label{eqn_Sc3}
\end{align}
where $v=\sigma_{z}(c)S(z)$, $w=\sigma_{z}(c)+z$, and $N_B$ is as in \eqref{eqn_Nc} with $\nu_B=M(t_B)$.
\end{lemma}
\begin{proof}
The first two equations are reformulations of (6) of Corollary \ref{cond1}  by using the results in  Theorem \ref{thm_matrix_G} and Lemma \ref{lem_thetasigmac}. The last two equations then follow from (3) of Corollary \ref{cond1} and the first two equations.
\end{proof}

\begin{lemma}\label{lem_F6}
For $b$, $c$ in $\F_q$ and $c'=\theta_{0,b,0}(c)$, we have
    \begin{align}
       \sigma_{c}&=\sigma_{c'}\theta_{0,c'M(b),0}, \label{eqn_sigc_7}\\
        \left(S(c)-N_B(\sigma_c)\right)^{\theta_{0,b,0}}&=S(c')-N_B(\sigma_{c'}).\label{eqn_ScTab2}
    \end{align}
\end{lemma}
\begin{proof}
By the results in  Theorem \ref{thm_matrix_G} and Lemma \ref{lem_thetasigmac}, we can reformulate the two equations in (7) of Corollary \ref{cond1} as
$\theta_{0,b,0}\sigma_c=\theta_{0,b+c'M(b),0}\sigma_c'$ and $S(c)^{\theta_{0,b,0}}+N_B(\theta_{0,b,0})=N_B(\theta_{0,b+c'M(b),0})^{\sigma_{c'}}+S(c')$. By the expression of $\theta_{0,x,0}$ in Theorem \ref{thm_matrix_G}, the first equation further reduces to \eqref{eqn_sigc_7}. Also, by \eqref{eqn_Nc_eq} we have
\[
N_B(\theta_{0,b,0})+N_B(\sigma_c)^{\theta_{0,b,0}}=N_B(\theta_{0,b+c'M(b),0})^{\sigma_c'}+N_B(\sigma_c').
\]
Then \eqref{eqn_ScTab2} follows by combining the two equations that involve $N_B$ and $S$.
\end{proof}

In the next two sections, we use these results and Theorem \ref{thm_matrix_G} to establish the following classification theorem for nonlinear regular groups of $\cQ^P$ for $q$ odd.
\begin{thm}\label{thm_nonlinear}
Let $G$ be a group that acts regularly on the points of the derived quadrangle $\cQ^P$ of $\cQ=W(q)$, $q$ odd and $q\ge 5$. If $G$ is not contained in $\textup{PGL}(4,q)$, then $G$ is conjugate to one of the groups in Constructions \ref{Conk0=k0}, \ref{Constquadodd} and \ref{Const_rC2} below.
\end{thm}

By the analysis in Section 2.1, we assume without loss of generality that $G$ is a nonlinear point regular  group  of the quadrangle $\cQ^P$ as defined in Theorem \ref{Main} with associated functions $T$ and $\theta$. By Theorem \ref{thm_GA_struc1}, we have $r_{A,B}\le 1$. We will handle the case $r_{A,B}=0$ in Section \ref{sec_nonlinear_req0}, and handle the case $r_{A,B}=1$ in Section \ref{sec_nonlinear_req1}.

\section{Proof of Theorem \ref{thm_nonlinear} for the case $r_{A,B}=0$}\label{sec_nonlinear_req0}

Let $G$ be a nonlinear point regular  group  of the quadrangle $\cQ^P$ with associated functions $T$ and $\theta$ as in Theorem \ref{Main}, and suppose that $q$ is odd. Take notation as in Notation \ref{notation_4.1}, and assume that $r_{A,B}=0$. Since $G$ is nonlinear, we must have $r_C=1$ by Corollary \ref{cor_rABnonlin} and Theorem \ref{thm_G_struc}. In this case, $s=\max\{0,r_C-r_{A,B}\}=1$. Since $\F_q/\F_p$ is a Galois extension, the order of  $g_C$ divides $[\F_q:\,\F_p]$, and so $q=p^{pl}$ for a positive integer $l$.

Take $g\in\Aut(\F_q)$ such that $g(x)=x^{p^l}$ for $x\in\F_q$, and we specify $g_C=g$ in Notation \ref{notation_4.1}. By Theorem \ref{thm_G_struc} and Lemma \ref{lem_mu_inv}, there is $\mu_C\in\F_q$ such that $g(\mu_C)=\mu_C$, $\cK_0^*=\{x\in\F_q:\,\tr_{\F_q/\F_p}(\mu_Cx)=0\}$. Here, $\cK_0^*=\{c\in\F_q:\,\sigma_c=1\}$, cf. Notation \ref{notation_4.1}.
The element $\mu_C$ lies in $\F_{p^l}$, the fixed subfield of $g_C$.  We deduce that $T(a,b,0)=M(b)\equiv 0$ by Theorem \ref{thm_matrix_G}, $S$ is additive on $\cK_0^*$ by  \eqref{eqn_Sc3}, and $T(a,b,c)=S(c)$ by  \eqref{eqn_T6}. Since $S(z)=T(0,0,z)$, cf. Notation \ref{notation_sigmaLMS}, we have
\[
\fg_{0,0,t_C}=(\cM_{0,0,t_C},\,g),\quad\cM_{0,0,t_C}=E(0,0,t_C,S(t_C)).
\]

\begin{lemma}\label{lem_K0s}
We have $\F_{p^l}\subseteq \cK_0^*$.
\end{lemma}
\begin{proof}
Since $q=p^{pl}$, we have $\tr_{\F_{q}/\F_p}(z)=\tr_{\F_{p^l}/\F_p}(\tr_{\F_q/\F_{p^l}}(z))=0$ for $z\in\F_{p^l}$. The claim now follows from  $\cK_0^*=\{x\in\F_q:\,\tr_{\F_q/\F_p}(\mu_Cx)=0\}$  and the fact $\mu_C\in\F_{p^l}$.
\end{proof}

By Corollary \ref{cor_subGcK}, $G_{\cK_0^*}$ is a normal subgroup of index $p$ in $G$, where $G_{\cK_0^*}:=\{\fg_{a,b,c}:\, a,\,b\in\F_q,\,c\in\cK_0^*\}$, cf. \eqref{eqn_GcK}.  We now examine its implications.
\begin{lemma}\label{lem_rAB0rC1norm}
There is a reduced linearized polynomial $S_1(X)\in\F_{p^l}[X]$ such that $S(c)=S_1(c)$ for $c\in\cK_0^*$.
\end{lemma}
\begin{proof}
By the normality of the subgroup $G_{\cK_0^*}$, we have $\fg_{0,0,t_C}^{-1}\circ \fg_{a,b,c}\circ \fg_{0,0,t_C}\in G_{\cK_0^*}$ for $a,\,b\in\F_q$, $c\in\cK_0^*$. By the multiplication rule \eqref{groupmulti} of $G$, it equals $(\cM',\,1)$ with $\cM'=\cM_{0,0,t_C}^{-1}\cdot \cM_{a,b,c}^g\cdot \cM_{0,0,t_C}$. We compute that the $(4,3)$-rd and the $(3,2)$-nd entries of $\cM'$ equal $c^g$ and  $S(c)^g$ respectively by using the calculations in Remark \ref{rem_EMult}. Therefore, $\fg_{0,0,t_C}^{-1}\circ \fg_{a,b,c}\circ \fg_{0,0,t_C}=\fg_{a',b',c^g}$ for some elements $a',b'$, and $T(a',b',c^g)=S(c)^g$.
It follows from the fact $T(x,y,z)=S(z)$ that $S(c^g)=S(c)^g$ for $c\in\cK_0^*$. The claim now follows from Lemma \ref{Lpoly}.
\end{proof}

\begin{lemma}\label{lem_rAB0rC1ind}
We have $\tr_{\F_q/\F_{p^l}}(\nu_C-S_1(t_C))=0$, where $S_1$ is as in Lemma \ref{lem_rAB0rC1norm} and $\nu_C:=S(t_C)$.
\end{lemma}
\begin{proof}
By Corollary \ref{cor_subGcK}, $G_{\cK_0^*}$ has index $p^{r_C}=p$ in $G$, so $\fg_{0,0,t_C}^{p}\in G_{\cK_0^*}$. Similarly, we compute that the $(3,2)$-nd and the $(4,3)$-rd entries of its matrix part equal $\tr_{\F_q/\F_{p^l}}(\nu_C)$ and $\tr_{\F_q/\F_{p^l}}(t_C)$ respectively. As in the proof of Lemma \ref{lem_rAB0rC1norm}, we deduce that $S(\tr_{\F_q/\F_{p^l}}(t_C))=\tr_{\F_q/\F_{p^l}}(\nu_C)$. We have $\tr_{\F_q/\F_{p^l}}(t_C)\in\cK_0^*$ by Lemma \ref{lem_K0s}, so $S(\tr_{\F_q/\F_{p^l}}(t_C))=S_1(\tr_{\F_q/\F_{p^l}}(t_C))$ by Lemma \ref{lem_rAB0rC1norm}. Since $S_1$ is additive and has coefficients in $\F_{p^l}$, it also equals
\[
  \sum_{i=0}^{p-1}S_1(t_C^{p^{il}})=\sum_{i=0}^{p-1}S_1(t_C)^{p^{il}}=\tr_{\F_q/\F_{p^l}}(S_1(t_C)).
\]
Therefore, $S(\tr_{\F_q/\F_{p^l}}(t_C))$ equals both $\tr_{\F_q/\F_{p^l}}(\nu_C)$ and $\tr_{\F_q/\F_{p^l}}(S_1(t_C))$, and the claim follows.
\end{proof}

It turns out that the conditions that we have derived so far are also sufficient. This leads to the following construction.
\begin{construction} \label{Conk0=k0}
Suppose that $q=p^{pl}$ with $p$ an odd prime and $l$ a positive integer, and let $g\in\Aut(\F_q)$ be such that $g(x)=x^{p^l}$. Take $\mu_C\in\F_{p^l}^*$, and define $K:=\{x\in\F_q:\,\tr_{\F_q/\F_p}(\mu_Cx)=0\}$. Take an element $t_C\in\F_q\setminus K$ and a linearized polynomial $S_1(X)\in\F_{p^l}[X]$. Let $\nu_C$ be an element of $\F_q$ such that $\tr_{\F_q/\F_{p^l}}(\nu_C-S_1(t_C))=0$.

Set $\cM_{a,b,c}=E(a,b,c,S_1(c))$ for $a,\,b\in\F_q$, $c\in K$,
and set $\cM_{0,0,t_C}:= E(0,0,t_C,\nu_C)$, where $E$ is as defined in  \eqref{eqn_EMat}.
Then $G_K:=\{\fg_{a,b,c}:\, a,\,b\in\F_q,\,c\in K\}$ is a group of order $q^3/p$, where $\fg_{a,b,c}=(\cM_{a,b,c},1)$. Let $G$ be the group generated by $G_K$ and $\fg_{0,0,t_C}:=(\cM_{0,0,t_C},\,g)$. Then $G$  is a  point regular group of $\cQ^P$.
\end{construction}
\begin{proof}
We give a sketch of the proof. We can use the calculations in Remark \ref{rem_EMult} to verify that $G_{K}$ is closed under multiplication, so that it forms a group of order $q^3/p$. By reversing the arguments in Lemmas \ref{lem_rAB0rC1norm} and \ref{lem_rAB0rC1ind}, we show that $\fg_{0,0,t_C}^p\in G_K$ and
$\fg_{0,0,t_C}^{-1}\circ \fg_{a,b,c}\circ \fg_{0,0,t_C}\in G_K$ for $a,\,b\in\F_q$, $c\in K$, so $G$ is a group of order $q^3$.

The orbit of $\la (0,0,0,1)\ra$ under the action of $G_K$ is $X_1=\{\la (a,b,c,1)\ra:\,a,b\in\F_q,c\in K\ra\}$ of size $q^3/p$, and $\la (0,0,0,1)\ra^{\fg_{0,0,t_C}}=\la (0,0,t_C,1)\ra$ is not in $X_1$. Hence the orbit of $\la (0,0,0,1)\ra$ under $G$ has size larger than $q^3/p$. Since its length also divides $|G|$, we conclude that it has size $q^3$, i.e., $G$ is point regular.
\end{proof}

To summarize, we have established the case $r_{A,B}=0$ of Theorem \ref{thm_nonlinear}.

\section{Proof of Theorem \ref{thm_nonlinear} for the case $r_{A,B}=1$}\label{sec_nonlinear_req1}

Let $G$ be a nonlinear point regular  group  of the quadrangle $\cQ^P$ with associated functions $T$ and $\theta$ as in Theorem \ref{Main}, and suppose that $q$ is odd. Take notation as in Notation \ref{notation_4.1}, and assume that $r_{A,B}=1$. By Theorem \ref{thm_G_struc}, we have $s=\max\{0,r_C-r_{A,B}\}\le 1$, and so $r_C\le 2$. Since $\F_q/\F_p$ is a Galois extension, the orders of $g_B$ and $g_C$ both divide $[\F_q:\,F_p]$. It follows that $q=p^{p^{s+1}l}$ for some integer $l$. Take $g\in\Aut(\F_q)$ such that $g(x)=x^{p^l}$ for $x\in\F_q$. Set $g_1:=g^{p^s}$, which has order $p$. In Notation \ref{notation_4.1} we specify $g_B=g_1$, and in Theorem \ref{thm_matrix_G} we specify $g_2=g_1$, so that $\theta_{a,b,0}=g_1^{\tr_{\F_q/\F_p}(\mu_B b)}$ and $\theta_{0,t_B,0}=g_1$.

By Theorem \ref{thm_G_struc} and Lemma \ref{lem_mu_inv}, there is $\mu_C\in\F_{q}$ such that $g_C(\mu_C)=\mu_C$ and
\begin{equation}\label{eqn_cK0smuC}
\cK_0^*=\{x\in\F_q:\,\tr_{\F_q/\F_p}(\mu_Cx)=0\},
\end{equation}
where $\cK_0^*=\{c\in\F_q:\,\sigma_c^p=1\}$. By Corollary \ref{cor_subGcK}, $G_{\cK_0^*}$ is a normal subgroup of index $p^s$ in $G$, where $G_{\cK_0^*}=\{\fg_{a,b,c}:\, a,\,b\in\F_q,\,c\in\cK_0^*\}$.

\subsection{The structure of the subgroup $G_{\cK_0^*}$}\label{subsec_nonlinearr11}

In this subsection, we derive some general results on $\sigma_c$ and $S(c)$ by exploring the group structure of the subgroup $G_{\cK_0^*}$, and as a byproduct we complete the proof of Theorem \ref{thm_nonlinear} for the case $r_{A,B}=1$, $r_C\le 1$.

\begin{lemma}\label{lem_cK0g1}
The subspace $\cK_0^*$ is $g_1$-invariant.
\end{lemma}
\begin{proof}
By Theorem \ref{thm_G_struc}, we have $s=\max\{0,r_C-r_{A,B}\}\le 1$. If $r_C\le r_{A,B}$, then $\cK_0^*=\F_q$, and the claim is trivial. If $r_C>r_{A,B}$, then we have $r_C=2$, $s=1$ by the assumption $r_{A,B}=1$. The element $g_1$ has order $p$, and $g_C=g$ has order $p^2$. The claim now follows from the fact that $\cK_0^*$ is $g_C$-invariant, cf. Theorem \ref{thm_G_struc}.
\end{proof}

\begin{lemma}\label{lem_cBsym}
For $c\in\F_q$ and $z\in\cK_0^*$, we have $\sigma_{c+z}(\sigma_c\sigma_{z})^{-1}=g_1^{\cB(c,z)}$, where
\begin{equation}\label{eqn_BF}
\cB(c,z):=\tr_{\F_q/\F_p}\left(\mu_Bc(N_B(\sigma_z)-S(z))\right).
\end{equation}
\end{lemma}
\begin{proof}
Fix an element $z\in\cK_0^*$. Then $\sigma_z^p=1$, and so there is $b\in\F_q$ such that $\theta_{0,b,0}=\sigma_z$.
By replacing $c$ with $\sigma_{z}^{-1}(c)$ in  \eqref{eqn_sigc3}, we get
\begin{equation}\label{eqn_sssss}
\sigma_{\sigma_{z}^{-1}(c)}=\sigma_{c+z}\sigma_{z}^{-1}g_1^{\tr_{\F_q/\F_p}(\mu_BcS(z))}.
\end{equation}
By replacing $c$ with $\theta_{0,b,0}^{-1}(c)$ in  \eqref{eqn_sigc_7}, we have
\begin{equation}\label{thetaabcrelat}
\sigma_{\theta_{0,b,0}^{-1}(c)}=\sigma_cg_1^{\tr_{\F_q/\F_p}(\mu_B cM(b))},\quad b,c\in\F_q.
\end{equation}
Since $\theta_{0,b,0}=\sigma_z$, the right hand sides of \eqref{thetaabcrelat} and \eqref{eqn_sssss} are equal. The desired equation $\sigma_{c+z}(\sigma_c\sigma_{z})^{-1}=g_1^{B(c,z)}$ then follows, since $M(b)=N_B(\theta_{0,b,0})=N_B(\sigma_z)$ by Theorem \ref{thm_matrix_G}.
\end{proof}

\begin{lemma}\label{lem_Bdef}
Let $\cB:\,\cK_0^*\times\cK_0^*\rightarrow\F_p$ be as defined in \eqref{eqn_BF}
and  set $Q(x):=\cB(x,x)$. Then $\cB$ is a symmetric bilinear form on $\cK_0^*$, and $Q(g_1(c))=Q(c)$ for $c\in\cK_0^*$. Moreover, there is $\alpha\in\F_q$ such that  $\alpha-g_1(\alpha)+\mu_BM(t_B)\in\F_p\cdot \mu_C$ and
\begin{equation}\label{eqn_sigz_exp}
\sigma_{z}=g_1^{\frac{1}{2}Q(z)+\tr_{\F_q/\F_p}(\alpha z)}\; \textup{ for }\, z\in\cK_0^*.
\end{equation}
\end{lemma}
\begin{proof}
Take $c,z$ in $\cK_0^*$. We have $\sigma_{c+z}(\sigma_c\sigma_{z})^{-1}=g_1^{B(c,z)}$ by Lemma \ref{lem_cBsym}. Its left hand side is symmetric in $c,z$, and the function $\cB(c,z)$ is $\F_p$-linear in $c$, so $\cB$ is a symmetric bilinear form on $\cK_0^*$. It follows that $\cB(c,z)=\frac{1}{2}(Q(c+z)-Q(c)-Q(z))$ for $c,z\in\cK_0^*$, where $Q(x)=\cB(x,x)$. We thus can rewrite $\sigma_{c+z}(\sigma_c\sigma_{z})^{-1}=g_1^{B(c,z)}$  as follows:
\[
\sigma_{c+z}g_1^{-\frac{1}{2}Q(c+z)}=\left(\sigma_c g_1^{-\frac{1}{2}Q(c)}\right)\cdot\left(\sigma_{z} g_1^{-\frac{1}{2}Q(z)}\right),\quad c,\,z\in\cK_0^*.
\]
That is, $z\mapsto \sigma_{z} g_1^{-\frac{1}{2}Q(z)}$ is a group homomorphism from $\cK_0^*$ to $\langle g_1\rangle$. Therefore, there exists an element $\alpha\in\F_q$ such that  $\sigma_{z}=g_1^{\frac{1}{2}Q(z)+\tr_{\F_q/\F_p}(\alpha z)}$ for $z\in\cK_0^*$ by Lemma \ref{lem_ff2}.

For $c\in\cK_0^*$ and $b\in\F_q$, we have $\theta_{0,b,0}^{-1}(c)\in\cK_0^*$ by Lemma \ref{lem_cK0g1}.  We plug the expressions of $\sigma_c$ and $\sigma_{\theta_{0,b,0}^{-1}(c)}$  into  \eqref{thetaabcrelat} and compare the exponents  to get
\[
\frac{1}{2} Q(c^{\theta_{0,b,0}^{-1}})+\tr_{\F_q/\F_p}(\alpha^{\theta_{0,b,0}}c)=\frac{1}{2}Q(c)+\tr_{\F_q/\F_p}(\alpha c)+\tr_{\F_q/\F_p}(\mu_BcM(b)).
\]
Since $\cK_0^*$ is a subspace, we have $\lambda c\in\cK_0^*$ for all $\lambda\in\F_p$. By replacing $c$ with $\lambda c$ ($\lambda \in\F_p$) and comparing the coefficients of $\lambda$ and $\lambda^2$, we deduce that $Q(c^{\theta_{0,b,0}^{-1}})=Q(c)$ and $\tr_{\F_q/\F_p}\left((\alpha^{\theta_{0,b,0}}-\alpha-\mu_BM(b))\cdot c\right)=0$. Take $b=t_B$, and we get the remaining claim in the lemma by Lemma \ref{lem_tr}.
\end{proof}

\begin{lemma}\label{lem_exp_S_sum}
There is a reduced linearized polynomial $S_1(X)$ over $\F_q$ and a map $H:\,\cK_0^*\rightarrow\F_p$ such that $S(z)=S_1(z)+N_B(\sigma_z)+\mu_B^{-1}\mu_CH(z)$ for $z\in\cK_0^*$.
\end{lemma}
\begin{proof}
If $\mu_C\ne 0$,  take  $e\in\F_q\setminus\cK_0^*$, and define $\tilde{\cB}(x+\lambda e,y+\lambda'e):=\cB(x,y)$ for $x,\,y\in\cK_0^*$ and $\lambda,\,\lambda'\in\F_p$. Then $\tilde{\cB}$ is a bilinear form over $\F_q$. Therefore, whether $\mu_C=0$ or not, there is a reduced linearized polynomial $S_1(X)$ over $\F_q$ such that $\cB(x,y)=-\tr_{\F_q/\F_p}(\mu_BxS_1(y))$ for $x,\,y\in\cK_0^*$ by Lemma \ref{lem_bform}. Together with  \eqref{eqn_BF}, we deduce that $\tr_{\F_q/\F_p}(\mu_Bx(N_B(\sigma_y)-S(y)+S_1(y)))=0$  for  $x,\,y\in\cK_0^*$.
It follows that $N_B(\sigma_y)-S(y)+S_1(y)\in\F_p\cdot \mu_B^{-1}\mu_C$ for $y\in\cK_0^*$ by Lemma \ref{lem_tr}. This completes the proof.
\end{proof}

For $c\in\cK_0^*$, we define
\begin{equation}\label{eqn_S2def}
S_2(c):=S(c)-N_B(\sigma_c).
\end{equation}
\eqref{eqn_ScTab2} in Lemma \ref{lem_F6} now takes the form $S_2(c^{\theta_{0,b,0}})=S_2(c)^{\theta_{0,b,0}}$, so
\begin{equation}\label{eqn_S2g1}
S_2(c^{g_1^i})=S_2(c)^{g_1^i},\quad\textup{ for } c\in\cK_0^*,\,i\ge 0.
\end{equation}
The expression of $\cB$ in \eqref{eqn_BF} now takes the form
$\cB(c,z)=-\tr_{\F_q/\F_p}(\mu_BcS_2(z))$ for  $c,\,z\in\cK_0^*$, and so $Q(c)=\cB(c,c)$ takes the form
\begin{equation}\label{eqn_Q_exp_cK0}
Q(c)=-\tr_{\F_q/\F_p}(\mu_BcS_2(c)),\quad c\in\cK_0^*.
\end{equation}

\begin{lemma}\label{lem_non_expthetaT}
For $a,\,b\in\F_q$ and $c\in\cK_0^*$, we have
\begin{equation*}
\begin{split}
\theta_{a,b,c}&=g_1^{\frac{1}{2}Q(c)+\tr_{\F_q/\F_p}(\alpha c+\mu_Bb)}, \\
T(a,b,c)&=S_2(c)+N_B(\theta_{a,b,c}),
\end{split}
\end{equation*}
\end{lemma}
\begin{proof}
Take $a,\,b\in\F_q$ and $c\in\cK_0^*$. We have $\theta_{0,b,0}=g_1^{\tr_{\F_q/\F_p}(\mu_Bb)}$ and $M(b)=N_B(\theta_{0,b,0})$ by Theorem \ref{thm_matrix_G}, where we specified $g_2=g_1$ in the theorem, cf. the beginning of this section. We have $\theta_{a,b,c}=\theta_{0,b,0}\sigma_c$ by \eqref{eqn_theta6}. We plug  the  expression of $\theta_{0,b,0}$ and that of $\sigma_c$ in    \eqref{eqn_sigz_exp} into  it and obtain the expression for $\theta$.

By \eqref{eqn_T6}, we have $T(a,b,c)=N_B(\theta_{0,b,0})^{\sigma_c}+S(c)$. By \eqref{eqn_S2def}, we have $S(c)=S_2(c)+N_B(\sigma_c)$. We compute that
\begin{align*}
T(a,b,c)&=N_B(\theta_{0,b,0})^{\sigma_c}+N_B(\sigma_c)+S_2(c)\\
&=N_B(\theta_{0,b,0}\sigma_c)+S_2(c)=S_2(c)+N_B(\theta_{a,b,c}).
\end{align*}
For the second equality, we used \eqref{eqn_Nc_eq}. This completes the proof.
\end{proof}

\begin{lemma}\label{lem_S2_eqn_new}
The map $x\mapsto S_2(x)$ as defined in  \eqref{eqn_S2def} is additive on $\cK_0^*$.
\end{lemma}
\begin{proof}
Take $c,\,z\in\cK_0^*$. Since $\cK_0^*=\{x\in\F_q:\,\sigma_x^{p}=1\}$, cf. Notation \ref{notation_4.1}, we have $\sigma_z\in\la g_1\ra$. By \eqref{eqn_sigc3}, we have $\sigma_c\sigma_z=\theta_{0,v,0}\sigma_{w}$ for some $v,w$. It follows from  \eqref{eqn_Nc_eq} that
\begin{equation}\label{eqn_notimp}
 N_B(\sigma_c)^{\sigma_z}+N_B(\sigma_z)=N_B(\theta_{0,v,0})^{\sigma_w}+N_B(\sigma_w).
\end{equation}

By \eqref{eqn_S2def}, we have $S(c)=S_2(c)+N_B(\sigma_c)$, $S(z)=S_2(z)+N_B(\sigma_z)$. Plug then into \eqref{eqn_Sc3}, i.e., $S(c)^{\sigma_{z}}+S(z)=N_B(\theta_{0,v,0})^{\sigma_w}+S(w)$, and we obtain $S_2(c)^{\sigma_z}+S_2(z)=S_2(c^{\sigma_z}+z)$ after canceling out the terms involving $N_B$ by using \eqref{eqn_notimp}. The lemma now follows from  \eqref{eqn_S2g1} and the fact that $\cK_0^*$ is $g_1$-invariant, cf. Lemma \ref{lem_cK0g1}.
\end{proof}

We are now ready to complete the classification in the case $r_C\le r_{A,B}=1$. In this case, we have $s=0$, $g=g_1$, $\cK_0^*=\F_q$, $\mu_C=0$. In particular, $\sigma_z=1$ for each $z\in\F_q$. The function $S=S_2$ is additive by \eqref{eqn_S2def} and Lemma \ref{lem_exp_S_sum}. Write $S_2(X)=\sum_{i=0}^{pl-1}s_iX^{p^i}$. Then from  $S_2(c^{g_1})=S_2(c)^{g_1}$ for $c\in\F_q$, cf. \eqref{eqn_S2g1}, we deduce that $g_1(s_i)=s_i$, i.e., $s_i\in\F_{p^{l}}$, for each $i$. The function $\cB(c,z)=-\tr_{\F_q/\F_p}(\mu_BcS_2(z))$ is a symmetric bilinear form over $\F_q$ if and only if $\mu_B s_i- s_{pl-i}^{p^i}\mu_B^{p^i}=0$ for $1\leq i\leq pl-1$ by Lemma \ref{lem_Tracelinear} with $\mu=\mu_B$ and $\eta=0$ there. The expressions of $T,\,\theta$ are as in  Lemma \ref{lem_non_expthetaT}, which involve the parameters $\mu_B$, $\alpha$, $S_2$ and $\nu_B:=M(t_B)$. By Lemma \ref{lem_Bdef}, $\alpha-g_1(\alpha)+\mu_B\nu_B=0$. It turns out that the restrictions on the parameters that we have derived so far are also sufficient, and we have the following construction.
\begin{construction} \label{Constquadodd}
Suppose that $q=p^{pl}$ with $p$ an odd prime and $l$ a positive integer, and let $g_1\in\Aut(\F_q)$ be such that $g_1(x)=x^{p^l}$.
\begin{enumerate}
\item[(i)] Take $\mu_B\in\F_{p^l}^*$.
\item[(ii)] Take a tuple $(s_0,\,s_1,\cdots,s_{pl-1})$ with entries in $\F_{p^l}$ such that $\mu_B s_i- s_{pl-i}^{p^i}\mu_B^{p^i}=0$ for $1\leq i\leq pl-1$.
\item[(iii)] Take  $\alpha$ in $\F_q$ and set $\nu_B:=\mu_B^{-1}(g_1(\alpha)-\alpha)$.
\end{enumerate}
Set $S_2(x):=\sum_{i=0}^{pl-1}s_ix^{p^i}$, and $Q(x):=-\tr_{\F_q/\F_p}(\mu_BxS_2(x))$ for $x\in\F_q$. Define
\begin{equation*}
       N_B(g_1^i):=\begin{cases}0,&\textup{ if }i=0,\\(1+g_1+\cdots+g_1^{i-1})(\nu_B),&\textup{ if }1\le i\le p-1.\end{cases}
\end{equation*}
Let $\theta$ and $T$ be as defined in Lemma \ref{lem_non_expthetaT}. Then the set $G$ with associated functions $T$ and $\theta$ as in Theorem \ref{Main}  is a point regular group  of $\cQ^P$.
\end{construction}
\begin{proof}
The number of tuples satisfying the conditions in (ii) equals $p^{(pl+1)l/2}$ by a similar argument to that in the proof of Lemma \ref{lem_num}. Since $\mu_B$ and the $s_i$'s lie in $\F_{p^l}$ and thus are fixed by $g_1$, we can verify that $S_2(g_1(x))=S_2(x)^{g_1}$, $Q(x^{g_1})=Q(x)$ holds for $x\in\F_q$. We also observe that the values of $\theta_{a,b,c}$ and $T(a,b,c)$ are independent of $a$.

Take $a$, $b$, $c$, $x$, $y$, $z$ in $\F_q$, and define $u$, $v$, $w$ as in Theorem \ref{Main}. To be specific, we have  $w=c^{\theta_2}+z$ and $v=b^{\theta_2}+y+c^{\theta_2}T(x,y,z)$, where $\theta_2=\theta_{x,y,z}$. We sketch how to verify the two conditions $\theta_{a,b,c}\theta_{x,y,z}=\theta_{u,v,w}$, $T(a,b,c)^{\theta_2}+T(x,y,z)=T(u,v,w)$ in Theorem \ref{Main} so that $G$ is a point regular group of $\cQ^P$.

We use the expression of $\theta$ to deduce that $\theta_{a,b,c}\theta_{x,y,z}=\theta_{u,v,w}$ is equivalent to $\tr_{\F_q/\F_p}((\alpha-\alpha^{\theta_2}+\mu_BN_B(\theta_2))c^{\theta_2})=0$ upon expansion and simplification. It needs to hold for all $c\in\F_q$, so we need to show that $g_1^i(\alpha)-\alpha=\mu_BN_B(g_1^i)$ for $1\le i\le p-1$. They follow from $g_1(\alpha)-\alpha=\mu_B\nu_B$ by induction. This establishes the first condition.

Similarly, the condition $T(a,b,c)^{\theta_2}+T(x,y,z)=T(u,v,w)$ takes the form
\[
S_2(c)^{\theta_2}+N_B(\theta_{a,b,c})^{\theta_2}+S_2(x)+N_B(\theta_{2})
=S_2(c^{\theta_2}+z)+N_B(\theta_{u,v,w}).
\]
The terms involving $S_2$ cancel out by the facts that $S_2(g_1(x))=S_2(x)^{g_1}$ and $S_2$ is $\F_p$-linear. The remaining terms involving $N_B$ also cancel out by \eqref{eqn_Nc_eq}. This establishes the second condition and completes the proof.
\end{proof}

To summarize, we have proved the case $r_C\le r_{A,B}=1$ of Theorem \ref{thm_nonlinear}.

\subsection{The normality of the subgroup $G_{\cK_0^*}$}\label{subsec_nonlinearr12}

In this subsection, assume that $r_{A,B}=1$, $r_C=2$, so that $s=\max\{0,r_C-r_{A,B}\}=1$.
By the argument in the beginning of this section, we have $q=p^{p^2l}$ for an integer $l$. We take $g\in\Aut(\F_q)$ such that $g(x)=x^{p^l}$ for $x\in\F_q$. Then  $g$ has order $p^2$, and $g_1:=g^p$ has order $p$. In Theorem \ref{thm_matrix_G}, we take $g_2=g_1$, so that $\theta_{a,b,0}=g_2^{\tr_{\F_q/\F_p}(\mu_B b)}$, where $g_C(\mu_B)=\mu_B$. In Notation \ref{notation_4.1}, we specify $g_C=g$, so that $\sigma_{t_C}=g$; we choose $t_B$ such that $\tr_{\F_q/\F_p}(\mu_B t_B)=1$ , so that $g_B=g_1$. Let $N_B$ be as in \eqref{eqn_Nc} with $\nu_B:=M(t_B)$, so that $M(b)=N_B(\theta_{0,b,0})$ by Theorem \ref{thm_matrix_G}.

We have $\cK_0^*=\{x\in\F_q:\,\tr_{\F_q/\F_p}(\mu_Cx)=0\}$ for some $\mu_C\in\F_{p^l}$ by \eqref{eqn_cK0smuC}, where $\cK_0^*=\{c\in\F_q:\,\sigma_c^p=1\}$.
For $c,z\in\cK_0^*$,  set $S_2(c)=S(c)-N_B(\sigma_c)$, $\cB(c,z)=-\tr_{\F_q/\F_p}(\mu_BcS_2(z))$  as in \eqref{eqn_S2def} and \eqref{eqn_BF}, and set $Q(c):=\cB(c,c)$. By Notation \ref{notation_sigmaLMS}, we have
\[
\fg_{0,0,t_C}=(\cM_{0,0,t_C},g),\, \textup{ with } \cM_{0,0,t_C}=E(0,0,t_C,S(t_C))
\]
with $E$ as in  \eqref{eqn_EMat}. Since $t_C\not\in\cK_0^*$, we have $\mu_C\ne 0$, $\tr_{\F_q/\F_p}(\mu_C t_C)\ne 0$.

In the sequel, we explore the condition that $G_{\cK_0^*}=\{\fg_{a,b,c}:\,a,b\in\F_q,\,c\in\cK_0^*\}$ is a normal subgroup of index $p$ in $G$, cf. Corollary \ref{cor_subGcK}. This splits into two steps: (1) we examine the normality of $G_{\cK_0^*}$, i.e., $\fg_{0,0,t_C}^{-1}\circ\fg_{a,b,c}\circ \fg_{0,0,t_C}\in G_{\cK_0^*}$ for $\fg_{a,b,c}\in G_{\cK_0^*}$; (2) we check the condition $[G:\,G_{\cK_0^*}]=p$, i.e., $\fg_{0,0,t_C}^p\in G_{\cK_0^*}$.\\

We now take the first step. Take $c\in\cK_0^*$. Choose an element $b$ such that
\begin{equation}\label{eqn_trb}
-\tr_{\F_q/\F_p}(\mu_Bb)=\frac{1}{2}Q(c)+\tr_{\F_q/\F_p}(\alpha c).
\end{equation}
Such $b$'s exist by the fact $\mu_B\ne 0$. We have $\theta_{a,b,c}=1$ for any $a\in\F_q$ by Lemma \ref{lem_non_expthetaT}. By the calculations in Remark \ref{rem_EMult}, we compute that
\begin{equation}\label{eqn_normalodd}
\fg_{0,0,t_C}^{-1}\circ\fg_{a,b,c}\circ \fg_{0,0,t_C}=\fg_{a',b',c'}.
\end{equation}
where $a'$ is irrelevant, $b'=b^{g}-t_C S_2(c)^g +c^gS(t_C)$  and $c'=c^g$. Here, $S_2(c)=S(c)-N_B(\sigma_c)$, cf.  \eqref{eqn_S2def}.
We have $c'\in \cK_0^*$, since $\cK_0^*$ is $g$-invariant.  By comparing the Frobenius part and the $(3, 2)$-nd entry of the matrix part of both sides of  \eqref{eqn_normalodd}, we get
\[
\theta_{a',b',c'}=1,\quad T(a,b,c)^g=T(a',b',c').
\]
\begin{lemma}\label{lem_S2Q_ginv}
We have $S_2(c)^g=S_2(c^g)$, $Q(c^g)=Q(c)$ for $c\in\cK_0^*$.
\end{lemma}
\begin{proof}
Take notation as above. By the expression of $T$ in Lemma \ref{lem_non_expthetaT}, we have $T(a,b,c)=S_2(c)$, $T(a',b',c')=S_2(c')$. It follows from $T(a,b,c)^g=T(a',b',c')$ that  $S_2(c)^g=S_2(c^g)$. This proves the first claim.

We have $Q(x)=-\tr_{\F_q/\F_p}(\mu_BxS_2(x))$ for $x\in\cK_0^*$, cf. \eqref{eqn_Q_exp_cK0}, so
\[
Q(c^g)=-\tr_{\F_q/\F_p}(\mu_Bc^gS_2(c)^g)=-\tr_{\F_q/\F_p}(g^{-1}(\mu_B)cS_2(c))=Q(c).
\]
Here, we used the fact that $g(\mu_B)=\mu_B$, cf. Theorem  \ref{thm_matrix_G}. This proves the second claim.
\end{proof}

We use the expression of $\theta$ in Lemma \ref{lem_non_expthetaT} to see that $\theta_{a',b',c'}=1$ is equivalent to
\begin{equation*}
-\tr_{\F_q/\F_p}(\mu_Bb)=\frac{1}{2}Q(c^g)+\tr_{\F_q/\F_p}(\alpha c^g-\mu_Bt_C S_2(c)^g +\mu_BS(t_C) c^g)).
\end{equation*}
Therefore, its right hand side equals that of \eqref{eqn_trb}. This yields
\begin{equation}\label{eqn_within1}
\tr_{\F_q/\F_p}\left(\mu_Bt_CS_2(c^g) \right)=\tr_{\F_q/\F_p}\left((-\alpha^g+\alpha+\mu_BS(t_C))c^g\right)\;\textup{ for }\, c\in\cK_0^*.
\end{equation}

\begin{lemma}\label{lemma_norsufficond1p1}
There is a polynomial $\sum_{i=0}^{p^2l-1}s_iX^{p^i}\in\F_{p^l}[X]$ whose restriction to $\cK_0^*$ is $S_2$. Write this polynomial as $S_2(X)$ by abuse of notation.
\begin{enumerate}
\item[(i)] There is $u\in\F_q$ such that $u-g(u)\in \F_p\cdot\mu_C$,  $(1-g)^2(u)=0$  and
    \begin{equation}\label{eqn_cond_si}
    -\mu_Bs_i+s_{m-i}^{p^i}\mu_B^{p^i}=\mu_C u^{p^i}-u\mu_C^{p^i},\quad    0\leq i\leq p^2l-1.
    \end{equation}
\item[(ii)] For $x,\,y\in\F_q$ and $\tr=\tr_{\F_q/\F_p}$, it holds that
   \begin{equation}\label{eqn_BFSymm}
    \begin{split}
    \tr \left(\mu_B xS_2(y)\right)= \tr \left(\mu_ByS_2(x)\right)
    &+\tr (u x)\cdot\tr (\mu_C y)-\tr (\mu_C x)\cdot\tr(u y).
    \end{split}
  \end{equation}
\end{enumerate}
\end{lemma}
\begin{proof}
By Lemma \ref{lem_S2Q_ginv},  we have $S_2(c)^{g}=S_2(c^{g})$ for $c\in\cK_0^*$. By Lemma \ref{lem_S2_eqn_new}, the map $x\mapsto S _2(x)$ is additive on $\cK_0^*$.  The existence of the desired reduced linearized polynomial $S_2(X)$ follows from  Lemma \ref{Lpoly}. \\

\noindent (i). The equation  \eqref{eqn_cond_si} follows by applying Lemma \ref{lem_Tracelinear} in the case $\mu=-\mu_B$, $\eta=\mu_C$ and $L=S_2$.  Since $\mu_B$ and each $s_i$ are in $\F_{p^l}$, the left hand side of  \eqref{eqn_cond_si} lies in $\F_{p^l}$. It follows that the right hand side is $g$-invariant. Since $\mu_C$ is in $\F_{p^l}$, we deduce that $\mu_C u^{p^i}-u\mu_C^{p^i}=\mu_C g(u)^{p^i}-g(u)\mu_C^{p^i}$, i.e., $ ((u-g(u))\mu_C^{-1})^{p^i}=(u-g(u))\mu_C^{-1}$. In the case $i=1$, it follows that $u-g(u)=h\mu_C$ for some $h\in\F_p$. Since $g(\mu_C)=\mu_C$, we deduce that $(1-g)^2(u)=h\mu_C-hg(\mu_C)=0$ as desired. \\

\noindent (ii). Let $\widetilde{S_2}$ be the trace dual of $S_2(X)$, so that $\tr_{\F_q/\F_p}(S_2(x)y)=\tr_{\F_q/\F_p}(x\widetilde{S_2}(y))$ for $x,y\in\F_q$. Also by Lemma \ref{lem_Tracelinear} with  $\mu=-\mu_B$, $\eta=\mu_C$ and $L=S_2$,   we have
\begin{equation*}
\widetilde{S_2}(\mu_Bx)= \mu_BS_2(x)+\mu_C\tr_{\F_q/\F_p}(u x)-u\tr_{\F_q/\F_p}(\mu_C x).
\end{equation*}
Multiply both sides  by $y$ and take the absolute traces, and we obtain  \eqref{eqn_BFSymm}.
\end{proof}

\begin{lemma}\label{lemma_norsufficond1p2}
Let $S_2(X)$ and $u$ be as in Lemma \ref{lemma_norsufficond1p1}, and write $\lambda_C=\tr_{\F_q/\F_p}(\mu_C t_C)$. There exists $\lambda'\in\F_p$ such that
\begin{equation}\label{eqn_alphagmal}
\alpha^g-\alpha=\mu_BS(t_C)-\mu_BS_2(t_C)+\lambda_Cu+\lambda' \mu_C.
\end{equation}
Moreover, we have $\tr_{\F_q/\F_{p^l}}(S(t_C)-S_2(t_C))=0$ and
\begin{equation}\label{eqn_sumalphagi}
\sum_{i=0}^{p-1}\alpha^{g^i}:=-\sum_{i=1}^{p-1}i \left( \mu_BS(t_C)^{g^{i-1}}-\mu_BS_2(t_C)^{g^{i-1}}+\lambda_C u^{g^{i-1}}\right).
\end{equation}
\end{lemma}
\begin{proof}
We observe that $\tr_{\F_q/\F_{p^l}}(u)=(1-g)^{p^2-1}(u)=0$ by  \eqref{eqn_1mgpow} and the fact $(1-g)^2(u)=0$ in (i) of Lemma \ref{lemma_norsufficond1p1}. Take $c\in\cK_0^*$, so that $\tr_{\F_q/\F_p}(\mu_Cc)=0$. Applying  \eqref{eqn_BFSymm} to the pair $(x,\,y)=(t_C,\,c^g)$, we obtain $\tr_{\F_q/\F_p}(\mu_Bt_CS_2(c^g) )=\tr_{\F_q/\F_p}(\mu_B c^gS_2(t_C)-\lambda_C uc^g )$. This is the left hand side of  \eqref{eqn_within1}. By collecting terms, \eqref{eqn_within1} reduces to
\[
\tr_{\F_q/\F_p}\left( (\alpha^g-\alpha-\mu_BS(t_C)+\mu_B S_2(t_C)-\lambda_C u)c^g \right)=0.
\]
This holds for all $c\in\cK_0^*$, so there exists $\lambda'\in\F_p$ such that  \eqref{eqn_alphagmal} holds by Lemma \ref{lem_tr}.
Taking the relative trace to $\F_{p^l}$ on both sides of  \eqref{eqn_alphagmal}, we deduce that $\tr_{\F_q/\F_{p^l}}(S(t_C)-S_2(t_C))=0$. For $0\le i\le p-2$ we have
\[
\binom{p-2}{i}=\frac{(p-2)\cdots(p-i)(p-i-1)}{2\cdots (i-1)i}\equiv(-1)^i(i+1) \pmod{p},
\]
and for $1\le i\le p-1$ we have $\binom{p-1}{i}\equiv(-1)^i\pmod{p}$. Therefore, by binomial expansion we have $(1-g)^{p-2}=\sum_{i=1}^{p-1}ig^{i-1}$, $(1-g)^{p-1}=1+g+\cdots+g^{p-1}$ in $\F_p[\la g\ra]$.  The equation \eqref{eqn_sumalphagi} follows by  applying $-(1-g)^{p-2}$ to both sides of  \eqref{eqn_alphagmal}.
\end{proof}

Finally, we explore the condition $\fg_{0,0,t_C}^p\in G_{\cK_0^*}$. For $i\ge 0$, set
\[
\fg_{a_i,b_i,c_i}:=\fg_{0,0,t_C}^i= (\cM_{0,0,t_C}^{g^{i-1}}\cdot\cdots\cdot\cM_{0,0,t_C},\ g^i).
\]
The value of $a_i$ is irrelevant, and we have $b_{i+1}=b_i^g+c_i^gS(t_C)$, $c_{i+1}=c_i^g+t_C$ by expansion. From $b_1=0$, $c_1=t_C$, we deduce that $b_{i+1}=\sum_{j=1}^{i} \sum_{k=1}^{j}g^{i-j}(t_C^{g^k}S(t_C))$, $c_{i+1}=\sum_{j=0}^{i} t_C^{g^{j}}$. In particular, we have
$\fg_{0,0,t_C}^p= \fg_{a'',b'',c''}$,  where $a''$ is irrelevant and
\begin{equation}\label{tCPin}
b''=b_p=\sum_{i=1}^{p-1}\sum_{k=1}^{i}g^{p-1-i}(t_C^{g^k}S(t_C)) ,\quad c''=c_p=\sum_{i=0}^{p-1}t_C^{g^{i}}.
\end{equation}
Since $g(\mu_C)=\mu_C$, we deduce that  $c''\in\cK_0^*$:
\[
\tr_{\F_q/\F_p}(\mu_Cc'')=\sum_{i=0}^{p-1}\tr_{\F_q/\F_p}(\mu_C^{g^{-i}}t_C)=p\tr_{\F_q/\F_p}(\mu_Ct_C)=0.
\]

\begin{lemma}\label{lem_rabrc2_p=3}
Take notation as in Lemma \ref{lemma_norsufficond1p2}, and write $\nu_B=M(t_B)$, $\nu_C=S(t_C)$, $\lambda_C=\tr_{\F_q/\F_p}(\mu_C t_C)$. Then  $p=3$, $\mu_C=u-u^g$, and $\nu_B=\sum_{i=0}^{p-1}\left(\nu_C-S_2(t_C)\right)^{g^i}$.
\end{lemma}
\begin{proof}
Since $\sigma_{t_C}^p=g^p=g_1$, the Frobenius part of $\fg_{0,0,t_C}^p$ is $g_1$. The Frobenius part of $\fg_{a'',b'',c''}$ is $g_1^D$ with  $D=\frac{1}{2}Q(c'')+\tr_{\F_q/\F_p}(\alpha c''+\mu_B b'')$ by the expression of $\theta$ in Lemma \ref{lem_non_expthetaT}. Since $\fg_{0,0,t_C}^p= \fg_{a'',b'',c''}$, we deduce that $D=1$.

By the expression of $T$ in Lemma \ref{lem_non_expthetaT}, the $(3,2)$-nd entry of the matrix part of $\fg_{a'',b'',c''}$ equals $T(a'',b'',c'')=\nu_B+S_2(c'')$. Here, recall that we chose $t_B$ such that $g_B=g_1$ in the beginning of this subsection. The $(3,2)$-nd entry of the matrix part of $\fg_{0,0,t_C}^p$ equals $\sum_{i=0}^{p-1}g^i(\nu_C)$ upon expansion. We deduce that $\nu_B+S_2(c'')=\sum_{i=0}^{p-1}g^i(\nu_C)$.

For the ease of notation, we write $\tr=\tr_{\F_q/\F_p}$ in this proof. We now compute $D$ explicitly. By Lemma \ref{lemma_norsufficond1p1}, $S_2$ has coefficients in $\F_{p^l}$, so $S_2(x^{g^i})=S_2(x)^{g^i}$ for $i\ge 0$ and $x\in\F_q$. Since $c''=\sum_{i=0}^{p-1}t_C^{g^{i}}$ is in $\cK_0^*$, we have $Q(c'')= -\tr(\mu_Bc''S_2(c''))$ by \eqref{eqn_Q_exp_cK0}. We compute that
\[\begin{array}{rcl}
 Q(c'')&=&-\sum_{i,j=0}^{p-1}\tr\left(\mu_B t_C^{g^i}S_2(t_C)^{g^j}\right) \\
 &=&-\sum_{i<j }  \tr\left(\mu_B t_C^{g^i}S_2(t_C)^{g^j}\right)-\sum_{ j<i}\tr\left(\mu_B t_C^{g^i}S_2(t_C)^{g^j}\right)\\
 &=&-\sum_{k=1}^{p-1}(p-k)\left(\tr (\mu_B t_CS_2(t_C^{g^k}) )+\tr (\mu_B t_C^{g^k}S_2(t_C) )\right)\\
 &=&\sum_{k=1}^{p-1}k\left(2\tr (\mu_B t_C^{g^k}S_2(t_C) )+\lambda_C\tr( u t_C)-\lambda_C\tr(u t_C^{g^k}) \right)\\
 &=& \sum_{k=1}^{p-1}k\cdot\tr(2\mu_B t_C^{g^k}S_2(t_C)-\lambda_C u t_C^{g^k} ).
\end{array}
\]
Here, in the second equality the $p$ terms with $i=j$ are equal and sum to $0$, in the fourth equality we used  \eqref{eqn_BFSymm} with $(x,\,y)=(t_C,\,t_C^{g^k})$, and in the fifth we used the fact $\sum_{i=1}^{p-1}i=\frac{p-1}{2}p\equiv0\pmod{p}$. Similarly, we have
\[
\begin{array}{rcl} \tr(\alpha c'')&=&\tr(t_C^{g^{p-1}}(\alpha+\cdots+\alpha^{g^{p-1}}))\\
&=&-\sum_{i=1}^{p-1}\tr\left(it_C^{g^{p-1}}  ( \mu_B\nu_C^{g^{i-1}}-\mu_BS_2(t_C)^{g^{i-1}}+\lambda_C u^{g^{i-1}} )\right) \\
&=& \sum_{i=1}^{p-1} i\cdot\tr\left( -t_C^{g^{p-i}} \mu_B\nu_C+ \mu_Bt_C^{g^{p-i}}S_2(t_C)- \lambda_C t_C^{g^{p-i}}  u \right) \\
&=& \sum_{k=1}^{p-1} k\cdot\tr(  \mu_B\nu_Ct_C^{g^k}-\mu_Bt_C^{g^{k}}S_2(t_C)+ \lambda_C ut_C^{g^{k}}  ),
\end{array}
\]
where in the second equality we used  \eqref{eqn_sumalphagi}, and in the last one we did a change of variable $p-i\mapsto k$. Finally, we have  $\tr(\mu_B b'')=-\sum_{i=1}^{p-1}i\cdot\tr( \mu_Bt_C^{g^{i}}\nu_C)$ by the fact $g(\mu_B)=\mu_B$ and Lemma \ref{lem_trace_int}. Putting those pieces together, we get
$D=\frac{\lambda_C}{2}\cdot\sum_{i=1}^{p-1}i\cdot\tr(ut_C^{g^i})$.

Set  $v:=u-u^g$, which is in $\F_p\cdot\mu_C$ by  Lemma \ref{lemma_norsufficond1p1}. In particular, $v$ is in $\F_{p^l}$ since $\mu_C$ is.   From $u^g=u-v$, we have $u^{g^i}=u-iv$ for $i\ge 0$ by induction. We have $\lambda_C=\tr(\mu_Ct_C)\ne 0$ by the fact $t_C\not\in\cK_0^*$. We  compute that
\[\begin{array}{rcl}
2\lambda_C^{-1}D &=& \sum_{i=1}^{p-1} i \cdot\tr( u^{g^{p-i-1}} t_C^{g^{p-1}})= \sum_{i=1}^{p-1} i\cdot \tr\left( (u-(p-i-1)v) t_C^{g^{p-1}}\right)\\
&=&\sum_{i=1}^{p-1} (i^2+i)\cdot \tr( v t_C^{g^{p-1}})=\frac{p(p-1)(p+1)}{3} \cdot \tr( v t_C).
\end{array}\]
This value is $0$ if $p>3$, so we must have $p=3$, in which case, $\frac{p(p-1)(p+1)}{3}\equiv -1 \pmod 3$. Since $D=1$, it follows from the above equation that $\lambda_C=\tr(v t_C)$.

Combining the facts $\lambda_C=\tr(\mu_C t_C)$, $\lambda_C=\tr( v t_C)$, and $v\in\F_p\cdot \mu_C$, we deduce that $\mu_C=v$, i.e., $\mu_C=u-u^g$.  Since  $S_2(x^g)=S_2(x)^g$ for $x\in\F_q$, the desired expression of $\nu_B$ follows from  $\nu_B=\sum_{i=0}^{p-1}g^i(\nu_C)-S_2(c'')$. This completes the proof.
\end{proof}

Recall that the functions $T$ and $\theta$ are as in Lemma \ref{lem_non_expthetaT} when restricted to $\F_q\times\F_q\times\cK_0^*$. It turns out that the conditions that we have derived on the parameters in Lemmas \ref{lemma_norsufficond1p2} and \ref{lem_rabrc2_p=3} are also sufficient. This leads to the following construction.
\begin{construction}\label{Const_rC2}
Suppose that $q=3^{9l}$ with $l$ a positive integer, and take $g\in\Aut(\F_q)$ such that $g(x)=x^{p^l}$. Set $g_1:=g^3$.
\begin{enumerate}
\item[(i)] Take $u\in \F_{q}$ such that $\mu_C:=u-u^g\in\F_{3^l}^*$;
\item[(ii)] Take   $t_C\in\F_q^*$  such that
    $\lambda_C:=\tr_{\F_q/\F_3}(\mu_C t_C)\ne 0$;
\item[(iii)] Take $\mu_B\in\F_{3^l}^*$;
\item[(iv)] Take a tuple $(s_0,\,s_1,\cdots,s_{9l-1})$ with entries in $\F_{3^l}$ that satisfies
    \[
    -\mu_B s_i+s_{9l-i}^{3^i}\mu_B^{3^i}=\mu_C u^{3^i}-u\mu_C^{3^i},\;1\leq i\leq 9l-1;
    \]
    Set $S_2(x):=\sum_{i=0}^{3^2l-1}s_ix^{3^i}$, $Q(x):=-\tr_{\F_q/\F_3}(\mu_BxS_2(x))$ for $x\in\F_q$;
\item[(v)] Take $\alpha\in\F_q$, $\lambda\in\F_3$, and set
    $\nu_C:=S_2(t_C)+\mu_B^{-1}(\alpha^g-\alpha-\lambda_Cu-\lambda\mu_C)$.
\end{enumerate}
Set $\nu_B:=\sum_{i=0}^{2}g^i(\nu_C-S_2(t_C))$, and let $N_B$ be as  defined in  \eqref{eqn_Nc} with the prescribed $\nu_B$. Set $K:=\{z\in\F_q :\,\tr_{\F_q/\F_3}(\mu_C z)=0\}$.

For $a,\,b\in\F_q$ and $c\in K$, let $\theta_{a,b,c}$ and $T(a,b,c)$ be as defined in  Lemma \ref{lem_non_expthetaT}, and set $\cM_{a,b,c}=E(a,b,c, T(a,b,c))$, where $E$ is as in  \eqref{eqn_EMat}. Then $G_{K}:=\{\fg_{a,b,c}:\,a,\,b\in\F_q,\, c\in K\}$ is a  subgroup  of  order  $q^3/3$, where $\fg_{a,b,c}=(\cM_{a,b,c},\,\theta_{a,b,c})$. Take $\fg_{0,0,t_C}=(\cM_{0,0,t_C},\, g)$ with $\cM_{0,0,t_C}:=E(0,0,t_C,\nu_C)$.
Then $G:=\la G_K,\, \fg_{0,0,t_C}\ra$ is a point regular group of $\cQ^P$.
\end{construction}	
\begin{proof}
The number of tuples satisfying the conditions in (iv) equals $3^{(9l+1)l/2}$ by a  similar argument to that in the proof of Lemma \ref{lem_num}. By the same argument as in the proof of Construction \ref{Constquadodd}, we can show that $G_K$ is a group of order $q^3/3$. Let $G_1$ be the subgroup of $G_K$ of index $3$ with a trivial Frobenius part. Then we can check that $\fg_{0,0,t_C}$ normalizes $G_1$ and $\fg_{0,0,t_C}^3\in G_K$ by reversing the arguments in this subsection. Since $\fg_{0,0,t_C}^3$ has Frobenius part $g_1=g^3$, we deduce that $G_K=\la G_1,\fg_{0,0,t_C}^3\ra$.  It follows that $G_K$ is normalized by $\fg_{0,0,t_C}$. We thus conclude that $G$ is a group of order $q^3$.

The fact that $G$ is point regular on $\cQ^P$ can be derived in exactly the same way as in the case of Construction \ref{Conk0=k0}, i.e., by showing that the orbit $X_1$ of $\la (0,0,0,1)\ra$ under the action of $G_K$ has size $q^3/3$ and $\fg_{0,0,t_C}$ maps $\la (0,0,0,1)\ra$ to a point outside $X_1$.
\end{proof}

To summarize, we have thus proved the case $r_{A,B}=1$, $r_C=2$ of Theorem \ref{thm_nonlinear} in this subsection. This completes the proof of Theorem \ref{thm_nonlinear}.

\section{The isomorphism issue in the odd characteristic case}\label{sec_iso}

\subsection{The conjugacy within $\PGaSp(4,q)_P$}
In this subsection, we complete the proof of the classification theorem for $q$ odd, namely,
Theorem \ref{thm_conjugate}. It will follow from Theorem \ref{thm_linearPRG} (the linear case) and Theorem \ref{thm_nonlinear} (the nonlinear case) provided that we can prove that the groups arising from Constructions \ref{Conk0=k0}, \ref{Constquadodd} and \ref{Const_rC2} are conjugate to those in Constructions \ref{const_S2}-\ref{const_S4} within $\PGaSp(4,q)_P$ respectively. We observe that  Constructions \ref{const_S2}-\ref{const_S4} are special cases of Constructions \ref{Conk0=k0}, \ref{Constquadodd} and \ref{Const_rC2} respectively. We establish the conjugacy claim for each of the three cases in the sequel, and thus establish Theorem \ref{thm_conjugate}. \\

\noindent(A) In Construction \ref{Conk0=k0}, set $\nu:=\nu_C-S_1(t_C)$. Since $\tr_{\F_q/\F_{p^l}}(\nu)=0$, there exists $u\in\F_q$ such that $u^g-u=\nu$ by Lemma \ref{lem_ff1}. Take $\fg_1:= (E(0,0,0,u),\, 1)$, which lies in $\PGaSp(4,q)_P$ and stabilizes the quadrangle $\cQ^P$. The group $\tilde{G}:=\fg_1^{-1} \circ G \circ \fg_{1}$ is also a point regular group of $\cQ^P$, so is of the form as in Theorem \ref{Main} for some functions $T'$ and $\theta'$ by the analysis in Section 2.2. Let $\fg_{a,b,c}'$ be the element of $\tilde{G}$ that maps $\la(0,0,0,1)\ra$ to $\la(a,b,c,1)\ra$, and set $\sigma_c':=\theta_{0,0,c}'$, $M'(y):=T'(0,y,0)$ and $S'(z):=T'(0,0,z)$. For $a,b\in\F_q$ and $c\in K$, we compute that
\[
    \fg'_{a,b,c}=\fg_1^{-1}\circ\fg_{a,b-c\nu,c}\circ \fg_{1}
     =(E(a,b,c,S_1(c)),\,1).
\]
It follows that   $T'(a,b,c)=S_1(c)$, $\theta'(a,b,c)=1$ if $c\in K$. In particular, $M'\equiv 0$, $\theta'_{0,y,0}\equiv 1$. Similarly, we compute that
\[
      \fg'_{0,ut_C,t_C}=\fg_{1}^{-1}\circ \fg_{0,0,t_C} \circ \fg_1=(E(0,ut_C,t_C,S_1(t_C)),\,g).
\]
That is,  $T'(0,ut_C,t_C)=S_1(t_C)$, $\theta'_{0,ut_C,t_C}=g$. By  \eqref{eqn_theta6} and \eqref{eqn_T6}, we have  $\theta_{a,b,c}'=\theta_{0,b,0}'\sigma_c'$, $T'(a,b,c)=M'(b)^{\sigma_c'}+S'(c)$ for $a,b,c\in\F_q$. Therefore,
$\sigma_{t_C}'=\theta'_{0,ut_C,t_C}=g$ and  $S'(t_C)=T'(0,ut_C,t_C)=S_1(t_C)$. It follows that $\fg_{0,0,t_C}'=(E(0,0,t_C,S_1(t_C)),\,g)$.

Set $\tilde{G}_K:=\fg_1^{-1} \circ G_K \circ \fg_{1}$, the kernel of the group homomorphism $\psi':\,\fg_{a,b,c}'\mapsto\theta_{a,b,c}'$. Then $\tilde{G}=\la \tilde{G}_K,\,\fg_{0,0,t_C}'\ra$.
Take a  triple $(x,y,z)$. There exist  $a,b\in\F_q$, $c\in K$ and $i\ge 0$ such that $\fg_{x,y,z}'=\fg_{a,b,c}'\circ\fg_{0,0,t_C}'^i$.  We compute that $z=g^i(c)+\sum_{k=0}^{i-1}g^k(t_C)$,  $T'(x,y,z)=S_1(z)$ and $\theta_{x,y,z}'=g^{\tr_{\F_q/\F_p}(\mu_C' z)}$, where $\mu_C'=\mu_C\tr_{\F_q/\F_p}(\mu_C t_C)^{-1}$. Therefore, $\tilde{G}$ arises from Construction \ref{const_S2}, and this establishes the conjugacy claim for Construction \ref{Conk0=k0}.\\

\noindent(B) In Construction \ref{Constquadodd}, take $u:=\mu_B^{-1}\alpha$ so that  $\nu_B =u^{g_1}-u $ by (iii). We deduce that $N_B(\theta_{a,b,c})=u^{\theta_{a,b,c}}-u$ from $\nu_B =u^{g_1}-u$, where $N_B$ is as in  \eqref{eqn_Nc}. Taking conjugation by $\fg_1= (E(0,0,0,u),\, 1)\in\PGaSp(4,q)_P$, we calculate that $\fg_{a,b,c}':=\fg_1^{-1} \circ \fg_{a,b-uc,c}\circ\fg_1$ equals $(E(a,b,c, S_2(c)), g_1^{\frac{1}{2}Q(c)+\tr_{\F_q/\F_p}(\mu_Bb)})$. By changing the notation $S_2$ to $S_1$ to be consistent with previous constructions, we see that $\fg_1^{-1} \circ G\circ\fg_1$ takes the form in Construction \ref{const_S3}. This establishes the conjugacy claim for  Construction \ref{Constquadodd}.\\

\noindent(C) In  Construction \ref{Const_rC2}, set $\nu:= \nu_C-S_2(t_C)$, which equals $\mu_B^{-1}(\alpha^g-\alpha-\lambda_Cu-\lambda\mu_C)$ by (v). We deduce that $\tr_{\F_q/\F_{3^l}}(\nu)=0$ by the facts $(1-g)^2(u)=0$ and $\mu_C\in\F_{3^l}$ in (i).
There exists $u_0\in\F_q$ such that $\nu=u_0^g-u_0$ by Lemma \ref{lem_ff1}. Set $\lambda:=\lambda_C^{-1}\tr_{\F_q/\F_3}(\mu_B u_0 t_C)\in\F_3$ and  $u_1:=u_0-\lambda \mu_B^{-1} \mu_C$, so that $\tr_{\F_q/\F_3}(\mu_Bt_C u_1)=0$.  Since $\mu_B$ and $\mu_C$ are in $\F_{3^l}^*$, we deduce that $u_1^g-u_1=\nu$. It follows  that  $\nu_B=\sum_{i=0}^{2}g^i(\nu)=g_1(u_1)-u_1$, where $g_1=g^3$.   The following are some facts that we need below:
\begin{enumerate}
\item[(1)] Set $\alpha':=\alpha-u_1\mu_B$. We deduce that $g(\alpha')-\alpha'=\lambda_Cu+\lambda\mu_C$ from the facts $u_1^g-u_1=\nu$ and $\nu=\mu_B^{-1}(\alpha^g-\alpha-\lambda_Cu-\lambda\mu_C)$. Since $u-u^g\in\F_{3^l}^*$ and $(g-1)^3=g_1-1$, we further deduce that  $g_1(\alpha')-\alpha'=(g-1)^2(\lambda_Cu+\lambda\mu_C)=0$, i.e., $\alpha'\in\F_{3^{3l}}$.
\item[(2)] Since $\tr_{\F_q/\F_3}(\mu_Bt_C u_1)=0$, by  \eqref{eqn_theta6} we have $\theta_{0,-t_Cu_1,0}= g_1^{\tr_{\F_q/\F_3}(-\mu_B t_Cu_1)}=1$ and
    $\theta_{0,-t_C u_1,t_C}=\theta_{0,-t_Cu_1,0}\cdot \sigma_{t_C}= g$.
\item[(3)]  By   \eqref{eqn_T6},  $T(0,-t_Cu_1,t_C)=N_B(\theta_{0,-t_Cu_1,0})^g+\nu_C$. We have  $\theta_{0,-t_Cu_1,0}=1$, so $M(-t_Cu_1)=N_B(1)=0$. It follows that $T(0,-t_Cu_1,t_C)=\nu_C$.
\end{enumerate}
Let $\tilde{G}:=\fg_1^{-1} \circ G\circ\fg_1$ with $\fg_1= (E(0,0,0,u_1),\, 1)\in\PGaSp(4,q)_P$, and let $\fg_{a,b,c}'$ be the element of $\tilde{G}$ that maps $\la(0,0,0,1)\ra$ to $\la(a,b,c,1)\ra$. By the analysis in Section 2.2, the group $\tilde{G}$ arises from Theorem \ref{Main} for some functions $T'$ and $\theta'$. By exactly the same argument as in the preceding case, we deduce that the subgroup $\tilde{G}_K:=\fg_1^{-1} \circ G_K\circ\fg_1$ consists of $\fg_{a,b,c}'=(\cM_{a,b,c}',\,\theta_{a,b,c}')$ with $a,\,b\in\F_q$, $c\in K$, where $\theta_{a,b,c}'=g_1^{\frac{1}{2}Q(c)+\tr_{\F_q/\F_p}(\alpha'c+\mu_B b)}$ and $\cM'(a,b,c)=E(a,b,c,S_2(c))$. Similarly, $\fg_{0,0,t_C}'=\fg_1^{-1}\circ \fg_{0,-t_Cu_1,t_C} \circ\fg_1$, and its Frobenius part is $g$ by the fact (2) above.  We compare their matrix parts and get $T'(0,0,t_C)=\nu_C-g(u_1)+u_1=S_2(t_C)$, where we used the fact (3) above. By changing the notation $S_2$ to $S_1$ for consistency,  we see that $\tilde{G}=\fg_1^{-1} \circ G\circ\fg_1$ takes the form in Construction \ref{const_S4}. This establishes the conjugacy claim for Construction \ref{Const_rC2}.\\

To conclude, we have now established the conjugacy claims for all constructions. This completes the proof of Theorem \ref{thm_conjugate}.

\subsection{The group invariants of the point regular groups}

Our first goal in this subsection is to show that the four constructions in general yield non-isomorphic point regular groups. We start with some properties of $G$.
\begin{lemma}\label{lem_S4T}
In each of Constructions \ref{const_S1}-\ref{const_S4}, we have $T(x,y,z)=S_1(z)$.
\end{lemma}
\begin{proof}
This is clear for the first three constructions, so assume that we are in the case of Construction \ref{const_S4}. Since the $s_i$'s are in $\F_{3^l}$, we have $S_1(g(x))=g(S_1(x))$ for $x\in\F_q$, i.e., $S_1$ and $g$ commutes. Since $G=\la G_K,\,\fg_{0,0,t_C}\ra$, for each triple $(x,\,y,\,z)$, there exist  $a,\,b\in\F_q$, $c\in K$ and $i\ge 0$ such that $\fg_{x,y,z}=\fg_{a,b,c}\circ\fg_{0,0,t_C}^i$. By direct expansion using the calculations in Remark \ref{rem_EMult} and the fact that $S_1$ and $g$ commutes, this equation yields  $z=g^i(c)+\sum_{k=0}^{i-1}g^k(t_C)$ and $T(x,y,z)=S_1(z)$. This proves the claim.
\end{proof}

\begin{lemma}\label{lem_P1P2}
Let $G$ be a group in either of Constructions \ref{const_S1}-\ref{const_S4}. Let $G_F$ be the subgroup
\begin{equation}\label{eqn_def_GF}
	G_F:=\{\fg_{a,b,c}:\,\theta_{a,b,c}=1\},
\end{equation}
and set $U:=\{c\in\F_q:\,\theta_{a,b,c}=1\textup{ for some }a,b\in \F_q\}$.
\begin{enumerate}
\item[(P1)] The set $U=\F_q$ for Constructions \ref{const_S1} and \ref{const_S3}, and $U$ is an $\F_p$-subspace of codimension $1$ for Constructions \ref{const_S2} and \ref{const_S4}.
\item[(P2)] For $\fg_{a,b,c}\in G_F$, it has order $p$ or $p^2$, and the latter occurs if and only if $p=3$ and $S_1(c)\ne0$.
\end{enumerate}
\end{lemma}
\begin{proof}
The claim (P1) follows by a case by case check and we omit the details. Set $T_1:=T(a,b,c)$ and $\fg_{a_i,b_i,c_i}:=\fg_{a,b,c}^i$. By using the calculations in Remark \ref{rem_EMult}, we deduce from $\fg_{a_{i+1},b_{i+1},c_{i+1}}=\fg_{a,b,c}\circ\fg_{a_i,b_i,c_i}$ that $c_{i+1}=c_i+c$, $b_{i+1}=b_i+c_i T_1$ and $a_{i+1}=a_i-b_ic+c_ib-c_icT_1$ for $i\ge 0$. Since $a_0=b_0=c_0=0$,  we deduce that
\[
c_i=ic,\,b_i=ib+\frac{(i-1)i}{2}cT_1,\,a_i=ia-\frac{(i^2-1)i}{6}c^2T_1.
\]
In particular, $\fg_{a,b,c}^p=\fg_{x,0,0}$ with $x=-\frac{(p^2-1)p}{6} c^2S_1(c)$, and $\fg_{a,b,c}^{p^2}=\fg_{x,0,0}^p=1$. Since $p$ is odd, $\frac{(p^2-1)p}{6}$ is nonzero in $\F_p$ if and only if $p=3$. Moreover, $S_1$ is additive, so $cS_1(c)\ne 0$ if and only if $S_1(c)\ne 0$. Therefore, $\fg_{a,b,c}$ has order $p^2$ if and only if $p=3$ and $S_1(c)\ne 0$. This establishes the claim (P2).
\end{proof}

\begin{thm}\label{thm_exponent}
	Let $G$ be the group in either of   Constructions \ref{const_S1}-\ref{const_S4}.
	\begin{enumerate}
		\item[(a)] In  Construction \ref{const_S1},  $\exp(G)=p$ or $p^2$, and the latter occurs if and only if $p=3$ and $S_1$ is not the zero map.
		\item[(b)] In  Construction \ref{const_S2} and \ref{const_S3}, $\exp(G)=p^2$ or $p^3$, and the latter occurs if and only if $p=3$ and the restriction of $S_1$ to $\F_{3^l}$ is not the zero map.
		\item[(c)] In  Construction \ref{const_S4}, $\exp(G)=3^3$ or $3^4$, and the latter occurs if and only if  the restriction of $S_1$ to $\F_{3^l}$ is not the zero map.
	\end{enumerate}
\end{thm}
\begin{proof}	
The analysis of each construction is similar, and we only explain the case of Construction \ref{const_S4} in detail. Let $G_F$ be as in  \eqref{eqn_def_GF}, and define $\psi:\,G\rightarrow\Aut(\F_q)$, $\fg_{a,b,c}\mapsto\theta_{a,b,c}$. Then $G_F=\ker(\psi)$.  We first make some observations.
\begin{enumerate}
\item[(1)] $\F_q=\la K,\,t_C\ra_{\F_3}$, since $t_C\not\in K$ and $K$ has codimension $1$. Moreover, $\la t_C+K\ra_{\F_3}$ contains $t_C$ and thus also $K$, so it equals $\F_q$.
\item[(2)] Since $[G:\,G_K]=3$ and $G_K=\{\fg_{a,b,c}:\,a,b\in\F_q,\,c\in K\}$, we have $G_K\circ\fg_{0,0,t_C}=\{\fg_{a,b,t_C+g(c)}:\,a,b\in\F_q,\,c\in K\}$.
\item[(3)] Since $\im(\psi)=\la g\ra$ has order $9$, we have $\fg^9\in G_F$ for any $\fg\in G$.
\end{enumerate}

We claim that if $S_1|_{\F_{3^l}}\ne0$, then there is $c\in K$ such that $S_1(w)\ne 0$, where $w=\tr_{\F_q/\F_{3^l}}(t_C+c)$. Otherwise, $t_C+K$ is in the kernel of the $\F_3$-linear map $x\mapsto S_1(\tr_{\F_q/\F_{3^l}}(x))$, and so $\la t_C+K\ra_{\F_3}=\F_q$ is also in the kernel. Since the trace map is surjective, we get a contradiction $\F_{3^l}\subseteq\ker(S_1)$. This proves the claim.

Let $(x,y,z)$ be a triple, and write $g^i=\theta_{x,y,z}$, $\fg_{u,v,w}=\fg_{x,y,z}^{9}$. We have $\fg_{u,v,w}\in G_F$ by (3). We also deduce that $w=\sum_{k=0}^8g^{ik}(z)$ by expansion. If further $z=t_C+g(c)$ for some $c\in K$, then $\fg_{x,y,z}\in G_K\circ\fg_{0,0,t_C}$ by (2) and so $\theta_{x,y,z}=g^{1+3j}$ for some $j$; correspondingly, $w=\sum_{k=0}^{8}\theta_{x,y,z}^k(z)=\tr_{\F_q/\F_{3^l}}(t_C+c)$ by the fact $\sum_{k=0}^8g^{k}(z)=\tr_{\F_q/\F_{3^l}}(z)$.

We claim that $3^3\le \exp(G)\le 3^4$. Since $\fg_{u,v,w}=\fg_{x,y,z}^{9}$ is in $G_F$, it has order at most $3^2$ by (P2) of Lemma \ref{lem_P1P2}, so $\exp(G)\le 3^4$. On the other hand, there exists $c\in K$ such that $\lambda:=\tr_{\F_q/\F_{3^l}}(t_C+c)\ne 0$ by (1). By the preceding paragraph, we have  $\fg_{0,0,t_C+c}^{9}=\fg_{u,v,\lambda}$ for some $u,v$, and so $\fg_{0,0,t_C+c}$ has order at least $3^3$. This proves the claim.

It remains to decide when there will be a triple $(x,y,z)$ such that $\fg:=\fg_{x,y,z}$ has order $3^4$. This is the case if and only if $\theta_{x,y,z}$ has order $9$ and $\fg_{u,v,w}=\fg^9$ also has order $9$ by (3). Suppose that both conditions are true. Since $[G:\,G_K]=3$, by replacing $\fg$ with $\fg^{-1}$ if necessary, we can assume that $\fg\in G_K\circ \fg_{0,0,t_C}$. It follows that $z=t_C+g(c)$ for some $c\in K$ by (2), and correspondingly $w=\tr_{\F_q/\F_{3^l}}(t_C+c)\in\F_{3^l}$ as we showed. The element $\fg_{u,v,w}$ has order $9$ if and only if $S_1(w)\ne 0$ by (P2) of Lemma \ref{lem_P1P2}. There is no such element if $S_1|_{\F_{3^l}}\equiv 0$. If $S_1$ is not constantly zero on $\F_{3^l}$, we showed earlier that there is $c\in K$ such that $S_1(w)\ne 0$, and so $\fg_{0,0,t_C+g(c)}^9$ has order $9$. The claim now follows.
\end{proof}

Theorem \ref{thm_exponent} does not help to distinguish Construction \ref{const_S2} and Construction \ref{const_S3}. We compute their Thompson subgroups which will do the work in general. Let $G$ be a group in either of Constructions \ref{const_S1}-\ref{const_S4}. Take $\fg_{a,b,c}\in G$. For $\fg_{x,y,z} \in C_G(\fg_{a,b,c})$,  by Theorem \ref{Main} we deduce from $\fg_{a,b,c}\circ\fg_{x,y,z}=\fg_{x,y,z}\circ \fg_{a,b,c}$ that
\begin{equation}\label{eqn_centereq}
\begin{array}{c}
a^{\theta_2}+x-b^{\theta_2}z+c^{\theta_2}y-c^{\theta_2}zS_1(z)=x^{\theta_1}+a-cy^{\theta_1}+bz^{\theta_1}-cz^{\theta_1}S_1(c);\\
b^{\theta_2}+y+c^{\theta_2}S_1(z)=b+y^{\theta_1}+z^{\theta_1}S_1(c),\\
c^{\theta_2}+z=c+z^{\theta_1}. \\
\end{array}
\end{equation}
where $\theta_{1}=\theta_{a,b,c}$, $\theta_{2}=\theta_{x,y,z}$. Here we used the fact $T(x,y,z)=S_1(z)$ in Lemma \ref{lem_S4T}.

\begin{thm}\label{thm_thompson}
	Let $G$ be the group in either of  Construction \ref{const_S2} or Construction \ref{const_S3}, and assume that $1<\deg(S_1)<q/p$. Then the Thompson subgroup of $G$ is $J(G)=\{\fg_{a,b,0}:\,a,\,b\in\F_q, \,\theta_{a,b,0}=1\}$, which has size $q^2$ in the former case and $q^2/p$ in the latter case.
\end{thm}
\begin{proof}
	We only deal with the case of  Construction \ref{const_S3} here, and the other case is similar. Let $d$ be the largest order of an abelian subgroup of $G$. It is routine to check that $G_{ab}:=\{\fg_{a,b,0}:\,a,\,b\in\F_q, \theta_{a,b,0}=1\}$ is an abelian subgroup of order $q^2/p$, so $d\ge q^2/p$.
	
	We first show that $J(G)\le G_F$, where $G_F$ is the set of elements of $G$ with a trivial Frobenius part. This is achieved by showing that any abelian subgroup of order $d$ is contained in $G_F$. Suppose to the contrary that $H$ is an abelian subgroup of order $d$ which contains an element $\fg_{a,b,c}\in G$ with $\theta_{a,b,c}=g^i\ne 1$. The subgroup $C_G(\fg_{a,b,c})$ contains $H$, so should have size at least $d\ge q^2/p$. We now estimate the size of $C_G(\fg_{a,b,c})$ in an alternative way. Fix an integer $j$, $0\le j\le p-1$. Suppose that $\fg_{x,y,z}\in C_G(\fg_{a,b,c})$ has Frobenius part $\theta_{x,y,z}=g^j$. By the third equation in \eqref{eqn_centereq}, we have $z^{g^i}-z=-c^{g^j}+c$. This equation in $z$ has at most $p^l$ solutions, since $o(g^i)=p$ and $q=p^{pl}$. For a given $z$, the second equation  in \eqref{eqn_centereq} has at most $p^l$ solutions in the variable $y$ for the same reason. Similarly, for a given pair $(y,z)$, the first equation in \eqref{eqn_centereq} has at most $p^l$ solutions in $x$. In total, we see that $|C_G(\fg_{a,b,c})|\le p^{3l+1}$. This number is less than $q^2/p=p^{2pl-1}$ by the fact $(2p-3)l\ge 3$.  This proves the claim.

	We next show that $J(G)\le G_{ab}$.  As in the previous paragraph, it suffices to show that for any $\fg_{a,b,c}\in G_F$ with $c\ne 0$ its centralizer in $G_F$ (not $G$) has size smaller that $q^2/p$. For  $\fg_{x,y,z}\in C_{G_F}(\fg_{a,b,c})$, it has a trivial Frobenius part, and the equations in
	\eqref{eqn_centereq} reduce to $cS_1(z)-zS_1(c)=0$ and $2y=2bc^{-1}z+zS_1(z)-zS_1(c)$. Recall that $S_1$ is  $\F_p$-linear. By the restriction on $\deg(S_1)$, we see that there are at most $q/p^2$ such $(y,z)$ pairs. Therefore, $C_{G_F}(\fg_{a,b,c})$ has size at most $q^2/p^2$ as desired. This proves the claim.
	
	Since $G_{ab}$ is abelian, we conclude that $J(G)=G_{ab}$ and it is the unique maximal abelian subgroup of order $q^2/p$. This completes the proof.
\end{proof}

Our next goal is to bound the nilpotency class of $G$. We start with the center of $G$. Recall that $G_A=\{\fg_{a,0,0,}:\,a\in\F_q\}$, cf. Notation \ref{notation_sigmaLMS}.
\begin{lemma}\label{lem_center}
	Let $G$ be the group in either of  Constructions \ref{const_S1}-\ref{const_S4}. Then its center is $Z(G)=\{\fg_{a,0,0}:\, a\in\F_{p^l}\}$.
\end{lemma}
\begin{proof}
Take  $\fg_{a,b,c}\in Z(G)$, and set $\theta_1:=\theta_{a,b,c}$. The equations in \eqref{eqn_centereq} hold for all triple $(x,y,z)$'s. First, we claim that $\theta_1=1$. Suppose that $\theta_1\ne 1$. For a triple $(x,\,y,\,z)$ such that $\theta_2:=\theta_{x,y,z}$ equals $1$, the third equation in  \eqref{eqn_centereq} reduces to $z=z^{\theta_1}$, i.e., $z$ lies in a proper subfield of $\F_q$. In particular, there are at most $\sqrt{q}$ $z$'s such that there is a triple $(x,\,y,\,z)$ with $\theta_{x,y,z}=1$. On the other hand, this number is at least $q/p$ by the property (P1) in Lemma \ref{lem_P1P2}: a contradiction. This proves the claim.
	
	Next, we show that $c=0$.  Take $(x,y,z)=(0,y_0,0)$ for a nonzero element $y_0$ such that $\theta_{0,y_0,0}=1$.  The first equation in \eqref{eqn_centereq} reduces to $2cy_0=0$, which gives $c=0$.
	
	By the facts $\theta_1=1$ and $c=0$, the equations in \eqref{eqn_centereq} reduce to $a^{\theta_2}-b^{\theta_2}z=a+bz$ and $b^{\theta_2}=b$, where $\theta_2=\theta_{x,y,z}$. For any triple $(x,y,z)$ with $z\ne 0$ and $\theta_{x,y,z}=1$, we deduce from the first equation that $2bz=0$, so $b=0$.  The equations in \eqref{eqn_centereq} further reduce to $a^{\theta_{x,y,z}}=a$ for all $x,\,y,\,z\in\F_q$. That is, $a^g=a$. This completes the proof.
\end{proof}

Let $g\in \Aut(\F_{q})$ be such that $g(x)=x^{p^l}$, $x\in\F_q$. We regard $\F_q$ as an $\F_{3^l}[\la g\ra]$-modules. Set $R_0:=\F_q$, and $R_i:=\{(1-g)^{i}(x):\,x\in\F_q\}$ for $1\le i\le p^e$.  By Lemma \ref{lem_ff1} and the fact that $\tr_{\F_q/\F_{p^l}}(x)=(1-g)^{p^e-1}(x)$, we have $R_{p^e-1}=\F_{p^l}$. There is a short exact sequence of $\F_{3^l}[\la g\ra]$-module homomorphisms
\[
0\longrightarrow\F_{p^l}\stackrel{1}\longrightarrow R_i\stackrel{1-g}\longrightarrow  R_{i+1}\longrightarrow 0,
\]
for each $0\le i\le p^e-1$. It follows that $\dim_{\F_p}R_i=(p^e-i)l$; in particular, $R_{p^e}=0$. In the sequel, we write $x\equiv y\pmod{R_i}$ if $x-y\in R_i$.
\begin{lemma}\label{lem_ZpeGA}
	Let $G$ be the group in either of  Constructions \ref{const_S1}-\ref{const_S4}, and let $p^e$ be the order of $g$. For the last three constructions, further assume that $l>1$. For $1\le i\le p^e$, we have $Z_{i}(G):=\{\fg_{a,0,0}:\,a\in (1-g)^{p^e-i}(\F_q) \}$. In particular, $Z_{p^e}(G)=G_A$.
\end{lemma}
\begin{proof}
The case $i=1$ follows from Lemma \ref{lem_center} and the fact $R_{p^e-1}=\F_{p^l}$. Assume that $1\le i\le p^e-1$. By a similar argument to that preceding Theorem \ref{thm_thompson}, we see that $\fg_{a,b,c}\in Z_{i+1}(G)$ if and only if
	\begin{equation}\label{eqn_newRel1}
	\begin{split}
	a^{\theta_2}+x-b^{\theta_2}z&+c^{\theta_2}y-c^{\theta_2}zS_1(z)\\
	&\equiv x^{\theta_1}+a-cy^{\theta_1}+bz^{\theta_1}-cz^{\theta_1}S_1(c)\pmod{R_{p^e-i}}.
	\end{split}
	\end{equation}
and the last two equations of \eqref{eqn_centereq} hold for all $x,\,y,\,z$, where $\theta_1=\theta_{a,b,c}$, $\theta_2=\theta_{x,y,z}$.

Suppose that $\fg_{a,b,c}\in Z_{i+1}(G)$. By the same argument as in the proof of Lemma \ref{lem_center}, we deduce that $\theta_1=1$. We claim that $c=0$. In the first two constructions, $\theta_{0,y_0,0}=1$ for all $y_0\in\F_q$; in the last two constructions, there are $q/p$ $y_0$'s in $\F_q$ such that $\theta_{0,y_0,0}=1$.  Take $(x,y,z)=(0,y_0,0)$ for such an element $y_0$,  \eqref{eqn_newRel1} reduces to $2cy_0\in R_{p^e-i}$, $1\leq i\leq p^e-1$. If $c\ne 0$, this leads to a contradiction by  comparing sizes. Here we used the assumption $l>1$ for the last two constructions. This proves the claim.
	
By the facts $\theta_1=1$ and $c=0$, the conditions now reduce to $b^{\theta_{x,y,z}}=b$ and $2bz\equiv a^{\theta_{x,y,z}}-a \pmod {R_{p^e-i}}$. Let $G_F$ and $U$ be as in Lemma \ref{lem_P1P2}. We consider two separate cases according as $i=p^e-1$ or not.
\begin{enumerate}
\item[(1)] First consider the case $i<p^e-1$. For each $z\in U$, take a triple $(x,y,z)$ with $\theta_{x,y,z}=1$, and  \eqref{eqn_newRel1} reduces to $2bz\in R_{p^e-i}$. By (P1) of Lemma \ref{lem_P1P2}, we deduce a contradiction by comparing sizes if $b\ne 0$. Hence $b=0$. The conditions further reduce to $a^{\theta_{x,y,z}}-a\in R_{p^e-i}$ for all $x,\,y,\,z$, or equivalently, $a^g-a\in R_{p^e-i}$. This holds if and only if $a\in R_{p^e-i-1}$ as desired.
\item[(2)] Next consider the case $i=p^e-1$.  In this case, $a^{g^k}-a=(g^k-1)(a)$ is always contained in $R_1=R_{p^e-i}$ for $k\ge 0$, so the conditions reduce to $b^{\theta_{x,y,z}}=b$ and $2bz\in R_{1}$ for all $x,\,y,\,z\in\F_q$. It follows that $b=0$.
\end{enumerate}
This completes the proof.
\end{proof}

Take the same notation and assumption as in Lemma \ref{lem_ZpeGA}. By \cite[(9.7)]{FGroup}, the nilpotency class of $H$ equals $1$ plus that of $H/Z(H)$ for any nilpotent group $H$. Inductively,  the nilpotency class of $H$ equals $i$ plus that of $H/Z_i(H)$. By Lemma \ref{lem_ZpeGA}, the nilpotency class of $G$ equals $p^e$ plus that of $G/G_A$, where $p^e=o(g)$. We now consider the nilpotency class of  $\bar{G}:=G/G_A=\{\overline{\fg_{a,b,c}}:\,a,b,c\in\F_q\}$. We make the observation that  \textit{$\overline{\fg_{a,b,c}}$ lies in $Z(\bar{G})$ if and only if the last three equations in  \eqref{eqn_centereq} hold for all $y,\,z\in\F_q$.}

\begin{thm}
	If $G$ is the group in  Construction \ref{const_S1}, then its nilpotency class is $2$ or $3$, and the latter occurs if and only if $\deg(S_1)>1$.
\end{thm}
\begin{proof}
In this case, we have $o(g)=p^e=1$. Since $\overline{\fg_{a,b,c}}$ is independent of $a$,  we write $\bar{\fg}_{b,c}:=\overline{\fg_{a,b,c}}$ for short.	 Suppose that $\bar{\fg}_{b,c}\in Z(\bar{G})$. Since $\theta_{x,y,z}\equiv 1$, the last two equations in  \eqref{eqn_centereq} reduce to $cS_1(z)=zS_1(c)$ for all $z\in\F_q$. If $c\ne 0$, then $S_1(z)=c^{-1}S_1(c)z$, so $\deg(S_1)\le 1$. Therefore, $Z(\bar{G})=\{\bar{\fg}_{b,0}:\,b\in\F_q\}$ if $\deg(S_1)>1$; moreover,  $\bar{G}/Z(\bar{G})$ is abelian in this case, i.e., $Z_2(\bar{G})=\bar{G}$. If $\deg(S_1)\le 1$, then $Z(\bar{G})=\bar{G}$, since $cS_1(z)=zS_1(c)$ holds for all $c,\,z\in\F_q$ in this case. This completes the proof.
\end{proof}	

For the last three constructions, the nilpotency class of the group $G$ varies in a large range, cf. Table \ref{table_ncODD}. It is infeasible to give an explicit description of the nilpotency class in general, so instead we give a reasonably tight bound under the assumption $l>1$. In the sequel, we change our strategy and consider the lower central series of $\bar{G}=G/G_A$. We introduce a chain of subgroups of $\bar{G}$ as follows:
\begin{equation}\label{eqn_barGi}
\bar{G}_i:=\{\overline{\fg_{a,b,c}}:\,a,b \in\F_q, \,c\in R_i\},\quad 1\le i\leq p^e,
\end{equation}
where $R_i=(1-g)^{i}(\F_q)$. In particular, $\bar{G}_{p^e}=\bar{G}_B$, where $G_B=\{\fg_{0,b,0}:\,b\in\F_q\}$. They will help to distinguish the entries of the lower central series $\{\gamma_{i}(\bar{G})\}$ of $\bar{G}$.

Take $\overline{\fg_{a,b,c}}\in \bar{G}$.  For $x,\,y,\,z\in\F_q$, we set
$\overline{\fg_{u,v,w}}:= \overline{\fg_{a,b,c}}^{-1} \circ\overline{\fg_{x,y,z}}^{-1}\circ \overline{\fg_{a,b,c}}\circ \overline{\fg_{x,y,z}}$. Here we are only concerned with the coordinate $w$, which we compute as
\begin{equation}\label{eqn_coordi_w}
w=(\theta_{0,y,z}-1)(c)-(\theta_{0,b,c}-1)(z).
\end{equation}

\begin{thm}\label{thm_G1nil}
	Let $G$ be the group in either of Constructions \ref{const_S2}-\ref{const_S4}, let  $p^e$ be the order of $g$, and assume $l>1$. Then the nilpotency class of $G$ lies in the range $[2p^e,3p^e]$.
\end{thm}
\begin{proof}
By Lemma \ref{lem_ZpeGA}, it suffices to show that the nilpotency class of $\bar{G}=G/G_A$ lies in the range $[p^e,2p^e]$. Let $U=\{c\in\F_q:\,\theta_{a,b,c}=1\textup{ for some }a,b\in \F_q\}$ be as in Lemma \ref{lem_P1P2}, write $\bar{\fg}_{b,c}:= \overline{\fg_{a,b,c}}$ for short, and let $\bar{G}_i$ be as defined in  \eqref{eqn_barGi}. We have  $e\ge 1$ in each of the three constructions.

We claim that  $\gamma_{i}(\bar{G})$ is contained in $\bar{G}_{i-1}$  but not in $\bar{G}_{i}$ for $2\leq i\leq p^e$, and $\gamma_{p^e+1}(\bar{G})$ is contained in $\bar{G}_{p^e}=\bar{G}_B$. We prove this by induction.
	\begin{enumerate}
		\item[(1)] First consider the case $i=2$. By the property (P1) and the assumption $l>1$, there exists $c\in U\setminus R_1$ by comparing sizes. Since $\theta_{a,b,c}$ is independent of $a$ by  \eqref{eqn_theta6}, there exists $b\in\F_q$ such that $\theta_{a,b,c}=\theta_{0,b,c}=1$ for such a $c$. For such a pair $(b,\,c)$ and $(y,\,z)=(0,\,t_C)$, we have $w=c^g-c$, where $w$ is as defined in  \eqref{eqn_coordi_w}. The element $w$ lies in $R_1\setminus R_2$, i.e., $[\bar{\fg}_{b,c},\,\bar{\fg}_{0,t_C}]\in\bar{G}_1\setminus\bar{G}_2$. On the other hand, for any $(b,\,c)$ and $(y,\,z)$, the corresponding $w$ always lies in $R_1$. This proves the case $i=2$.
		\item[(2)] Suppose that the claim has been established for $2\le i\le p^e-1$. Take $\bar{\fg}_{b,c}\in\gamma_i(\bar{G})$. We have $\theta_{0,b,c}=1$, since  $\gamma_i(\bar{G})$ is contained in $\bar{G}'$.  The equation \eqref{eqn_coordi_w} reduces to $w=(\theta_{0,y,z}-1)(c)$, which always lies in $R_{i}$ by induction.  There exists $\bar{\fg}_{b,c}\in\gamma_i(\bar{G})$ with $c\in R_{i-1}\setminus R_i$ by induction. Take $(y,\,z)=(0,\,t_C)$, and correspondingly $w=(g-1)(c)$. This element $w$ lies in $R_i\setminus R_{i+1}$. This proves the case $i+1$.
	\end{enumerate}
The claim on $\gamma_{p^e+1}(\bar{G})$ is proved in the same way as in (2). This completes the proof of the claim. In particular, $\gamma_{p^e+1}(\bar{G})$ is contained in the group $\bar{G}_B$, while $\gamma_{p^e}(\bar{G})$ is not. This gives the lower bound on the nilpotency class of $\bar{G}$.
	
The subgroup $\gamma_{p^e+1}(\bar{G})$ is contained in $\bar{G}'$, and thus has a trivial Frobenius part. Since it is also in $\bar{G}_B$,  it is a subgroup of $H:=\{\bar{\fg}_{b,0}:\,\theta_{0,b,0}=1\}$ of $\bar{G}_B$. Let $K_B$ be the set of $b\in\F_q$ such that $\bar{\fg}_{b,0}\in H$. In Construction \ref{const_S2}, we have $K_B=\F_q$.  In Construction \ref{const_S3} and \ref{const_S4}, we have $K_B=\{b:\,\tr_{\F_q/\F_p}(\mu_Bb)=0\}$.  We define $E_0(H):=H$, and inductively $E_{i+1}(H):=[E_{i}(H),\bar{G}]$ for $i\ge 1$. It is routine to check that $\bar{\fg}_{b,0}^{-1} \circ\bar{\fg}_{y,z}^{-1}\circ \bar{\fg}_{b,0}\circ \bar{\fg}_{y,z}=\bar{\fg}_{-b+b^{\theta_{0,y,z}},0}$. It readily follows by induction that $E_i(H)=\{\bar{\fg}_{b,0}:\,b\in\ (1-g)^i(K_B)\}$ for $1\leq i\leq p^e$. In particular, $E_{p^e}(H)=\{1\}$. From $\gamma_{p^e+1}(\bar{G})\le E_{0}(H)$,  we deduce that $\gamma_{p^e+i}(\bar{G})\le E_{i-1}(H)$ for $i\ge 1$ by induction.  Since $E_{p^e}(H)=\{1\}$, we have $\gamma_{2p^e+1}(\bar{G})=1$. This gives the desired upper bound on the nilpotency class of $\bar{G}$.
\end{proof}

In Table \ref{table_ncODD}, we have listed some explicit values of nilpotency classes for some special cases of Constructions \ref{const_S2} and  \ref{const_S3}. We demonstrate the calculations by the following example, and omit the details for the other cases in the table.

\begin{example}\label{examp_nc}We give the upper central series of the group $G$ arising from Construction \ref{const_S2}  in some special cases that are listed in Table \ref{table_ncODD}. The terms $Z_i(G)$, $1\le i\le p$, have been determined in Lemma \ref{lem_ZpeGA}, so we do not list them below. Recall that $q=p^{pl}$, $g(x)=x^{p^l}$, and we set $\mu_C=1$, $K=\{x\in \F_q : \tr_{\F_q/\F_p}(x)=0\}$. Further assume that $l>1$.  As usual, set  $R_0:=\F_q$ and $R_i:=(1-g)^i(\F_q)$ for $1\le i\le p$.
	\begin{enumerate}
		\item[(1)] If $S_1(z)\equiv 0$, then
		\[
		Z_{p+i}(G)=\{\fg_{a,b,c}:\,a\in\F_q, b,\,c\in (1-g)^{p-i}(\F_q)\},\quad 0\leq i\leq p.
		\]
		%In this case, its derived subgroup and Frattini subgroup are both equal to $\{\fg_{a,b,c}: a\in \F_q,\, b\in R_1, c\in (1-g)(K)\}.$ This is one of the rare cases where we can explicitly write down these two subgroups.
		\item[(2)] If $S_1(z)=z^{p^k}$ with $l\nmid k$, then
		\begin{align*}
		Z_{p+i}(G)&=\{\fg_{a,b,0}:\,a\in\F_q,\, b\in R_{p-i}\},\quad 0\leq i\leq p,\\
		Z_{2p+i}(G)&=\{\fg_{a,b,c}:\,a,\,b\in\F_q, \,c\in R_{p-i}\},\quad 0\leq i\leq p.
		\end{align*}
		\item[(3)] If $S_1(z)=(1-g)^k(z)$ for $1\leq k\leq p-1$, then the expression of $Z_{p+i}(G)$, $1\le i\le p-k$, is the same as in case (2), and
		\begin{align*}
		Z_{2p-k+j}(G)&=\{\fg_{a,b,c}:\,a\in\F_q,\,b\in R_{k-j},\,  c\in R_{p-j}\},\quad 1\leq j\leq k,\\
		Z_{2p+j}(G)&=\{\fg_{a,b,c}:\,a,\,b\in\F_q, \,c\in R_{p-k-j}\},\quad 1\leq j\leq p-k.
		\end{align*}
	\end{enumerate}
\end{example}

 \section{Concluding remarks}

In this paper, we have determined all the point regular groups of the Payne derived quadrangle $\cQ^P$ of the classical symplectic quadrangle $\cQ=W(q)$ in the case $q$ is odd.  We  have considered the isomorphism issues amongst the different constructions by calculating their group invariants such as exponents and Thompson subgroups. We also have obtained tight upper and lower bounds on the nilpotency classes of the resulting groups. As a corollary, we see that the finite groups that act regularly on the points of a finite generalized quadrangle can have unbounded nilpotency class. Prior to our work, the only known such groups have nilpotency class at most $3$ except for computer data for some groups of small orders.

In Section 5, we have also determined all the point regular subgroups of $\cQ^P$ contained in $\textup{PGL}(4,q)$ in the case $q$ is even and $q\ge 5$.  There are parallel results to Constructions \ref{const_S2}-\ref{const_S4} in the even characteristic case, which we do not include here. Furthermore, computer data indicates that there are examples with $r_{A,B}=2$ when $q=2^4$. It remains a challenging problem to fully classify the even characteristic case.\\

\noindent\textbf{Acknowledgement.} This work was supported by National Natural Science Foundation of China under Grant No. 11771392. The authors thank the reviewers for their detailed comments and valuable suggestions that helped to improve the presentation of the paper greatly.  \\

\scriptsize
 \setlength{\bibsep}{0.5ex}  % vertical spacing between references
\bibliographystyle{plain}
\footnotesize
\linespread{0.5}
\bibliography{PRSGRef}

\end{document}